\documentclass[reqno]{amsart}

\usepackage{amssymb}

\usepackage{stmaryrd}

\usepackage{amsthm,amsmath}
\usepackage{amsfonts}
\usepackage{graphicx}
\usepackage{enumitem}
\usepackage{bm}
\usepackage[utf8]{inputenc}
\numberwithin{equation}{section}
\newtheorem{theorem}{Theorem}[section]
\newtheorem{lemma}[theorem]{Lemma}
\newtheorem{proposition}[theorem]{Proposition}
\newtheorem{corollary}[theorem]{Corollary}
\newtheorem{definition}[theorem]{Definition}
\theoremstyle{remark}
\newtheorem{remark}[theorem]{Remark}

\usepackage{caption}
\usepackage{tikz}

\renewcommand\arg{\mathop{\mathrm{arg}}\nolimits}

\newcommand\diam{\mathop{\mathrm{diam}}\nolimits}

\renewcommand{\Im}{\operatorname{Im}}
\renewcommand{\Re}{\operatorname{Re}}

\newcommand\sn[2]{\operatorname{sn}(#1\,|\,#2)}
\newcommand\cn[2]{\operatorname{cn}(#1\,|\,#2)}
\newcommand\dn[2]{\operatorname{dn}(#1\,|\,#2)}
\renewcommand\sc[2]{\operatorname{sc}(#1\,|\,#2)}
\newcommand\sd[2]{\operatorname{sd}(#1\,|\,#2)}
\newcommand\cd[2]{\operatorname{cd}(#1\,|\,#2)}

\newcommand\C{\mathbb{C}}
\newcommand\R{\mathbb{R}}
\newcommand\T{\mathbb{T}}
\newcommand\Z{\mathbb{Z}}

\newcommand\cA{\mathcal{A}}
\newcommand\cB{\mathcal{B}}
\newcommand\cD{\mathcal{D}}
\newcommand\cF{\mathcal{F}}
\newcommand\cL{\mathcal{L}}
\newcommand\cQ{\mathcal{Q}}
\newcommand\cR{\mathcal{R}}
\newcommand\cS{\mathcal{S}}
\newcommand\cX{\mathcal{X}}

\newcommand\cG{\mathcal{G}}
\newcommand\rG{\mathrm{G}}

\newcommand\E{\mathbb{E}}
\renewcommand\P{\mathbb{P}}

\newcommand\thetaa{\widehat{\theta}}
\renewcommand\thetag{\overline{\theta}}
\newcommand\thetae{\check{\theta}}
\newcommand\alphae{\check{\alpha}}

\newcommand\rr{1}
\newcommand\ri{i}

\newcommand\free{{\,\mathrm{f}}}
\newcommand\wired{\mathrm{w}}

\newcommand\BAP{\textsc{BAP}(\theta_0)}

\newcommand\scc[1]{{\scriptscriptstyle #1}}
\newcommand\ccc[1]{c_{\scc{#1}}}

\newcommand\dexp[3]{{\bm{\mathrm{e}}}(#1\,;#2,#3)}
\newcommand\dexpk[4]{{\bm{\mathrm{e}}}(#2\,|\,#1\,;#3,#4)}

\newcommand\tempskipped[1]{}
\newcommand\qtoem{q=\frac12m\delta}
\newcommand\magnet{(k^*)^\frac 14}
\newcommand\sccorr{(k^*)^\frac 12}

\begin{document}

 \global\long\def\re{\Re}
 \global\long\def\im{\Im}
\global\long\def\eps{\varepsilon}

\global\long\def\pa{\partial}
\global\long\def\ccor#1{\langle#1\rangle}
\global\long\def\sn#1#2{\mathrm{sn}\left(#1\,|\,#2\right)}
\global\long\def\sc#1#2{\mathrm{sc}\left(#1\,|\,#2\right)}
\global\long\def\sd#1#2{\mathrm{sd}\left(#1\,|\,#2\right)}
\global\long\def\nd#1#2{\mathrm{nd}\left(#1\,|\,#2\right)}
\global\long\def\cn#1#2{\mathrm{cn}\left(#1\,|\,#2\right)}
\global\long\def\cd#1#2{\mathrm{cd}\left(#1\,|\,#2\right)}
\global\long\def\dn#1#2{\mathrm{dn}\left(#1\,|\,#2\right)}
\global\long\def\nc#1#2{\mathrm{nc}\left(#1\,|\,#2\right)}
\global\long\def\dc#1#2{\mathrm{dc}\left(#1\,|\,#2\right)}
\global\long\def\ns#1#2{\mathrm{ns}\left(#1\,|\,#2\right)}

\global\long\def\corvz#1#2{\mathrm{G}_{[#1]}(#2)}
\global\long\def\coruz#1#2{\mathrm{G}_{[#1]}(#2)}
\global\long\def\corzz#1#2{\mathrm{G}_{(#1)}(#2)}
\global\long\def\fctr{\Gamma}
\global\long\def\Erruz{R_{\chi\sigma}}
\global\long\def\Errvz{R_{\chi\mu}}
\global\long\def\Errzz{R_{\chi\chi}}
\global\long\def\ea#1{#1}
\global\long\def\mup{\nu}
\global\long\def\ang#1{\alpha_{#1}}
\global\long\def\eang#1{\alpha_{#1}}
\global\long\def\lamb{\bar{\varsigma}}
\global\long\def\lambb{\varsigma}
\global\long\def\ellip#1{\check{#1}}
\global\long\def\ellipa#1{\check{\alpha}_{#1}}
\global\long\def\ang#1{\alpha_{#1}}
\global\long\def\MCon#1{\cX_{#1}^{\delta}}

\title[Spin correlations in the Ising model on isoradial graphs]{Universality of spin correlations in the Ising model on isoradial graphs}

\author[Dmitry Chelkak]{Dmitry Chelkak$^\mathrm{a,b}$}
\author[Konstantin Izyurov]{Konstantin Izyurov$^\mathrm{c}$}
\author[R\'emy Mahfouf]{R\'emy Mahfouf$^\mathrm{a}$}

\thanks{\textsc{${}^\mathrm{A}$ D\'epartement de math\'ematiques et applications, \'Ecole Normale Sup\'erieure, CNRS, PSL University, 45 rue d'Ulm, 75005 Paris, France.}}

\thanks{\textsc{${}^\mathrm{B}$ Holder of the ENS--MHI chair funded by MHI. On leave from St.~Petersburg Dept. of Steklov Mathematical Institute RAS, Fontanka 27, 191023 St.~Petersburg, Russia.}}

\thanks{\textsc{${}^\mathrm{C}$ Department of Mathematics and Statistics, University of Helsinki,
P.O. Box 68 (Pietari Kalmin katu 5), 00014 Finland.}}

\thanks{\emph{E-mail:} \texttt{dmitry.chelkak@ens.fr}, \texttt{konstantin.izyurov@helsinki.fi}, \texttt{remy.mahfouf@ens.fr}}

\begin{abstract}
We prove universality of spin correlations in the scaling limit of
the planar Ising model on isoradial graphs with uniformly bounded
angles and Z\textendash invariant weights. Specifically, we show that
in the massive scaling limit, i.\,e., as the mesh size $\delta$ tends
to zero at the same rate as the inverse temperature goes to the critical
one, the two-point spin correlations in the full plane behave as
\[
\delta^{-\frac{1}{4}}\E\left[\sigma_{u_{1}}\sigma_{u_{2}}\right]\ \to\ C_{\sigma}^{2}\cdot\Xi\left(|u_{1}-u_{2}|,m\right)\quad\text{as\ensuremath{\quad\delta\to0}},
\]
where the universal constant $C_{\sigma}$ and the function $\Xi(|u_{1}-u_{2}|,m)$
are independent of the lattice. The mass $m$ is defined by the relation
$k'-1\sim4m\delta$, where $k'$ is the Baxter elliptic parameter.
This includes $m$ of both signs as well as the critical case when
$\Xi(r,0)=r^{-1/4}${\normalsize{}$.$}

These results, together with techniques developed to obtain them,
are sufficient to extend to isoradial graphs the convergence of multi-point
spin correlations in finite planar domains on the square grid, which
was established in a joint work of the first two authors and C.\,Hongler at criticality, and by S.\,C.\,Park in the sub-critical
massive regime. We also give a simple proof of the fact that the infinite-volume
magnetization in the Z\textendash invariant model is independent of
the site and of the lattice.

As compared to techniques already existing in the literature, we streamline
the analysis of discrete (massive) holomorphic spinors near their
ramification points, relying only upon discrete analogues of the kernel
$z^{-1/2}$ for $m=0$ and of $z^{-1/2}e^{\pm2m|z|}$ for $m\ne0$.
Enabling the generalization to isoradial graphs and providing a solid ground for further generalizations, our approach also
considerably simplifies the proofs in the square lattice setup.
\end{abstract}



\maketitle


\section{Introduction}

\subsection{General context} Isoradial graphs, or, equivalently, rhombi tilings, were introduced
by Duffin \cite{duffin1968potential} as a natural family of embedded
planar graphs admitting a nice discretization of complex analysis
and potential theory. They latter attracted considerable attention
both in the physics and the mathematics communities in connection
with lattice models of two-dimensional statistical mechanics. Although
the latter live on abstract planar graphs with some additional structure,
e.\,g., weights on edges, 
it often turns out that embedding the graph isoradially sheds light on the behavior of the model at
or near criticality. By now, there is an extensive literature on statistical
mechanics on isoradial graphs; e.\,g., see \cite{boutillier2010Periodic_case,boutillier2011Locality,boutillier2012statistical, duminil2018universality,duminil2020rotational,cimasoni2012discrete,cimasoni2012critical,DubedatDimersCR,grimmett2013universality, kenyon2002laplacian,li2017conformal,mercat2001discrete}
and references therein. Apart from being a natural framework for establishing
universality, an additional motivation to study models on isoradial
graph is that they form a flexible family to approximate arbitrary
Riemann surfaces \cite{mercat2001discrete,cimasoni2012critical,cimasoni2012discrete},
for which regular lattices are too rigid.

A relevance of rhombi tilings for the two-dimensional statistical mechanics can be described
as follows. As uncovered by Baxter \cite{baxter1978solvable,baxter1986free,baxter2016exactly},
many models of statistical mechanics with local interactions have
\emph{Z-invariant weights.} The ``invariance'' here refers to the
fact that the $\mathrm{Y}$\textendash$\Delta$ transform of the graph on which the model is defined
leaves unchanged the partition function and many (if not all) observables of interest.
As Baxter has shown, the flexibility this entails allows one to solve
many such models exactly. Although these algebraic techniques does not directly rely upon embedding of graphs into the complex plane, it turns out that representing them as isoradial grids allows one to give a direct geometrical meaning to the parameters of the corresponding R-matrices.
This is more than a numerical
coincidence as the small mesh size limits of \emph{critical} 2D lattice models are expected to have more symmetries than just the symmetries
of the lattice: e.\,g., in this setup one expects the scaling, rotational and even \emph{conformal} invariance.
To exhibit such a symmetry, one needs to pick embeddings of graphs and the isoradial/rhombic lattices are known to be the correct choice provided that one can reformulate the lattice model under consideration using this geometric framework.
For rhombic lattices, the $\mathrm{Y}$\textendash$\Delta$ moves are nothing but the so-called \emph{cube flips}, which allow one to transform (big pieces of) different lattices into each other without affecting quantities of interest. Apart from being at the heart of physicists' predictions of the universality for Z-invariant lattice models, this idea also recently led to a rigorous proof of the rotational invariance of several Z-invariant lattice models; we refer the interested reader to~\cite{duminil2020rotational} for further details.

In this paper we consider the critical and near-critical (aka massive) \emph{Ising model} on isoradial graphs with Z-invariant weights. Before discussing our main results, let us first briefly mention the results available in the \emph{square grid} setup.
For the \emph{critical} model (see also~\cite{Chelkak2016ECM,chelkak2017ICM} for more details), the convergence and conformal invariance was proven by Smirnov~\cite{Smirnov_Ising} for basic fermionic observables; by Hongler and Smirnov~\cite{HonglerSmirnov,hongler_thesis} for the energy density correlations; and by the first two authors and Hongler~\cite{ChelkakHonglerIzyurov} for spin correlations. Gheissari, Hongler, Park, Viklund and Kyt\"ol\"a studied more general
local fields in~\cite{gheissari2019ising,HonglerViklundKytolaCFT} and a unified framework to treat mixed correlations of all primary fields (i.\,e., fermions, energy densities, spins and disorders) was recently developed in~\cite{CHI_Mixed}.
Another aspect of conformal invariance is the convergence of interfaces
and loop ensembles in the domain walls representation of the model to SLE$_{3}$
and CLE$_{3}$; see~\cite{ChelkakSmirnov_et_al} and~\cite{BenoistHongler}, respectively.
Beyond the critical case, the \emph{massive} limit of spin correlations in smooth domains and of fermionic observables in general (i\,.e., not necessarily smooth) ones was recently treated by S.\,C.\,Park in~\cite{park2018massive} and~\cite{Park2021Fermionic}, respectively.

The \emph{universality} -- within the isoradial family -- of basic fermionic observables was shown in~\cite{ChelkakSmirnov2} for the critical model and, very recently, in~\cite{Park2021Fermionic} for the massive one. These results admit a direct extension 
to the energy density correlations as the energy density field -- in sharp contrast to spins themselves -- can be directly expressed via fermions. However, the analysis of spin correlations is considerably more subtle.
Of the aforementioned work, the proofs given in~\cite{ChelkakHonglerIzyurov,park2018massive} relied
especially heavily on the properties of the square lattice, and hence
they did not admit a simple generalization to the isoradial case.
In this paper, we provide the missing ingredients and prove the universality
for spin correlations on isoradial graphs, both in the critical and
in the massive setup. This also paves a way to a proof of the universality
of correlations of all primary fields \cite{CHI_Mixed} on isoradial graphs and possibly
beyond, although we do not discuss it here.

\subsection{Main results} We work with the Z-invariant Ising model defined on (subsets of) an infinite planar
isoradial grid~$\Gamma^{\circ,\delta}$; note that the dual grid $\Gamma^{\bullet,\delta}$
is also isoradial with the same radii, and $\Lambda^{\delta}:=\Gamma^{\circ,\delta}\cup \Gamma^{\bullet,\delta}$
forms a rhombi tiling of mesh size~$\delta$. Throughout
the paper, we assume that all rhombi tilings satisfy the following bounded angle property for some (fixed)~$\theta_0>0$:
\[
\BAP:\ \text{no rhombus has an angle smaller than $2\theta_{0}$.}
\]
The Z-invariant weights { (per edge in the low-temperature -- aka the \emph{domain walls} -- expansion of the
model on~$\Gamma^{\circ,\delta}$ or, equivalently, in the high-temperature expansion of the dual model on~$\Gamma^{\bullet,\delta}$)}
are given by
\begin{equation}
\label{eq: Z-inv}
{ x_e=\exp[-2\beta^\circ J_e^\circ]=\tanh[\beta^\bullet J_e^\bullet]\ =\ \tan\tfrac{1}{2}\thetaa_{e},}
\qquad \sin\thetaa_{e}\ =\ \sn{\tfrac{2K}{\pi}\thetag_{e}}k,
\end{equation}
where $\thetag_{e}$ is half of the angle of the rhombus containing~$e$ adjacent to a vertex of $\Gamma^{\bullet,\delta}$ (see Fig.~\ref{fig:notation}),
and \mbox{$K=K(k)$} is the complete elliptic integral of the first kind; see
\cite[Eq.~19.2.8]{NIST:DLMF} (we routinely refer to the Digital Library of Mathematical Functions
\cite{NIST:DLMF} in the above format).  Thus, the edge weights
are determined by the isoradial embedding of the graph plus a single
temperature-like parameter, the \emph{elliptic modulus} $k\in (0,1)\cup i\R$,
or equivalently the \emph{nome} $q\in\R$; see \cite[Eq.~22.2.1]{NIST:DLMF}. We refer the reader to a recent work of Boutillier, de~Tili\`ere and Raschel~\cite{boutillier2017z,boutillier2019z,deTiliere-Zinv} for an extensive discussion of thus defined Ising model and its links with dimers and uniform spanning trees. { Note that~\eqref{eq: Z-inv} agrees with the parametrization used in~\cite{boutillier2019z} for the interaction parameters~$\beta^\bullet J_e^\bullet$ while we will typically work with the dual model defined on~$\Gamma^{\circ,\delta}$ (i.\,e., we usually assign spins to \emph{faces} of~$\Gamma^\delta=\Gamma^{\bullet,\delta}$).}

The critical point is $k=q=0$, in which case $\thetaa_{e}=\thetag_{e}$. The case~$q>0$ and $k\in (0,1)$ gives rise to a sub-critical model on~$\Gamma^{\bullet,\delta}$ and a super-critical model on~$\Gamma^{\circ,\delta}$ while $q<0$ and $k\in i\R$ correspond to the opposite situation. (Thus, the nome~$q$ has the same monotonicity as the temperature in the model on~$\Gamma^{\circ,\delta}$.) { Also, note that the parametrization~\eqref{eq: Z-inv} is symmetric with respect to exchanging the roles of~$\Gamma^{\bullet,\delta}$ and~$\Gamma^{\circ,\delta}$ in the following sense: one also has
\[
\tanh[\beta^\circ J^\circ_e]=\exp[-2\beta^\bullet J^\bullet_e]\ =\ \tan\tfrac{1}{2}\thetaa{}_e^*,\qquad \sin\thetaa{}_e^*\ =\ \sn{\tfrac{2K}{\pi}\thetag{}_{e}^*}{k^*},
\]
where $\thetaa_e+\thetaa{}_e^*=\frac{\pi}{2}=\thetag_e+\thetag{}_e^*$ and the dual elliptic parameter~$k^*$ is given by
\begin{equation}
\label{eq: KW-duality-kq}
k^*=\frac{ik}{\sqrt{1-k^2}}\quad \Leftrightarrow\quad  q^*=-q\,.
\end{equation}
This transform is nothing but a simple way of writing the \emph{Kramers--Wannier duality} (e.\,g., see~\cite[Section~7.5]{duminil-parafermion} and~\cite[Section~4.5]{boutillier2019z}) in the elliptic context.}
We study the \emph{massive scaling limit}, as~$\delta\to0$ and simultaneously $q\to0,$ with the relation
\begin{equation}
\label{eq:qtoem}
\textstyle \qtoem
\end{equation}
between the two, where the parameter $m\in\R$ is called \emph{mass},
alluding to a massive field theory conjecturally describing this limit. As a particular case, we study the scaling limit at criticality, $m=0$.
Our results are uniform with respect to lattices satisfying the uniformly bounded angles assumption $\BAP$
for a fixed $\theta_{0}>0$, hence we do not assume that~$\Lambda^\delta$ at
each scale are related to each other in any way.

Throughout our paper, we use the notation~$\E^{(m)}$ to denote the expectation in the Z-invariant Ising model defined on the isoradial grid of mesh size~$\delta$ with the elliptic parameter obtained from~$\delta$ via~\eqref{eq:qtoem}. { We write $\E^{(m),\wired}$ and $\E^{(m),\free}$ for expectations considered in finite domains in order to specify the \emph{wired} and \emph{free} boundary conditions, respectively. Also, in Proposition~\ref{prop:intro-sub} we use the notation~$\E^{+}$ for the infinite volume limit of the sub-critical model with \emph{plus} boundary conditions.}
\begin{theorem}
\label{thm:intro-mass} For each~$m\in\R$ there exists a function~$\Xi(\cdot,m):\R_+\to\R_+$ such that
the following holds in the massive scaling limit~$\qtoem\to 0$:
\begin{equation}
\delta^{-\frac{1}{4}}\E_{\Gamma^{\circ,\delta}}^{(m)}[\sigma_{u_{1}}\sigma_{u_{2}}]\ \to\ \mathcal{C}_{\sigma}^{2}\cdot\Xi(|u_{2}-u_{1}|,m)\ \ \text{as}\ \ \delta\to0\,,\label{eq: thm_1_intro}
\end{equation}
where the constant~$\mathcal{C}_{\sigma}=2^{\frac{1}{6}}e^{\frac{3}{2}\zeta'(-1)}$ is independent of both the rhombic lattice~$\Lambda^\delta$ and the mass~$m$. One has~$\Xi(r,0)\equiv r^{-\frac{1}{4}}$ and~$\Xi(r,m)\sim r^{-\frac{1}{4}}$ as~$r\to0$ for all~$m$.
\end{theorem}
\begin{proof}
See Section~\ref{sub:univ-crit} for the case~$m=0$ and Section~\ref{sub:conv-mass} for~$m\ne 0$.
\end{proof}
\begin{remark}\label{rem:PainleveIII}
The explicit expression for $\Xi(r,m)$ in terms of Painlev\'e~III transcendents is given by a celebrated formula
of Wu, McCoy, Tracy and Barouch~\cite{wu1976spin,SMJ_1_4,kadanoff1980smj} for the massive model on the \emph{square grid}; see also~\cite[Corollary~1.2 and Section~4.2]{park2018massive}. The main content of Theorem~\ref{thm:intro-mass} is that this result holds universally within the class of isoradial
graphs; note in particular that the constant~$C_{\sigma}$ does not depend on the local geometry of~$\Lambda^\delta$ near $u_{1,2}$. A similar result was hinted
by Dub\'edat~\cite[Proposition 27]{DubedatDimersCR} for magnetic correlators
of the Gaussian free field; presumably, \emph{at criticality} the convergence~\eqref{eq: thm_1_intro}
can be alternatively derived therefrom via the combinatorial bosonization correspondence~\cite{dubedat2011exact}.
\end{remark}

The left-hand side of (\ref{eq: thm_1_intro}) is a correlation in the infinite-volume thermodynamic limit, i.\,e., the limit
of correlations in increasing finite domains $\Omega_{1}^{\delta}\subset\Omega_{2}^{\delta}\subset\ldots\subset\Gamma^{\delta}$
for a fixed temperature parameter $q$. The existence of such a limit can be shown by standard monotonicity arguments and RSW bounds
at criticality; see Section~\ref{sub:RSW} for details. Theorem \ref{thm:intro-mass}
then concerns another limit as one lets both~$q,\delta\to0$ so that $\qtoem$.
In fact, to prove Theorem~\ref{thm:intro-mass} we rely upon the fact that the RSW bounds are uniform with respect to~$\delta$ and~$q$, which allows us to work in (sufficiently large) finite domains $\Omega^\delta$ instead of~$\Gamma^\delta$. In particular, along the way we prove the convergence and universality for spin-spin correlations in smooth simply connected domains with appropriate boundary conditions: wired if~$m\le 0$ and free if~$m\ge 0$. Our analysis also implies the following:

As in \cite{ChelkakHonglerIzyurov} and \cite{park2018massive}, one of the key ingredients of our proof in the case~$m\le 0$ is a (uniform) convergence result for the ``discrete logarithmic derivatives''
\begin{equation}
\label{eq: intro_log_der}
\log\frac{\E_{\Omega^{\delta}}^{(m),\wired}[\sigma_{u'_{1}}\sigma_{u_{2}}]}{\E_{\Omega^{\delta}}^{(m),\wired}[\sigma_{u_{1}}\sigma_{u_{2}}]}\ =\ \re\big[(u'_{1}-u_{1})\mathcal{A}_{\Omega}^{(m)}(u_{1},u_{2})\big]+o(\delta)
\end{equation}
where $u'_{1}\sim u_{1}$ are nearest neighbors on the lattice~$\Gamma^{\circ,\delta}$ and the quantity~$\mathcal{A}_{\Omega}^{(m)}(u_{1},u_{2})$ is expressed via the scaling limit of spinor fermionic observables; see Section~\ref{sub:intro-techniques} for a more detailed discussion. This
result extends without any effort to multi-spin correlation, leading
to convergence results for ratios of such correlations:
\begin{equation}
\label{eq:intro-multipoint}
\frac{\E_{\Omega^{\delta}}^{(m),\wired}[\sigma_{u'_{1}}\sigma_{u'_{2}}\ldots\sigma_{u'_{n}}]}
{\E_{\Omega^{\delta}}^{(m),\wired}[\sigma_{u_{1}}\sigma_{u_{2}}\ldots\sigma_{u_{n}}]}\ \to\ \frac{\langle\sigma_{u'_1}\sigma_{u'_2}\ldots\sigma_{u'_n}\rangle_\Omega^{(m),\wired}} {\langle\sigma_{u_1}\sigma_{u_2}\ldots\sigma_{u_n}\rangle_\Omega^{(m),\wired}}\ \ \text{as}\ \ \delta\to 0\,,
\end{equation}
where the ``continuum correlation functions'' in the right-hand side are defined as
\[
\textstyle \langle\sigma_{u_1}\sigma_{u_2}\ldots\sigma_{u_n}\rangle_\Omega^{(m),\wired}\ :=\ \exp\int\Re\big[\sum_{s=1}^n\cA_\Omega^{(m)}(u_s,u_1,\ldots,u_{s-1},u_{s+1},\ldots,u_n)du_s\big],
\]
with an appropriate multiplicative normalization. Once the convergence~\eqref{eq:intro-multipoint} is established, the asymptotics (\ref{eq: thm_1_intro}) together with usual RSW-type arguments are enough to fix the { explicit} multiplicative normalization of the correlation functions. Thus, our analysis also implies the following:

\begin{theorem}\label{thm:intro-multi}
The results of~\cite[Theorem~1.2]{ChelkakHonglerIzyurov} for~$m=0$ and~\cite[Theorem~1.1]{park2018massive} for~$m<0$ asserting convergence
\[
\delta^{-\frac n8}\E_{\Omega^{\delta}}^{(m),\wired}[\sigma_{u_{1}}\sigma_{u_{2}}\ldots\sigma_{u_{n}}]\ \to\ \mathcal{C}_\sigma^n\cdot \langle\sigma_{u_1}\sigma_{u_2}\ldots\sigma_{u_n}\rangle_\Omega^{(m),\wired}\ \ \text{as}\ \ \delta\to 0
\]
of multi-point spin correlations in discrete approximations~$\Omega^\delta$ of bounded simply connected domains~$\Omega$ (with smooth boundaries if~$m<0$), hold true, without any change, for the massive Ising model on isoradial grids with Z-invariant weights.
\end{theorem}

We stress once again here that not only 
the continuous correlation functions
are universal within this class of the lattices/weights, but also the constant $\mathcal{C}_{\sigma}$
in front.

\begin{proof} The proof of \eqref{eq:intro-multipoint} in the case $n=2$ is given in Corollary \ref{cor:conv-to-A-crit} { for $m=0$} and in Corollary \ref{cor:conv-to-A-mass} { for $m<0$.} Since the analysis is local, it extends without any change to the case $n>2$. Given \eqref{eq:intro-multipoint}, the rest of the proofs in~\cite{ChelkakHonglerIzyurov} and~\cite{park2018massive} amount to fixing the overall normalization by sending { the points $u_s$} pairwise to each other or to the boundary { of~$\Omega$,} and the only part of that argument that relied on the specific properties of the square lattice was the computation of the full-plane two-point correlation. This ingredient is now supplied by Theorem \ref{thm:intro-mass}.
\end{proof}

\begin{remark} In order to prove Theorem~\ref{thm:intro-mass} for~$m>0$ we rely upon an analogue of~\cite[Theorem~1.7]{ChelkakHonglerIzyurov} which gives the convergence of the ratio of the two-point correlations in dual Ising models. Thus, similarly to a work of S.\,C.\,Park~\cite{park2018massive}, our methods do~\emph{not} directly imply an analogue of Theorem~\ref{thm:intro-multi} for~$m>0$. However, we believe that the techniques developed in this paper for the analysis of spinor fermionic observables near their branching points allow to prove such a convergence, at least in smooth simply connected domains, following the framework of~\cite{CHI_Mixed}.
\end{remark}

A sharp control of the discrete logarithmic derivative (\ref{eq: intro_log_der}),
in principle, suffices to recover the scaling function $\Xi(r,m)$
in the statement of Theorem \ref{thm:intro-mass}, but not the fact
that the constant $C_{\sigma}$ is lattice-independent. In order to
complete the proof, we use an additional \emph{gluing argument} (see Section~\ref{sub:star-ext} for details), showing
that a finite piece of an arbitrary isoradial grid can be glued, staying within the same family, to a piece of a regular (rectangular) lattice, so that the sizes of these
pieces and the distance between them are comparable. Then, we can
move the spins from the irregular part to the regular one, controlling
how the correlation changes in the process, and establishing universality.

A similar argument can be applied to analyze the \emph{magnetization in the infinite-volume limit} of the sub-critical model on a fixed isoradial grid~{ $\Gamma^\circ$, say, with~$\delta=1$.} As a by-product of our analysis of the massive model, we also get the following result. We are not aware
of its detailed proof in full generality
in the literature, although it was probably known, at least { for some particular lattices,} in the folklore.
\begin{proposition}[{\bf  Baxter's formula}]
\label{prop:intro-sub} For~$q<0$, the infinite-volume magnetization in the sub-critical Ising model on~$\Gamma^\circ$ with Z-invariant weights is universal:
\[
\E_{\Omega_{n}}^{+}[\sigma_{u}]\to\magnet\ \ \text{as}\ \ \Omega_n\uparrow\Gamma^{\circ}\quad  \text{and}\quad \E_{\Gamma^\circ}[\sigma_{u_{1}}\sigma_{u_{2}}]\to \sccorr\ \ \text{as}\ \ |u_{1}-u_{2}|\to\infty.
\]
{ For~$q>0$, similar results hold for the model on~$\Gamma^{\bullet}$ with $k^*$ replaced by~$k$.}
\end{proposition}
\begin{proof} See Section~\ref{sub:spin-sub}. As in the case of Theorem~\ref{thm:intro-mass} and the constant $\mathcal{C}_{\sigma}$,
we do not compute these limits explicitly, but rather show
that they are universal. 
To this end, we glue a large enough piece of a given isoradial grid to a piece
of a rectangular lattice, on which we can apply the celebrated Onsager--Yang
result~\cite{yang1952spontaneous,mccoy2014two} in the form given
by Baxter \cite[Eq.~7.10.50]{baxter2016exactly}; see also~\cite[Section~3]{chelkak_Hongler_Mahfouf} for a simplified derivation.
{ (Note that the elliptic parameter~$k$ in \cite{boutillier2019z} corresponds to~$k'$ of Baxter; see footnote in \cite[Section 2.2.2]{boutillier2019z}.)}
\end{proof}

\subsection{Techniques and related projects}\label{sub:intro-techniques} The general strategy of our proof of the key convergence result~\eqref{eq: intro_log_der} follows that of \cite{ChelkakHonglerIzyurov,park2018massive}.
One introduces an observable, a properly normalized spin-fermion-disorder correlator, which, as a function of the position of the fermion, is
a \emph{massive s-holomorphic spinor} living on a double cover of the original discrete domain ramified at the positions of the spin and the
disorder. (The { notion} of massive s-holomorphic functions on isoradial graphs { and the regularity theory thereof were} independently developed in a recent work of S\,.C\,.Park~\cite{Park2021Fermionic}; see Section~\ref{sub:mass-s-hol} for more details.) We then prove the convergence of this observable to a massive holomorphic (i.\,e., satisfying the { Dirac} equation~$\overline{\partial}f+im\overline{f}=0$) limit as~$\delta\to 0$; uniformly away from the boundary and from the branching points.

After this is done, one uses the fact that both spin-spin correlations in (\ref{eq: intro_log_der})
can be recovered from the values of the observable by placing
the fermion next to the disorder. We thus need a way to express a
value of a massive s-holomorphic spinor next to its ramification point
in terms of its values at a definite distance therefrom. This is where
our main \emph{technical innovation} comes into play: we introduce a very
simple version of the Cauchy integral formula for spinors that allows
one to do such a reconstruction using an explicit 
full-plane kernel, the discrete analogue of a spinor $e^{\mp i\frac\pi 4}\cdot z^{-\frac{1}{2}}e^{\pm 2m|z|}$.
Such kernels were essentially constructed by Dub\'edat in~\cite{DubedatDimersCR} for the critical case~$m=0$; we extend his construction
to the massive setup using the theory of massive discrete exponentials, recently developed by Boutillier, de Tili\`ere and Raschel
in~\cite{boutillier2017z,boutillier2019z}. The construction and the asymptotic analysis of the required branching kernels are presented in Section~\ref{sec:asymptotics}.

Recall that in~\cite{ChelkakHonglerIzyurov,park2018massive} a considerably more complicated reconstruction procedure was employed. In particular, it also required an explicit construction of discrete analogs of the kernels $e^{\mp i\frac\pi 4}z^{\frac{1}{2}}$ (or their massive modifications) and an argument based on the symmetrization procedure and the discrete Beurling inequality; this technique is unavailable in the isoradial setup due to the lack of symmetries of the lattice. Thus, not only our new argument enables generalization of the results to the isoradial setup, it also leads to a much simpler proof in the square lattice case. In particular, the companion paper~\cite{CHI_Mixed} devoted to a unified treatment of mixed correlations of primary fields, borrowed the arguments of the present paper in what concerns the analysis near the branching points. Moreover, this Cauchy formula can be written in a purely abstract form without any assumption on the embedding or weights of the model under consideration (see Lemma~\ref{lem:XG-monodromy}), which paves the way for further generalizations of the convergence results for spin correlations \emph{beyond} the Z-invariant setup; cf. recent results on the convergence of fermionic observables on the so-called s-embeddings of planar weighted graphs~\cite{chelkak2020ising}.

Another -- though not strictly necessary for our analysis -- new idea implemented in this paper is a {re-embedding} of the massive Z-invariant Ising model { on~$\Lambda^\delta$} into the complex plane using the aforementioned \emph{s-embeddings}~$\cS^\delta$; see Section~\ref{sub:s-emb} for more details. This allows us to benefit from a general regularity theory developed for s-holomorphic functions on s-embeddings in~\cite[Section~2]{chelkak2020ising} and ~\cite[Section~6]{chelkak2020dimer}. Under this procedure, the original massive s-holomorphic observables on~$\Lambda^\delta$ and new s-holomorphic observables on~$\cS^\delta$ are linked by a simple explicit formula given in Proposition~\ref{prop:Fm=FS}, which immediately allows us to deduce the a priori regularity of massive s-holomorphic functions on~$\Lambda^\delta$ from the results of~\cite{chelkak2020ising,chelkak2020dimer}. Also, this provides a concrete illustration of a general phenomenology, which says that the mass in a planar Ising model manifests itself as the mean curvature of an s-embedding of the model into the Minkowski space~$\mathbb{R}^{2+1}$; see~\cite[Section~2.7]{chelkak2020ising} for a discussion. Though, as already mentioned above, these re-embedding techniques are not necessary for the analysis of the Z-invariant model (e.\,g., see~\cite{Park2021Fermionic}, where the relevant a priori regularity estimates developed directly on~$\Lambda^\delta$), we believe that they are flexible enough to be applied { to} less rigid setups.

{ Recall that throughout this paper we assume that all isoradial grids satisfy the uniform bounded angles property~$\BAP$, which plays an essential role in several places of our analysis as we frequently use the fact that the graph distances on rhombic lattices are comparable with Euclidean ones: e.\,g., when passing from~\eqref{eq: intro_log_der} to~\eqref{eq:intro-multipoint}. However, there are certain indications that this assumption is not strictly relevant, at least at criticality. Notably, similar techniques can be applied to a 2D graphical expansion (e.,g., see~\cite{Ioffe-quantumIsing,JHLi} and references therein) of the \emph{quantum 1D Ising model}, leading to similar convergence results for correlation functions;~we refer the interested reader to a forthcoming paper~\cite{JHLi-Mahfouf} for more details. This expansion can be thought of as a 2D Ising model on a rectangular grid with an infinitesimally small aspect ratio, a limit that obviously cannot be achieved under~$\BAP$.}

Finally, in terms of a framework developed in the companion paper~\cite{CHI_Mixed} for the analysis of mixed correlations of primary fields in (possibly) multiply connected domains, 
this paper provides the following ``building blocks'':
\begin{itemize}
\item an explicit construction and asymptotic analysis of required infinite-volume kernels~$\cG_{[u]}$, $\cG_{[v]}$ and~$\cG_{(a)}$ (see Section~\ref{sec:asymptotics} for more details);
\item analysis of the two-point spin correlation in the full plane (Theorem~\ref{thm:intro-mass}).
\end{itemize}
Note that a work~\cite{Park2021Fermionic} of S.\,C.\,Park contains two more such ``building blocks'', namely
\begin{itemize}
\item a ``quantitative convergence'' of basic fermionic FK-Ising observables;
\item uniform RSW-type estimates for the massive FK-Ising model; cf.~\cite{DuGaPe-near-crit}.
\end{itemize}
Thus, the only remaining input required to extend the results of~\cite{CHI_Mixed} to the massive Z-invariant model on isoradial grids is the analysis of solutions to certain Riemann-type boundary value problems for massive holomorphic functions \emph{in continuum} (which is considerably more complicated for~$m\ne 0$ than at criticality, cf.~\cite{Park2021Fermionic}).

\subsection*{Organization of the paper} We start Section \ref{sec: prelim} by fixing the notation and recalling the definition of the Z-invariant Ising model. 
We then recall the construction of fermionic observables via Kadanoff--Ceva order-disorder formalism, the propagation equation, and Smirnov's ``integrating the square'' procedure. { We also} recall a construction of the infinite volume limit of the model and describe the ``star extension'' procedure, which allows to glue a big piece of a given rhombic lattice to a regular one. In Section~\ref{sec: Discr_complex_analysis}, we review the ``massive discrete complex analysis'' techniques, namely, the a priori regularity of ``massive s-holomorphic functions'' constructed from solutions to the spinor propagation equation, and the fact that subsequential limits of such functions satisfy the Dirac equation. We also state the properties of the massive discrete holomorphic full-plane ``kernels'' used in our proofs, and introduce the aforementioned Cauchy integral formula. In Section~\ref{sec: Proofs}, we prove the main results of the paper. For convenience of the reader { we start with the critical case and then refer to it when discussing the massive setup.} In Section \ref{sec:asymptotics}, we construct the required full-plane kernels and perform the asymptotic analysis of thereof, using the ``massive discrete exponentials'' of Boutillier, de Tili\`ere, and Raschel~\cite{boutillier2017z}.

\subsection*{Acknowledgements} We are grateful to S.\,C.\,Park for many fruitful discussions of the massive Ising model. We also would like to thank C\'edric Boutillier, B\'eatrice de Tili\`ere, Jhih-Huang Li, Ioan Manolescu, Paul Melotti and Yijun Wan for helpful comments and remarks. D.\,C. is a holder of the ENS--MHI chair funded by the MHI, whose support is gratefully acknowledged.  The work of K.\,I. was supported by Academy of Finland via academy project ``Critical phenomena in dimension two: analytic and probabilistic methods''. The work of D.\,C. and R.\,M. was also partially supported by the ANR-18-CE40-0033 project DIMERS.

\section{Preliminaries and basic facts}
\label{sec: prelim}
\subsection{Notation: graphs, double covers, spin-disorder correlations}\label{sub:notation} We rely upon the spin-disorder formalism of Kadanoff and Ceva~\cite{KadanoffCeva}; e.\,g. see~\cite[Section~2]{chelkak2017ICM} or~\cite[Section~2]{CHI_Mixed} for more details and~\cite{ChelkakCimasoniKassel} for links of this approach with other combinatorial formalisms used to study the planar Ising model. For a planar graph~$G=G^\bullet$, let
\begin{itemize}
\item $G^\circ$ be the { graph dual} to~$G^\bullet$ 
(note that we typically work with the planar Ising model defined on~$G^\circ$ and not on~$G^\bullet$);
\smallskip
\item $\Lambda(G)$ be a planar graph whose set of vertices is the union of $G^\bullet$ and~$G^\circ$, with edges connecting adjacent vertices~$v\in G^\bullet$ and faces~$u\in G^\circ$ of~$G$;
\smallskip
\item $\diamondsuit(G)$ be the { graph dual} to~$\Lambda(G)$, we often call its vertices~$z\in\Lambda(G)$ \emph{quads} referring to this duality (note that all faces of~$\Lambda(G)$ have degree four);
\smallskip
\item $\Upsilon(G)$ be the medial graph of~$\Lambda(G)$ (i.\,e., vertices of~$\Upsilon(G)$ are in a bijective correspondence with edges~$(uv)$ of~$\Lambda(G)$), we often call its vertices~$c\in\Upsilon(G)$ \emph{corners} referring to the fact that they are in a bijective correspondence with corners of faces of~$G^\circ$ (or of~$G=G^\bullet$, note that faces of~$\Upsilon(G)$ correspond either to~$v\in G^\bullet$ or to~$u\in G^\circ$ or to~$z\in\diamondsuit(G)$);
\smallskip
\item $\Upsilon^\times(G)$ be a double cover of~$\Upsilon(G)$ that branches around \emph{each} of its faces
{ (e.\,g., see~\cite[Fig.~27]{mercat2001discrete} or~\cite[Fig.~6]{ChelkakSmirnov2} or~\cite[Fig.~3]{chelkak2020ising});}
\smallskip
\item more generally { (e.\,g., see Fig.~\ref{fig:double-covers} for an example with~$\varpi=[v,w]$)}, given a subset~$\varpi$ of vertices of~$\Lambda(G)=G^\bullet\cup G^\circ$ let
\begin{itemize}
\item $\Upsilon_\varpi(G)$ be a double cover of~$\Upsilon(G)$ branching \emph{only} around~$w\in \varpi$,
\item $\Upsilon^\times_\varpi(G)$ be a double cover ramified at all faces of~$\Upsilon(G)$ \emph{except}~$\varpi$.
\end{itemize}
\end{itemize}
In our paper the graph~$G$ is usually a discrete domain on an isoradial grid \mbox{$\Gamma^\delta=\Gamma^{\bullet,\delta}$} of mesh size~$\delta\to 0$, which approximates a planar bounded simply connected domain \mbox{$\Omega\subset\C$}. We use the notation~$\Omega^{\delta}\subset\Gamma^\delta$ for such approximations. We will often view vertices (dual vertices, etc.) as complex numbers giving their position in~$\C$.

\begin{figure}
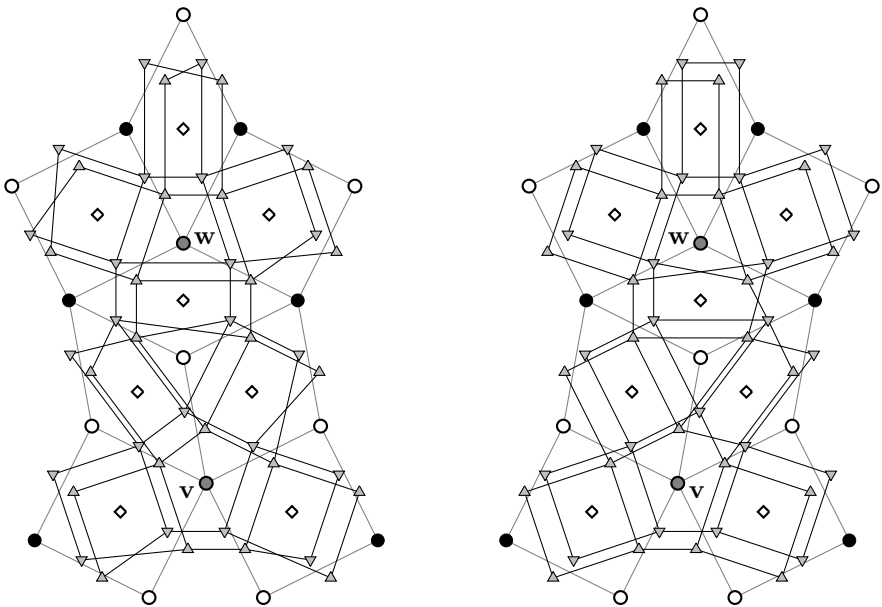

\begin{center}\begin{tikzpicture}[scale=0.17]
\input{BranchingTimesVW.txt}
\end{tikzpicture}
\hskip 40pt
\begin{tikzpicture}[scale=0.17]
\input{BranchingVW.txt}
\end{tikzpicture}\end{center}
\caption{{ \textsc{Left:} local structure of a double cover~$\Upsilon^\times_{[v,w]}$ of the medial graph~$\Upsilon$; vertices~$v\in \Gamma^\bullet$ and~$w\in\Gamma^\circ$ are marked in gray. \textsc{Right:} local structure of the corresponding double cover~$\Upsilon_{[v,w]}$.}}
\label{fig:double-covers}
\end{figure}

We often speak about \emph{spinors} defined on double-covers of graphs, which are functions on the double cover in question whose values at two lifts of the same vertex differ by the sign. Provided that an embedding of all these graphs into the complex plane is fixed, an important example of a spinor on~$\Upsilon^\times(G)$ is given by
\begin{equation}
\label{eq:eta-c-def}
\eta_c\ :=\ \varsigma\cdot\exp\big[-\tfrac{i}{2}\arg(v(c)-u(c))\big],\qquad \varsigma:=e^{i\frac{\pi}{4}},
\end{equation}
where $u(c)\in G^\circ$ and~$v(c)\in G^\bullet$ are endpoints of the edge~$(u(c)v(c))$ corresponding to the corner~$c$. A particular choice of the prefactor~$\varsigma$ is unimportant but influences the notation in what follows; we choose the value $e^{i\frac\pi 4}$ in order to keep the presentation consistent with~\cite{ChelkakSmirnov2} and~\cite{Park2021Fermionic} (note however that~\cite{ChelkakHonglerIzyurov} and~\cite{park2018massive} use another convention~$\varsigma=i$).

The (ferromagnetic) \emph{nearest-neighbor} Lenz--Ising model on the \emph{dual} graph~$G^\circ$ is a random assignment of spins~$\sigma_u\in\{\pm 1\}$ to the~\emph{faces} $u$ of~$G$ such that the probability of a spin configuration~$\sigma=(\sigma_u)$ is proportional to
\begin{equation}\label{eq:P[sigma]-def}
\textstyle \mathbb{P}_G[\,\sigma\,]\propto\exp\,[\,{ \beta^\circ} \sum_{u\sim w} { J^\circ_e}\sigma_u\sigma_w\,]\,,\qquad e=(uw)^*,
\end{equation}
where a positive parameter~${ \beta^\circ}=1/kT$ is called the \emph{inverse temperature}, the sum is taken over all pairs of adjacent faces~$u,w$ of~$G$ (equivalently, edges~$e$ of~$G$ or quads~$z\in\Lambda(G)$), and~${ J^\circ=(J^\circ_e)}$ is a collection of positive interaction constants indexed by edges of~$G$. Below we use the following \emph{parametrization} of $ \beta^\circ J^\circ_e$:
\begin{equation}\label{eq:parametrization-model}
\textstyle x_e\ =\ \tan\frac{1}{2}\thetaa_e\ :=\ \exp[\,-2{\beta^\circ J^\circ_e}\,].
\end{equation}
Note that the quantities~$x_e\in (0,1)$ and~$\thetaa_e:=2\arctan x_e\in (0,\frac{1}{2}\pi)$ have the same monotonicity as the temperature~$ (\beta^\circ)^{-1}$. We also often write~$\thetaa_z$ instead of~$\thetaa_e$ if a quad~$z\in\Lambda(G)$ corresponds to an edge~$e$ of~$G$. On isoradial graphs, we depart from arbitrary parameters $\thetaa_{z}$ and restrict to \emph{Z-invariant weights} given in terms of a global elliptic parameter~$k$ and geometric angles~$\thetag_e$ of the embedding by (\ref{eq: Z-inv}).

Note that~$\tan\thetaa_{z}=(1+4q)\cdot\tan\thetag_{z}+O(q^{2})$ as $q\to 0$,
which implies that
\begin{equation}
\thetaa_{z}-\thetag_{z}\ =\ 4q\cdot\sin\thetag_{z}\cos\thetag_{z}+O(q^{2})\ =\ 2\delta m\sin\thetag_{z}\cos\thetag_{z}+O(\delta^{2})
\label{eq:thetaa-thetag}
\end{equation}
in the massive limit $\qtoem\to 0$. In particular, $\thetaa_{z}=\thetag_{z}$ at criticality.

We denote by~$\E^{(m),\wired}_{\Omega^\delta}$ the expectation under the measure~\eqref{eq:P[sigma]-def} with Z-invariant weights given by (\ref{eq: Z-inv}), with $\qtoem$. We will omit $\delta$ when we consider a fixed lattice of mesh $\delta=1$. The superscript~$\wired$ stands for \emph{wired} boundary conditions: since we use the convention that spins~$\sigma_u$ are assigned to faces of~$\Omega^\delta$, instead of considering a single outer face~$u_\mathrm{out}$ of~$\Omega^\delta$ one can think about all boundary points~$u\in\Gamma^{\circ,\delta}$ as being wired to each other. Sometimes, we also fix the spin of $u_\mathrm{out}$ to be~$+1$ (i.\,e., impose `$+$' boundary conditions) and write~$\E^+_{\Omega^\delta}$ instead of~$\E^\wired_{\Omega^\delta}$.

The \emph{Kramers--Wannier duality} (e.\,g., see~\cite[Section~7.5]{duminil-parafermion}) provides a link between the Ising model (on faces of~$\Omega^\delta$) described above and another nearest-neighbor Ising model defined on vertices of~$\Omega^\delta$ with interaction parameters~{ $\beta^\bullet J^\bullet_e$ such that~$\tanh[\beta^\bullet J^\bullet_e]=\exp[-2\beta^\circ J^\circ_e]$ or, equivalently, $\sinh[2\beta^\circ J^\circ_e]\sinh[2\beta^\bullet J^\bullet_e]=1$.} In terms of the parametrization~\eqref{eq:parametrization-model} this duality reads as
    \[
    \exp[-2\beta^\circ J^\circ_e]=\tan\tfrac12\thetaa_z,\qquad \exp[-2\beta^\bullet J^\bullet_e]=\tan\tfrac12\thetaa^*_z,\qquad \thetaa_z+\thetaa_z^*=\tfrac{1}{2}\pi.
    \]
The angles $\thetaa^*_z$ also { admit parametrization}~\eqref{eq: Z-inv} with a { dual} elliptic parameter~$k^*$; the relation is simplest in terms of the nome { $q^*=-q$; see~\eqref{eq: KW-duality-kq}.} We denote the expectation in this dual model (defined on vertices of \mbox{$\Omega^\delta$}) as~$\E^{(-m),\free}_{\Omega^{*,\delta}}$, where the superscript~$\free$ stands for \emph{free} boundary conditions and emphasizes that no restrictions on the boundary spins are imposed.

We use the Kadanoff--Ceva \emph{disorder variables}: for a subset $\gamma$ of edges of $G^\bullet$, put
\[
\textstyle \mu_{\gamma}\ :=\ \prod_{u\sim w:(uw)\cap \gamma\neq \emptyset} e^{-2 \beta^\circ J^\circ_{(uw)^*}\sigma_u\sigma_w}.
\]
Up to the sign, the correlations of these variables with spins only depend on $\pa \gamma$ viewed as a chain modulo $2$. Hence, we will simply write
\[
 \mu_\gamma=\mu_{v_1}\dots\mu_{v_m},\ \ \text{where}\ \ \pa\gamma=\{v_1,\dots,v_m\}
\]
inside such correlations. We have the following identity:
\begin{equation}
\label{eq:KW-duality}
\E_{\Omega^\delta}^{\wired}[\,\mu_{v_1}\ldots\mu_{v_m}\sigma_{u_1}\ldots\sigma_{u_n}\,]\ =\ \E_{\Omega^{*,\delta}}^\free[\,\mu_{u_1}\ldots\mu_{u_n}\sigma_{v_1}\ldots\sigma_{v_m}\,],
\end{equation}
where $u_1,\ldots,u_n\in\Omega^{\circ,\delta}$,~$v_1,\ldots,v_m\in\Omega^{\bullet,\delta}$, and both~$n,m$ are assumed to be even (otherwise, these correlators do not make sense); { in fact, both sides of~\eqref{eq:KW-duality} lead to the same sums over subgraphs of~$\Omega^\delta=\Omega^{\bullet,\delta}$ if one uses the low-temperature expansion for the left-hand side and the high-temperature expansion for the right-hand one.} In particular, the disorder variables~$\mu$ are dual objects to spins~$\sigma$ under the Kramers--Wannier duality. We refer the reader to~\cite{KadanoffCeva,ChelkakCimasoniKassel,chelkak2017ICM,CHI_Mixed} for more details on the Kadanoff--Ceva formalism.

\begin{remark}
\label{rem:mu-sigma-spinor}
Recall also that under a natural rule for tracking the signs \cite[Section 2.2]{CHI_Mixed}, both sides of~\eqref{eq:KW-duality} change the sign when one of the vertices~$v_p$ makes a turn around one of~$u_q$; in other words one should view~\eqref{eq:KW-duality} as a spinor defined on an appropriate double cover of the set~$(\Omega^{\bullet,\delta})^{\times m}\times(\Omega^{\circ,\delta})^{\times n}$ with the removed diagonal ($v_p=v_{p'}$ or $u_q=u_{q'}$ for some~$p\ne p'$ or~$q\ne q'$). In full generality, the extension of~\eqref{eq:KW-duality} 
onto the diagonal requires certain technicalities in fixing the signs. However, no 
problems of that kind arise in the simplest setup $m=2$ that we only need below.
\end{remark}

\subsection{Fermionic observables and functions~$\bm{H_X}$} The Kadanoff--Ceva fermionic variable~$\chi_c$ is formally defined as
$\chi_c\ :=\ \mu_{v(c)}\sigma_{u(c)}$ and the corresponding fermionic observables read as
\[
X_\varpi(c)\ :=\ { \E_{G}^{\wired}}\big[\chi_c\,\mu_{v_1}\ldots\mu_{v_{m-1}}\sigma_{u_1}\ldots\sigma_{u_{n-1}}\big]\,,\quad c\in\Upsilon^\times_{\varpi}(G),
\]
where~$\varpi=\{v_1,\ldots,v_{m-1},u_1,\ldots,u_{n-1}\}$ and we assume that both~$n,m$ are even.
\begin{remark}
Due to Remark~\ref{rem:mu-sigma-spinor}, $X_\varpi$ is a \emph{spinor} on the double cover $\Upsilon^\times_\varpi(G)$. This implies that the product~$\eta_c X_\varpi(c)$, where~$\eta_c$ is given by~\eqref{eq:eta-c-def}, is a spinor on the double cover~$\Upsilon_\varpi(G)$, which branches \emph{only} over points from~$\varpi$; { see Fig.~\ref{fig:double-covers}.}
\end{remark}

\begin{figure}
\begin{center}\begin{tikzpicture}[scale=0.24]
\input{SpinUnivNotationG.txt}
\end{tikzpicture}
\hskip 24pt
\begin{tikzpicture}[scale=0.24]
\input{SpinUnivNotationA.txt}
\end{tikzpicture}\end{center}
\caption{{ \textsc{Left:} Local notation at a quad \mbox{$z=(v_0u_0v_1u_1)\in\diamondsuit$}. Identity~\eqref{eq:3-terms} expresses the value~$X(c_{00})$ of a Kadanoff--Ceva fermionic observable as a linear combination of~$X(c_{01})$ and~$X(c_{10})$ with coefficients~$\cos\thetaa_z=\cn{\frac{2K}{\pi}\thetag_z}k$ and~$\sin\thetaa_z=\sn{\frac{2K}{\pi}\thetag_z}k$.
\textsc{Right:} Following~\cite{park2018massive,Park2021Fermionic}, to define the value $F(z)\in\C$ of the corresponding  massive \mbox{s-holomorphic} function (aka Smirnov's fermionic observable) at~$z\in\diamondsuit$, we imagine a tilted rhombus with edges~$e^{\pm i(\thetaa_z-\thetag_z)}(v_p-u_q)$ centered at~$z$ and then apply the usual definition~\cite{ChelkakSmirnov2} of s-holomorphic functions to this `virtual' rhombus.}}
\label{fig:notation}
\end{figure}

It is well known (e.\,g., see~\cite{mercat2001discrete} or~\cite[Section~3.5]{ChelkakCimasoniKassel} and references therein) that Kadanoff--Ceva fermionic observables satisfy the so-called \emph{propagation equation}
\begin{equation}
\label{eq:3-terms}
X(\ccc{00})\ =\ X(\ccc{01})\cos\thetaa_z+X(\ccc{10})\sin\thetaa_z\,,
\end{equation}
which holds for each corner $\ccc{00}$ and its neighbors~$\ccc{01},\ccc{10}$ on the corresponding double cover~$\Upsilon^\times_\varpi(G)$ such that all three are incident to the same quad $z$; see Fig.~\ref{fig:notation} for the notation.
A consequence of this identity is that, given a fermionic observable~$X$, one can define (up to an additive constant) a function~$H_X$ on the graph~$\Lambda(G)$ by prescribing its increments between neighboring vertices~$u_q\in G^\circ$ and~$v_p\in G^\bullet$ as
\begin{equation}
\label{eq:HX-def}
H_X(v_p)-H_X(u_q)\ :=\ (X(c_{pq}))^2.
\end{equation}
(Note that for the critical Ising model on~$\mathbb{Z}^2$, this is nothing but Smirnov's definition from~\cite[Lemma~3.6]{Smirnov_Ising}; see also Section~\ref{sub:mass-s-hol} below.) More generally, if two observables~$X_1,X_2$ are (locally) defined on the same double cover~$\Upsilon^\times_\varpi(G)$, then one can extend the above definition and introduce a function
$H[X_1,X_2]$, also defined up to an additive constant, by setting
\begin{equation}
\label{eq:HX1X2-def}
H[X_1,X_2](v_p)-H[X_1,X_2](u_q)\ :=\ X_1(c_{pq})X_2(c_{pq})
\end{equation}
(this quantity does not depend on the lift of~$c_{pq}$ onto~$\Upsilon^\times_\varpi(G)$; cf. Lemma~\ref{lem:XG-monodromy}.)

\begin{remark}
\label{rem:H-Dirichlet-bc} From the combinatorial perspective, considering a discrete simply connected domain~$\Omega^\delta\subset\Gamma^\delta$ with wired boundary conditions boils down to identifying all `white' vertices along the boundary with each other. In particular, all the values of~$H_X$ (or of $H[X_1,X_2]$) at `white' boundary vertices~$u\in \partial\Omega\cap G^\circ$ are the same, i.\,e., the function~$H_X|_{G^\circ}$ always has \emph{Dirichlet boundary conditions} in this case. However, note that this is not the case for boundary values of $H_X$ on `black' boundary vertices. Nevertheless, there exists a trick (originally suggested in~\cite[Section~3.6]{ChelkakSmirnov2} in the critical setup) which allows to modify the grid~$\Lambda^\delta$ near the boundary of~$\Omega^\delta$ staying in the isoradial family and artificially define~$H_X$ on new obtained `black' vertices so that it also has Dirichlet boundary values and all required `discrete complex analysis estimates' hold true; see Remark~\ref{rem:bdry-trick} below.
\end{remark}

It is also convenient to extend the definition of~$H_X$ to the set~$\diamondsuit(G)$ as follows:
\begin{equation}
\label{eq:HX(z)-def}
\begin{array}{rcl}
H_X(v_p)-H_X(z)&:=& X(c_{p{\scriptscriptstyle 0}})X(c_{p{\scriptscriptstyle 1}})\cos\thetaa_z,\\
H_X(z)-H_X(u_q)&:=& X(c_{{\scriptscriptstyle 0}q})X(c_{{\scriptscriptstyle 1}q})\sin\thetaa_z,
\end{array}
\end{equation}
where $c_{p\scriptscriptstyle 0}$ and~$c_{p\scriptscriptstyle 1}$ are assumed to be chosen as neighbors on~$\Upsilon^\times_\varpi(G)$ to avoid the ambiguity in the sign; note that this definition is consistent with~\eqref{eq:HX-def} due to the propagation equation~\eqref{eq:3-terms}. Let~$\Lambda(G)\cup\diamondsuit(G)$ denotes a planar graph whose edges consist of all edges of~$\Lambda(G)$ and edges linking each vertex~$z\in\diamondsuit(G)$ to its four neighbors~$z\sim w\in \Lambda(G)$.

\smallskip

The following lemma holds for all planar graphs and all interaction parameters.

\begin{lemma}\label{lem:maxH-principle} Let a function~$H_X$ be (locally) constructed via~\eqref{eq:HX(z)-def} and~\eqref{eq:HX-def} from a fermionic observable~$X$ defined on a double cover~$\Upsilon^\times_\varpi(G)$. Then, $H_X$ satisfies
\begin{itemize}
\item the maximum principle on~$\Lambda(G)\cup\diamondsuit(G)$ except at points of~$\varpi\cap G^\bullet$;
\item the minimum principle on~$\Lambda(G)\cup\diamondsuit(G)$ except at points of~$\varpi\cap G^\circ$.
\end{itemize}
\end{lemma}
\begin{proof} See~\cite[Proposition~2.10]{chelkak2020ising}, which proves both the maximum and the minimum principles for~$H_X$ in absence of (non)branching vertices, i.\,e., if~$\varpi=\emptyset$. If~$\varpi\ne\emptyset$, note that the value at a vertex~$u\in\varpi\cap G^\circ$ is always smaller than the values at neighboring vertices from $G^\bullet$ due to~\eqref{eq:HX-def}, thus~$H_X$ cannot attain a maximum at~$u$. Similarly,~$H_X$ cannot attain a minimum at a point~$v\in \varpi\cap G^\bullet$.
\end{proof}
\begin{remark}\label{rem:maxH-principle} Let~$w\in\varpi$ be an isolated point of~$\varpi$ and assume that~$X(c)=0$ at one of the nearby corners~$w\sim c\in\Upsilon^\times_\varpi(G)$. Then, it is easy to see that the function $H_X$ satisfies \emph{both} the maximum and the minimum principle near~$w$ as its values at two neighboring vertices (one from~$G^\circ$, the other from~$G^\bullet$) of~$c$ are the same.
\end{remark}

\subsection{Infinite-volume of the Ising model with~$\bm{q\ne 0}$ and RSW estimates} \label{sub:RSW}
Given an isoradial grid~$\Lambda^\delta$ of mesh size~$\delta$  and a point~$u\in\C$ we denote by~$\Lambda^\delta_R(u)$ a discretization of the square box
\[
[\Re u-R\,,\Re u+R]\times[\Im u-R\,,\Im u+R]
\]
on~$\Lambda^\delta$ with appropriate boundary conditions which can vary depending on the context. We write~$\Lambda^\delta_R$ instead of~$\Lambda^\delta_R(0)$ if~$u=0$ and skip the superscript if~$\delta=1$.

It is well-known that spin-spin correlations $\E_G[\sigma_u\sigma_w]$ in the Ising model on a finite graph can be written as the probability that~$u$ and~$w$ are connected in the so-called Fortuin--Kasteleyn (or random cluster) representation of the model; e.\,g., see~\cite[Section~7]{duminil-parafermion} for more details.

For the \emph{critical} (i.\,e., $m=0$) Ising model on isoradial grids it is well known that the following Russo--Seymour--Welsh-type estimate holds {uniformly with respect to boundary conditions}:
\[
\P\big[\,\text{there exists a wired circuit in the annulus~$\Lambda_{3R}(u)\smallsetminus \Lambda_R(u)$}\,\big]\ \ge\ p_0\ >\ 0.
\]
(For a proof, one can, e.\,g., use~\cite[Theorem~C]{ChelkakSmirnov2} to show that crossings of rectangles with self-dual (i.\,e., wired/free/wired/free) boundary conditions have probability uniformly bounded away from zero, and then apply the proof of~\cite[Proposition~2.10]{DuGaPe-near-crit}; see also~\cite[Section~5.6]{chelkak2020ising}.) In particular, the existence of such circuits implies the uniqueness of the infinite-volume Gibbs measure in the critical FK-Ising model and allows one to speak about infinite-volume correlations~$\E_{\Lambda}[\sigma_u\sigma_w]$.

Assume now that~$\qtoem\le 0$. By monotonicity with respect to interaction parameters, the (uniform with respect to boundary conditions) existence of wired circuits in annuli~$\Lambda_{3R}(u)\smallsetminus\Lambda_R(u)$ also holds in this case. This allows one to define the \emph{infinite-volume} limit of the sub-critical FK-Ising model on~$\Lambda$ and, moreover, implies the following uniform estimate:
\begin{quote}
for each~$\varepsilon>0$ there exists $A=A(\varepsilon)\gg 1$ such that for all~$m\le 0$ and all~$D\ge 1$  one has
\begin{equation}
\label{eq:RSW-sigma-sigma}
(1-\varepsilon)\cdot \E^{(m),\wired}_{\Lambda_{AD}(u)}[\sigma_u\sigma_w]\ \le\ \E^{(m)}_{\Lambda}[\sigma_u\sigma_w]\ \le\ \E^{(m),\wired}_{\Lambda_{AD}(u)}[\sigma_u\sigma_w]
\end{equation}
provided that~$|w-u|\le D$.
\end{quote}
Indeed, the latter inequality in~\eqref{eq:RSW-sigma-sigma} is trivial due to the monotonicity with respect to boundary conditions. The former follows from the fact that
\[
\E^{(m),\free}_{\Lambda_{AD}(u)}[\sigma_u\sigma_w]\ \ge\ \P^{(m),\free}_{\Lambda_{AD}(u)\smallsetminus \Lambda_D(u)}\big[\,\text{there exists a wired circuit}\,\big] \times \E^{(m),\wired}_{\Lambda_{AD}(u)}[\sigma_u\sigma_w],
\]
which holds due to the FKG inequality and the monotonicity of the probability that~$u$ and~$w$ are connected with respect to a domain. (Note that the first factor can be made arbitrary close to~$1$ by choosing~$A\ge 3^{\lfloor\log\varepsilon/\log(1-p_0)\rfloor+1}$.)
\begin{remark}
\label{rem:RSW-free-circuits}
One can similarly use dual-wired circuits in order to define the infinite-volume limit of the off-critical FK-Ising model on~$\Lambda$ and the infinite-volume spin-spin correlations for~$m\ge 0$. It is also worth noting that in the \emph{massive} regime $\qtoem$, $\delta\to 0$, one also has uniform RSW-type estimates for \emph{both} primary and dual crossings/circuits on scales $R\asymp 1$; see~\cite{DuGaPe-near-crit} and~\cite{Park2021Fermionic} for more details. Of course, this is a much deeper property of the massive Ising model as compared to the simple monotonicity with respect to~$m$ discussed above.
\end{remark}

\subsection{`Star extension' of a finite box on an isoradial grid} \label{sub:star-ext} In order to prove the universality of spin-spin correlations with respect to a grid, i.\,e., the fact the correlations on two isoradial grids~$\Gamma^{\circ,\delta}_1$ and $\Gamma^{\circ,\delta}_2$ behave in a similar way, we typically consider a `mixed' rhombic lattice { that} contains large pieces of both~$\Lambda^\delta_1$ and~$\Lambda^\delta_2$ and analyse the Ising model on { this lattice}. This strategy relies upon a possibility to `glue together' boxes of size~$R$ cut from different isoradial grids, with an additional requirement that these two pieces are located at~$O(R)$ distance from each other.

\begin{figure}
\begin{centering}
\includegraphics[width=\textwidth]{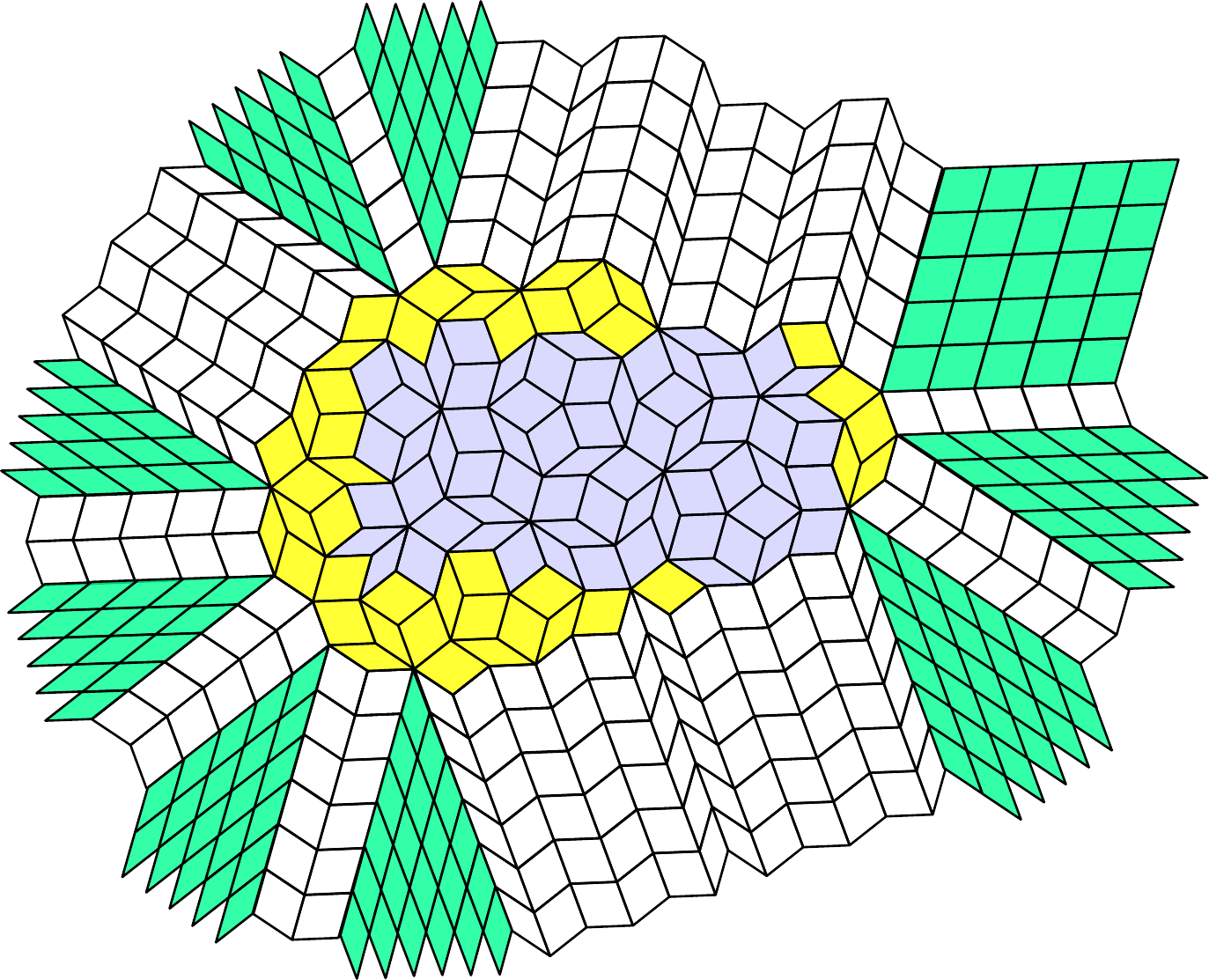}
\par\end{centering}
\bigskip{}

\begin{centering}
\includegraphics[width=0.9\textwidth]{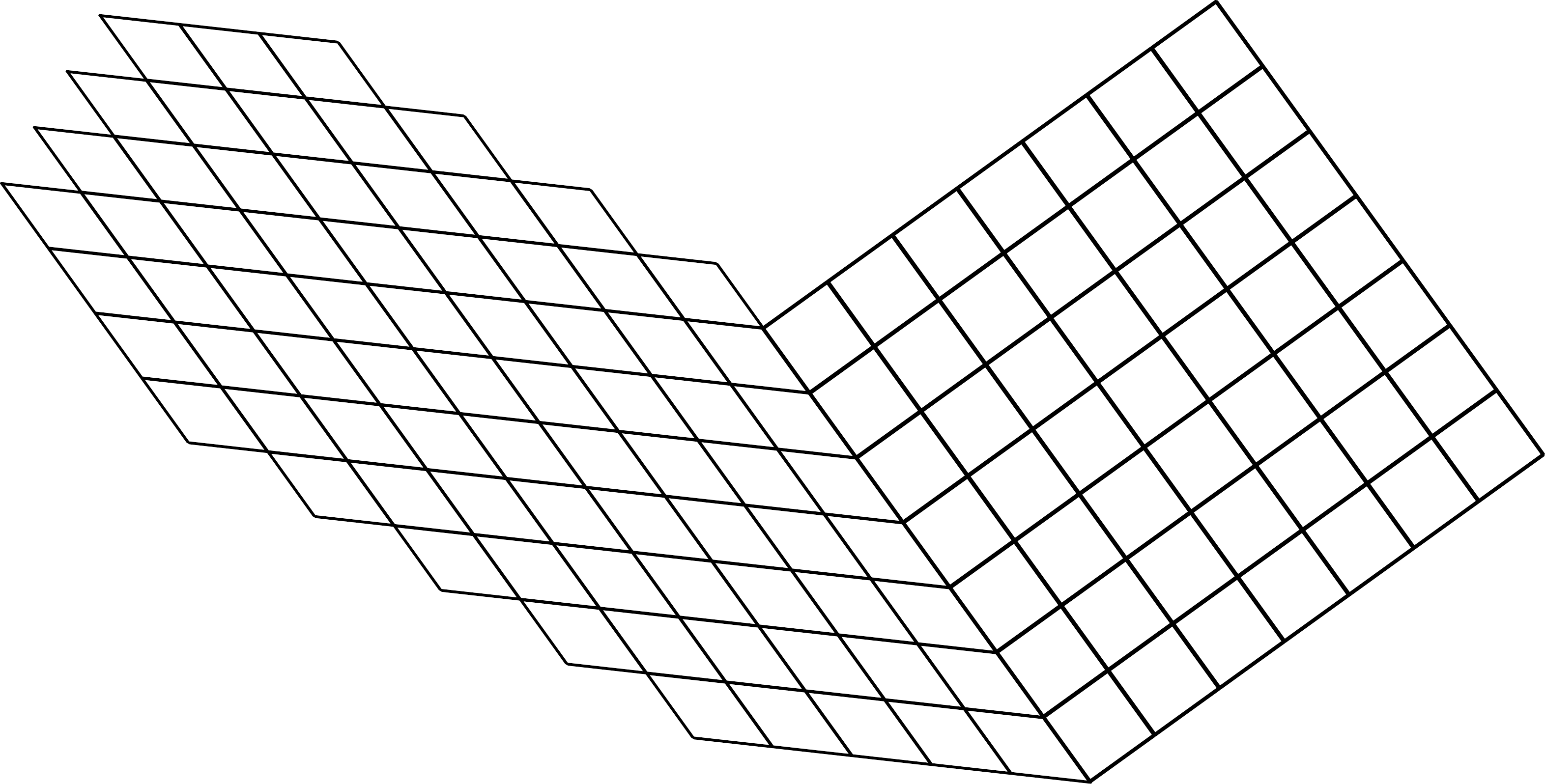}
\par\end{centering}
\caption{\textsc{Top:} the `star extension' procedure {  for a rectangular box on a Penrose rhombic tiling. The box is shown in blue and the rest of the corresponding set~$\Theta$ } is in yellow. The boundary of $\Theta$ consists of nine train-tracks, extended infinitely outwards, the wedge-shaped parts of regular lattices are in green. The opening angles of the wedges are bounded from below by $2\theta_0$ of the $\BAP$ property. { For Penrose rhombic tilings, $2\theta_0=\frac\pi{5}$; thus, in this case} the boundary of $\Theta$ could not consist of more than ten train tracks. \textsc{Bottom:} a gluing of a { regular rhombic lattice corresponding to a rectangular grid~$\Lambda^{\mathrm{rect}}$ with the square one.}}
\label{fig:star-extension}
\end{figure}

Clearly, no problem arises if~$\Lambda^\delta_1$ comes from a rectangular lattice while \mbox{$\Lambda^\delta_2=\delta\mathbb{Z}^2$} (or, more generally, if both~$\Lambda^\delta_{1,2}$ are obtained from rectangular lattices); see Fig.~\ref{fig:star-extension}. Consider now an irregular
rhombic lattice~$\Lambda$ satisfying the bounded angles property $\BAP$. In this section we show that one can construct a new rhombic lattice~$[\Lambda_R(u)]^\star$ satisfying the property~$\BAP$ such that
\begin{itemize}
\item $\Lambda$ and~$[\Lambda_R(u)]^\star$ have the same box $\Lambda_R(u)$ of size $R$ centered at $u$;
\smallskip
\item the modified lattice~$[\Lambda_R(u)]^\star$ contains infinite wedge-shaped subsets of rectangular grids located at distance at most $O(R)$ from the point~$u$; here and below the implicit constant in the estimate $O(R)$ depends on~$\theta_0$ only.
\end{itemize}
It is worth noting that these properties are scale invariant and thus the same procedure can be applied to rhombic lattices~$\Lambda^\delta$ (with mesh size~$\delta<R$) instead of~$\Lambda$.

Recall that a \emph{train-track} on~$\Lambda^\delta$ is a (infinite) sequence of adjacent rhombi such that they share a common direction of edges, called the transversal direction of the train-track { (e.\,g., see~\cite{KenyonSchlenker} for a discussion of this notion).} The construction of~$[\Lambda_R(u)]^\star$ goes in two steps, see also Fig.~\ref{fig:star-extension}:
\begin{itemize}
\item {\bf Step 1.} Let $\Theta_R(u)\supset \Lambda_R(u)$ be the connected component of the set of quads~\mbox{$z\in\diamondsuit$} such that \emph{both} train-tracks passing through~$z$ intersect $\Lambda_R(u)$.
\end{itemize}

\begin{lemma}
The set~$\Theta_R(u)$ has diameter~$O(R)$ and is train-track-convex, i.\,e., if two rhombii lying on the same train-track belong to the set~$\Theta_R(u)$, then the whole segment of the train-track between these rhombi belong to~$\Theta_R(u)$.
\end{lemma}
\begin{proof} The train-track-convexity of the set~$\Theta_R(u)$ readily follows from its definition. Indeed, denote by $\mathcal{T}$ the set of all train-tracks that intersect the box~$\Lambda_R(u)$. Let~$z_1,z_2\in \Theta_R(u)$ lie on the same train track~$t\in\mathcal T$ and let $z\in t$ be located between~$z_1$ and $z_2$. Denote by $t'$ the second train track passing through~$z$; note that it must intersect~$\Theta_R(u)$ since this set is connected. By definition of~$\Theta_R(u)$ this implies that $t'\in\mathcal T$ and hence $z\in \Theta_R(u)$.

To estimate the diameter of~$\Theta_R(u)$, note that for each~$t\in\mathcal T$ the number of quads in $t\cap \Theta_R(u)$ cannot exceed the total number~$|\mathcal T|$ of train-tracks in~$\mathcal T$ (since each of the remaining train-tracks $t'\in\mathcal T$ can intersect $t$ at most once). Thus, the estimate~$\diam \Theta_R(u)=O(R)$ follows from 
$|\mathcal T|\leq 2|\pa\Lambda_R(u)|=O(R)$.
\end{proof}

The train-track convexity of $\Theta_R(u)$ implies that it is simply connected and its boundary $\pa \Theta_R(u)$ is a simple closed broken line such that no quad adjacent to it from outside can have two sides on $\pa \Theta_R(u)$. Let us divide these adjacent quads into (maximal) arcs $\tau_k$, each of which is a part of a single-train track $t_k$. The above observation means that as we follow the boundary counterclockwise, the transversal directions of these train tracks also rotates counterclockwise, and the bounded angles property~$\BAP$ implies that it rotates at least $2\theta_0$ between two consecutive arcs; in particular, there are at most $\pi/\theta_0$ arcs. The construction of~$[\Lambda_R(u)]^\star$ concludes as follows:
\begin{itemize}
\item {\bf Step~2.} Duplicate (infinitely many times) each of the boundary arcs~$\tau_k\subset t_k$ of $\Theta_R(u)$ in the transversal direction of 
    $t_k$ and fill the remaining wedge-shaped regions (of angle at least $2\theta_0$) by regular rhombic tilings; see Fig.~\ref{fig:star-extension}.
\end{itemize}
We call $[\Lambda_R(u)]^\star$ a \emph{star extension} of $\Lambda_R(u)$. (In fact, we slightly abuse the notation since~$[\Lambda_R(u)]^\star$ is not defined only by~$\Lambda_R(u)$ and also depends on the structure of~$\Lambda$ near this box. However, this does not create any confusion in what follows.)

\section{Discrete complex analysis techniques}
\label{sec: Discr_complex_analysis}

\subsection{Massive s-holomorphic functions on isoradial grids} \label{sub:mass-s-hol}
In this section we discuss the notion of massive s-holomorphic functions on isoradial grids; see also a recent paper~\cite{Park2021Fermionic} where this notion is discussed in more detail. Similarly to the critical case (see~\cite[Section~3.2]{ChelkakSmirnov2}), given a real-valued Kadanoff--Ceva fermionic observable (i.\,e., a real-valued spinor on a subset of~$\Upsilon^\times(\Lambda^\delta)$ satisfying the propagation equation~\eqref{eq:3-terms}) one can construct a complex-valued Smirnov fermionic~$F^{(m),\delta}$ observable (defined on a subset of~$\diamondsuit^\delta$) in such a way that the propagation equation~\eqref{eq:3-terms} for~$X^{(m),\delta}$ is replaced by the identities between projections of~$F^{(m),\delta}$ onto certain directions; see Definition~\ref{def:shol-def-mass} below.

The 'abstract' definition~\eqref{eq:HX-def} of functions~$H^{(m),\delta}_X$ can be then -- again, similarly to the critical case -- thought of as considering the primitive of a discrete differential form~$\frac{1}{2}\Im[(F^{(m),\delta}(z))^2dz]$, which turns out to be closed on both~$\Gamma^{\bullet,\delta}$ and~$\Gamma^{\circ,\delta}$ up to certain local multiples that disappear in the limit~$\delta\to 0$; see Lemma~\ref{lem:H=intImF2} and Remark~\ref{rem:H=intImF2} below. In Proposition~\ref{prop:H-sub} (see also~\ref{rem:bdry-trick}) we  prove (at least, for small enough~$\delta$) that the function~$H^{(m),\delta}$ is subharmonic on~$\Gamma^{\bullet,\delta}$ provided that~$m\le 0$ and superharmonic on~$\Gamma^{\circ,\delta}$ provided that~$m\ge 0$, thus generalizing to the massive setup a (rather mysterious) observation made by Smirnov in~\cite{Smirnov_Ising} for the critical Ising model on the square grid.

Further, we briefly discuss the a priori regularity properties of massive s-holo\-morphic functions; see Proposition~\ref{prop:regularity}, this statement was obtained in~\cite{Park2021Fermionic}. We also suggest an alternative approach to this regularity theory, which relies upon a general framework developed in~\cite{chelkak2020ising} and uses an \emph{s-embedding} of the massive Ising model on~$\Lambda^\delta$ into the complex plane; see Section~\ref{sub:s-emb} for more details. Finally, this reduction to the framework of~\cite{chelkak2020ising} also directly implies that massive s-holomorphic functions on~$\Lambda^\delta$ satisfy the discrete maximum principle (up to a multiplicative constant); see Lemma~\ref{lem:max-principle} below.

\smallskip

The following definition is adopted from~\cite{Park2021Fermionic}.
\begin{definition} \label{def:shol-def-mass}
Let~$z\in\diamondsuit^\delta$ be a quad of a rhombic lattice~$\Lambda^\delta$ with mesh size~$\delta$.
Given a real-valued spinor $X^{(m),\delta}$ satisfying the three-terms identity~\eqref{eq:3-terms} on corners adjacent to~$z$ we define a complex value~$F^{(m),\delta}(z)$ by requiring that
\begin{equation}
\label{eq:shol-def-mass}
\Pr\,[\,F^{(m),\delta}(z)\,;\,\widehat{\eta}_{c,z}\R\,]\ =\ \delta^{-\frac12}\cdot \widehat{\eta}_{c,z}X(c)\ \ \text{for all}\ \  c\sim z\,,
\end{equation}
where
\[
\widehat{\eta}_{c,z}\ :=\ \eta_c\cdot \exp\big[\mp \tfrac{i}{2}(\thetaa_z-\thetag_z)\big]\ =\ \varsigma\cdot\exp\big[-\tfrac{i}{2}(\arg(v(c)-z)\pm \thetaa_z)\big]
\]
and the $\pm$ sign is chosen so that~$\arg(v(c)-u(c))=\arg(v(c)-z)\pm\thetag_z$; see Fig.~\ref{fig:notation}.
\end{definition}
\begin{remark} Functions~$F^{(m),\delta}:\diamondsuit^\delta\to\C$ satisfying the condition~\eqref{eq:shol-def-mass} were called \emph{massive \mbox{s-holomorphic}} in~\cite{Park2021Fermionic} by analogy with more common s-holomorphic functions that appear in the case~$m=0$ (and hence~$\thetaa_z=\thetag_z$); see~\cite{Smirnov_Ising,ChelkakSmirnov2}. However, note that the values~$\eta_c X^{(m),\delta}(c)$ differ from the values~$F^{(m),\delta}(c)$ used in~\cite{Smirnov_Ising,ChelkakSmirnov2,Park2021Fermionic} and other related papers by the factor~$\delta^{-\frac12}$. We adopt this convention on \emph{different} scalings of Kadanoff--Ceva (real-valued) and Smirnov (complex-valued) fermionic observables as it better fits the general framework developed in~\cite{chelkak2020ising}.
\end{remark}
\begin{lemma} \label{lem:H=intImF2}
Let~$H^{(m),\delta}_X$ and~$F^{(m),\delta}$ be constructed from the spinor~$X^{(m),\delta}$ via identities~\eqref{eq:HX-def} and~\eqref{eq:shol-def-mass}, respectively. Then (see Fig.~\ref{fig:notation} for the notation),
\begin{align*}
H^{(m),\delta}_X(v_1)-H^{(m),\delta}_X(v_0)\ &=\ \frac{\cos\thetaa_z}{\cos\thetag_z}\cdot \tfrac{1}{2}\Im\big[(F(z))^2\cdot (v_1-v_0)\big]\,,\\
H^{(m),\delta}_X(u_1)-H^{(m),\delta}_X(u_0)\ &=\ \frac{\sin\thetaa_z}{\sin\thetag_z}\cdot \tfrac{1}{2}\Im\big[(F(z))^2\cdot (u_1-u_0)\big]\,.
\end{align*}
Similar identities hold for functions~$H[X_1^{(m),\delta},X_2^{(m),\delta}]=\frac{1}{4}\big(H^{(m),\delta}_{X_1+X_2}-H^{(m),\delta}_{X_1-X_2}\big)$.
\end{lemma}
\begin{proof} Both formulas easily follow from a similar computation for~$m=0$ (e.\,g., see~\cite[Proposition~3.6]{ChelkakSmirnov2}) performed for a `virtual' rhombus with half-angle~$\thetaa_z$ instead of~$\thetag_z$; see Fig.~\ref{fig:notation} and~\cite[Lemma~2.5]{Park2021Fermionic}. (The additional factor~$\frac{1}{2}$ appears due to a tiny mismatch between the notation used in our paper and that in~\cite{Smirnov_Ising,ChelkakSmirnov2,Park2021Fermionic}.)
\end{proof}
\begin{remark}\label{rem:H=intImF2}
 Below we often refer to the identities from Lemma~\ref{lem:H=intImF2} by writing
\[
H^{(m),\delta}_X\ =\ \frac{1}{2}\int^{[(m),\delta]}\Im\big[\,(F^{(m),\delta}(z))^2dz\,\big]
\]
(or simply~$\int^{[\delta]}$ if $m=0$). Though this identity is not true `as is' even for discrete contour integrals due to the presence of additional factors~$\cos\thetaa_z/\cos\thetag_z$ and $\sin\thetaa_z/\sin\thetag_z$ in one-step increments, these factors (uniformly under the assumption~$\BAP$) disappear as~$\delta\to 0$. In other words, if functions~$F^{(m),\delta}$ convergence to a continuous function~$f$, then the corresponding functions~$H^{(m),\delta}$ converge to the function~$h:=\frac{1}{2}\int\Im[(f(z))^2dz]$.
\end{remark}
The following lemma and its corollary are adopted from~\cite[Section~A.1]{Park2021Fermionic}. They imply that limits of massive s-holomorphic functions on refining isoradial grids~$\Lambda^\delta$, if { they} exist, are massive holomorphic, i.\,e., satisfy the equation~\eqref{eq:mass-hol}.

\begin{lemma}\label{lem:mass-hol-discrete} Let~$u^{-}vu^{+}$ be a half-rhombus of an isoradial grid~$\Lambda^\delta$ oriented counterclockwise (i.\,e., either~$u^-=u_0$, $v=v_1$, $u^+=u_1$ or $u^-=u_1$, $v=v_0$, $u^+=u_0$ in the notation of Fig.~\ref{fig:notation}) and~$X^{(m),\delta}$ satisfies the identities~\eqref{eq:3-terms}. Then,
\begin{align*}
\frac{\sin\frac{1}{2}(\thetaa_z+\thetag_z)}{2\sin\thetag_z}\,F^{(m),\delta}(z)\,(u^-\!-u^+)\ &+\ F^{(m),\delta}(c^-)(v-u^-)\ +\ F^{(m),\delta}(c^+)(u^+\!-v)\\
&=\ \frac{\sin\frac{1}{2}(\thetaa_z-\thetag_z)}{\delta\sin\thetag_z\cos\thetag_z}\,\overline{F^{(m),\delta}(z)}\cdot\mathrm{Area}(u^-vu^+)\,,
\end{align*}
where the value~$F^{(m),\delta}(z)$ is defined by~\eqref{eq:shol-def-mass} and~$F^{(m),\delta}(c^\pm):=\eta_{c^\pm}\cdot\delta^{-\frac12}X^{(m),\delta}(c^\pm)$. A similar identity (with the prefactor~$\cos\frac{1}{2}(\thetaa_z+\thetag_z)/2\cos\thetag_z$ in the first term) holds for half-rhombi of the form~$v^-uv^+$.
\end{lemma}
\begin{proof} Let~$v-u^\pm=\delta\cdot e^{i(\phi_v\mp\thetag_z)}$, where~$\phi_v:=\arg(v-z)$. It follows from~\eqref{eq:shol-def-mass} that
    \begin{align*}
    F^{(m),\delta}(c^\pm)\ 
    =&\ \varsigma e^{-\frac{i}{2}(\phi_v\mp \thetag_z)}\cdot \tfrac{1}{2}\big[F^{(m),\delta}(z)\cdot \overline{\varsigma}e^{\frac{i}{2}(\phi_v\mp\thetaa_z)}+\overline{F^{(m),\delta}(z)}\cdot\varsigma e^{-\frac{i}{2}(\phi_v\mp\thetaa_z)}\big]\\[2pt]
    =&\ \tfrac{1}{2}\big[\,F^{(m),\delta}(z)\cdot e^{\pm\frac{i}{2}(\thetag_z-\thetaa_z)}+\overline{F^{(m),\delta}(z)}\cdot ie^{-i\phi_v\pm\frac{i}{2}(\thetaa_z+\thetag_z)}\,\big],
    \end{align*}
recall that~$\varsigma=e^{i\frac\pi4}$. The desired identity follows by a straightforward computation as
$u^+-u^-=2\delta\cdot ie^{i\phi_v}\sin\thetag_z$ and $\mathrm{Area}(u^-vu^+)=\delta^2\sin\thetag_z\cos\thetag_z$.
\end{proof}

\begin{corollary}\label{cor:mass-hol}
Let~$F^{(m),\delta}$ be massive s-holomorphic functions on isoradial grids~$\Lambda^\delta$ satisfying the uniform bounded angles condition~$\BAP$. Assume that~$F^{(m),\delta}$ { converge} (on a certain open set~$U\subset\C$) to a function~$f^{(m)}:U\to\C$ as~$\delta\to 0$. Then,~$f^{(m)}$ is differentiable and satisfies the massive holomorphicity equation
\begin{equation}
\label{eq:mass-hol}
\overline{\partial}f^{(m)}+im\overline{f^{(m)}}=0\,.
\end{equation}
\end{corollary}
\begin{proof} It follows from the asymptotics~\eqref{eq:thetaa-thetag} that
\[
\frac{\sin\frac{1}{2}(\thetaa_z+\thetag_z)}{2\sin\thetag_z}\to\frac{1}{2}\,,\qquad
\frac{\sin\frac{1}{2}(\thetaa_z-\thetag_z)}{\delta\sin\thetag_z\cos\thetag_z}\to m\ \ \text{as}\ \ \delta\to 0\,,
\]
uniformly with respect to~$z$ (provided that the condition~$\BAP$ holds). Therefore, summing the identity of Lemma \ref{lem:mass-hol-discrete} over any subset $V\subset U$ with smooth boundary, and passing to the limit, leads to the identity
\[
\frac{1}{2}\oint_{\partial V}f^{(m)}(z)dz\ =\ m\iint_V\overline{f^{(m)}(z)}dA(z),
\]
which is nothing but the weak form of the equation~$i\overline{\partial}f^{(m)}=m\overline{f^{(m)}}$. A massive version of Morera's theorem (or Cauchy integral formula) then implies that any weak solution is differentiable.
\end{proof}
\begin{remark}\label{rem:Delta-h} It is easy to see that, if~$f^{(m)}$ satisfies the massive holomorphicity equation~\eqref{eq:mass-hol}, then the primitive~$h^{(m)}:=\frac{1}{2}\int\Im[(f^{(m)}(z))^2dz]$ is well defined and~$\Delta h^{(m)}=-4m|f^{(m)}|^2$. In particular, the function~$h^{(m)}$ is subharmonic if~$m\le 0$.
\end{remark}

In the next proposition we show that the subharmonicity property mentioned in Remark~\ref{rem:Delta-h} also holds for \emph{discrete} primitives of massive s-holomorphic functions on isoradial grids; a property that we need in the proof of Proposition~\ref{prop:bc-are-OK}. Note that the paper~\cite{Park2021Fermionic} also contains estimates of the discrete Laplacian
\begin{equation}
\label{eq:Delta-def}
[\Delta^{\bullet,\delta} H^{(m),\delta}](v)\ :=\ (\mu^\delta(v))^{-1}\sum_{s=1}^n\,\tan\thetag_s\cdot (H^{(m),\delta}(v_s)-H^{(m),\delta}(v)),
\end{equation}
where~$v_1,\ldots,v_n\in \Gamma^{\bullet,\delta}$ are the neighbors of~$v\in \Gamma^{\bullet,\delta}$ and~$\mu^\delta(v)=\frac{1}{2}\delta^2\sum_{s=1}^n\sin 2\theta_s$ is an appropriately chosen normalizing factor, which is irrelevant in what follows.

\begin{proposition}\label{prop:H-sub} Let~$m\le 0$ and~$\delta\le\delta_0(\theta_0)$ be small enough. Then, for each isoradial grid~$\Lambda^\delta$ satisfying the condition~$\BAP$ and for each spinor~$X^{(m),\delta}$ satisfying the propagation equation~\eqref{eq:3-terms} near $v$, we have
\[
[\Delta^{\bullet,\delta}H^{(m),\delta}](v)\ \ge\ 0,\quad  v\in\Gamma^{\bullet,\delta}\,,
\]
where the function~$H^{(m),\delta}$ is constructed from~$X^{(m),\delta}$ via the identity~\eqref{eq:HX-def}. In other words, $H^{(m),\delta}$ is discrete subharmonic on~$\Gamma^{\bullet,\delta}$ if~$m\le 0$.
\end{proposition}

\begin{remark} Similarly, the function $H^{(m),\delta}$ is superharmonic on~$\Gamma^{\circ,\delta}$ if~$m\ge 0$. If~$m=0$, then \emph{both} these properties hold (see~\cite[Lemma~3.8]{Smirnov_Ising} and~\cite[Proposition~3.6]{ChelkakSmirnov2}), which considerably simplifies the analysis at criticality.
\end{remark}

\begin{proof}

We adopt the notation used in the proof of~\cite[Propostion~3.6]{ChelkakSmirnov2}. Let~$x_s\in\R$, $s=1,\ldots,n$, be the values of the spinor~$X^{(m),\delta}$ at corners surrounding~$v$. Then,
\[
[\Delta^{\bullet,\delta} H^{(m),\delta}](v)\ =\ (\mu^\delta(v))^{-1}\sum_{s=1}^n\, \tan\thetag_s\cot^2\thetaa_s\cdot\bigl[x_s^2+x_{s+1}^2\mp 2(\cos\thetaa_s)^{-1}x_sx_{s+1}\bigr],
\]
where the sign~$\mp$ stands for~$-$ if~$s=1,\ldots,n-1$ and for~$+$ if~$s=n$ (this convention corresponds to the fact that the double cover~$\Upsilon^\times(\Lambda^\delta)$, on which the spinor~$X^{(m),\delta}$ is defined, branches around~$v$). We now expand this quadratic form in the parameter~$q\to 0$. Since~$\tan\thetaa_s=(1+4q)\tan\thetag_s+O(q^2)$ for all~$s=1,\ldots,n$, this leads to the expression
\[
(1-8q)Q^{(n)}_{\thetag_1,\ldots,\thetag_n}(x_1,\ldots,x_n)-8qR^{(n)}_{\thetag_1,\ldots,\thetag_n}(x_1,\ldots,x_n)+O(q^2),
\]
where the leading term
\[
Q^{(n)}_{\thetag_1,\ldots,\thetag_n}(x_1,\ldots,x_n)\ =\ \sum_{s=1}^n\frac{\cos\thetag_s\cdot (x_s^2+x_{s+1}^2)\mp 2x_sx_{s+1}}{\sin\thetag_s}\ \ge\ 0
\]
corresponds to the critical case~$m=0$ and
\[
R^{(n)}_{\thetag_1,\ldots,\thetag_n}(x_1,\ldots,x_n)\ =\ \sum_{s=1}^{n-1}x_sx_{s+1}\sin\thetag_s-x_1x_n\sin\thetag_n.
\]
From the proof of~\cite[Proposition~3.6]{ChelkakSmirnov2} it is easy to see that the (two-dimensional) kernel of the form~$Q^{(n)}_{\thetag_1,\ldots,\thetag_n}$ consists of vectors~$x_s=\mathrm{cst}\cdot \cos \chi_s$, where~$\chi_{s+1}=\chi_s-\thetag_s$ for all~$s=1,\ldots,n-1$. A straightforward computation shows that
\begin{align*}
  R^{(n)}_{\thetag_1,\ldots,\thetag_n}(\cos\chi_1,\ldots,\cos\chi_n)\ &=\  R^{(n-1)}_{\thetag_1,\ldots,\thetag_{n-2},\thetag_{n-1}+\thetag_n}(\cos\chi_1,\ldots,\cos\chi_{n-1})\\ &
  +\ R^{(3)}_{\pi-\thetag_{n-1}-\thetag_n,\thetag_{n-1},\thetag_n}(\cos\chi_1,\cos\chi_{n-1},\cos\chi_n)
  \end{align*}
  and that~$R^{(3)}_{\thetag_1,\thetag_2,\thetag_3}(\cos\chi_1,\cos\chi_2,\cos\chi_3)=\sin\thetag_1\sin\thetag_2\sin\thetag_3$ if $\thetag_1+\thetag_2+\thetag_3=\pi$.

   Therefore, the form $R^{(n)}_{\thetag_1;\ldots;\thetag_n}$ is strictly positive definite on the kernel of~$Q^{(n)}_{\thetag_1;\ldots;\thetag_n}$, which proves the required positivity property for small enough~$\delta$.
\end{proof}

\begin{remark}\label{rem:bdry-trick} It is worth noting that the so-called \emph{'boundary modification trick'} (used in~\cite[Section~3.6]{ChelkakSmirnov2} to control the boundary values of the function~$H^{(m),\delta}_X$) admits a straightforward generalization to the massive setup. By definition, on the wired boundary of a discrete domain~$\Omega^\delta$, the function~$H^{(m),\delta}_X$ satisfies Dirichlet boundary conditions~$H^{(m),\delta}|_{\partial\Omega^{\circ,\delta}}=0$ on 'white' boundary vertices.
Following~\cite{ChelkakSmirnov2}, in order to fit these boundary conditions and the subharmonicity of~$H^{(\delta),m}_X$ on 'black' inner vertices, one replaces each boundary half-rhombus by two rhombi with twice smaller angles and formally define~$H^{(\delta),m}_X(v):=0$ on newly constructed 'black' boundary vertices; see~\cite[Fig.~7]{ChelkakSmirnov2}. For~$m<0$, such a modification \emph{increases} the Laplacian~\eqref{eq:Delta-def} evaluated at near-to-boundary vertices since
\[
\tan\thetag_z\cot^2\thetaa_z\cdot [2x^2-2(\cos\thetaa_z)^{-1}x^2]\ =\ -2x^2\tan\thetag_z\cot\thetaa_z\tan\tfrac{1}{2}\thetaa_z\ <\ -2x^2\tan\tfrac{1}{2}\thetag_z
\]
due to~$\thetaa_z<\thetag_z$ and the monotonicity of the function~$\theta\mapsto \cot\theta\tan\frac{1}{2}\theta$ on~$(0,\frac{1}{2}\pi)$.

In particular, after the 'boundary modification trick' from~\cite[Section~3.6]{ChelkakSmirnov2} is performed, the function~$H^{(m),\delta}$ \emph{remains subharmonic} on 'black' vertices and has Dirichlet boundary conditions; similarly to the critical setup discussed in~\cite{ChelkakSmirnov2}.
\end{remark}

\begin{proposition} \label{prop:regularity}
Let refining isoradial grids~$\Lambda^\delta$ satisfy the property~$\BAP$ and~$F^{(m),\delta}$ be massive s-holomorphic functions on $\diamondsuit^\delta\cap U$. Assume that the functions~$H^{(m),\delta}=\frac{1}{2}\int^{[(m),\delta]}\Im[(F^{(m),\delta}(z))^2dz]$ remain uniformly bounded on compact subsets of~$U$ as~$\delta\to 0$. Then~$F^{(m),\delta}$ are also uniformly bounded and, moreover, (H\"older-)equicontinuous on compact subsets of~$U$ as~$\delta\to 0$.
\end{proposition}

\begin{proof} Without loss of generality assume that~$m\le 0$, the other case follows by exchanging the roles of~$\Gamma^{\bullet,\delta}$ and~$\Gamma^{\circ,\delta}$. There are two different proofs of the required regularity estimates. For the first one -- which actually gives the uniform \emph{Lipschitzness} of functions~$F^{(m),\delta}$ -- we refer the reader to~\cite[Section~4.2]{Park2021Fermionic} and notably to ~\cite[Propositon~4.6 and Propoition~A.7]{Park2021Fermionic}. In this approach, one first estimates the $L^2$ norms of functions~$F^{(m),\delta}$ on compacts via the maximum (or oscillations) of~$H^{(\delta),m}$. Then, the pointwise estimate and the Lipschitzness of~$F^{(m),\delta}$ can be obtained by applying an appropriate discrete massive Cauchy formula and using asymptotics of the massive s-holomorphic Cauchy kernel { $\cG_{(a)}$ discussed in Section~\ref{sec:asymptotics}: see definition~\eqref{eq: def_chi_chi}, Proposition~\ref{prop: kernels_values} and asymptotics~\eqref{eq: asymp_chi_chi}.} Note also that one can give a similar proof by first estimating the $L^4$ norm of~$F^{(m),\delta}$ (i.\,e., the~$L^2$ norm of the gradient of~$H^{(\delta),m}$) via a discrete version of the Caccioppoli inequality applied to bounded subharmonic (on~$\Gamma^{\bullet,\delta}\cap U$) functions~$H^{(m),\delta}$.

The second proof relies upon an s-embeddings framework developed in~\cite{chelkak2020ising} and related to the context of this paper in Section~\ref{sub:s-emb} below. In this approach, one \emph{re-embeds} the isoradial grid~$\Lambda^\delta$ carrying the massive Ising model so as to obtain the Ising model on an appropriate s-embedding~$\cS^\delta$
While the (real-valued) Kadanoff--Ceva fermionic observables~$X^{(m),\delta}$ and the functions~$H^{(m),\delta}$ do not depend on a particular way in which the graph is embedded into~$\C$, the complex values (Smirnov's observables) $F^{(\delta),m}(z)$ \emph{change} to new values $F_{\cS^\delta}(z)$ under this procedure; the relation between the two is given by
\begin{equation}
\label{eq:x-Fm=Fs}
F^{(m),\delta}(z)\ =\ \cD_+^{\delta}(z)\cdot F_{\cS^\delta}(z)+\cD_-^{\delta}(z)\cdot \overline{F_{\cS^\delta}(z)}\,,\quad z\in\diamondsuit^\delta\,;
\end{equation}
see Proposition~\ref{prop:Fm=FS} for the definition of coefficients~$\cD^\delta_\pm(z)$.

The H\"older regularity of massive s-holomorphic functions~$F^{(m),\delta}$ on (subsets of)~$\Lambda^\delta$ now follows from the regularity of s-holomorphic functions~$F_{\cS^\delta}$ on~$\cS^\delta$ (see~\cite[Section~2.6]{chelkak2020ising}) and from the fact that the mappings~$z\mapsto\cS^\delta(z)$ and~$z\mapsto \cD_\pm^\delta(z)$ are uniformly Lipschitz on compact subsets of~$\Lambda^\delta$ due to Theorem~\ref{thm:F1Fi-asymp}.
\end{proof}
\begin{lemma}\label{lem:max-principle} Let~$\Lambda^\delta$ satisfy the property~$\BAP$,~$F^{(m),\delta}$ be a massive s-holo\-morphic function defined inside a nearest-neighbor contour~$C^\delta\subset\diamondsuit^\delta$. Then, { provided that~$\delta$ is small enough (depending on~$\diam(C^\delta)$ and~$m$ only)} we have
\begin{equation}
\label{eq:max-principle}
|F^{(m),\delta}(\cdot)|\ \le\ \mathrm{cst}(\theta_0,m,\mathrm{diam}(C^\delta))\cdot\max\nolimits_{z\in C^\delta}|F^{(m),\delta}(z)|
\end{equation}
at all points lying inside~$C^\delta$, where the constant does not depend on~$\delta$ and/or~$\Lambda^\delta$.
\end{lemma}
\begin{proof} Without loss of generality, assume that $0$ lies inside the contour~$C^\delta$ and consider an s-embedding~$\cS^\delta$ of the massive Ising model on~$\Lambda^\delta$ (see Section~\ref{sub:s-emb}) below. Recall that the values~$F^{(m),\delta}(z)$ and~$F_{\cS^\delta}(\cS^\delta(z))$ are related to each other by the formula~\eqref{eq:x-Fm=Fs}. The s-holomorphic (on~$\cS^\delta$) function~$F_{\cS^\delta}$ satisfies the maximum principle; see~\cite[Remark~2.9]{chelkak2020ising}. Therefore, a similar statement for~$F^{(m),\delta}$ follows from the (uniform on bounded subsets) estimates
\[
|\cD^\delta_-(z)|\,/\,|\cD^\delta_+(z)|\le\mathrm{cst}<1\quad \text{and}\quad |\cD^\delta_+(z)|\le\mathrm{cst}<+\infty.
\]
In their turn, these estimates follow from the identities~$\cD_\pm^\delta=\frac{1}{2}(\cF_\rr^\delta\mp i\cF_\ri^\delta)$ and asymptotics of the functions~$\cF_\rr^\delta$, $\cF_\ri^\delta$ given in Theorem~\ref{thm:F1Fi-asymp}.
\end{proof}
\begin{remark} It is worth noting that the massive holomorphicity equation~\eqref{eq:mass-hol} implies that~$\Delta f= 4m^2f$ and thus a usual maximum principle for $|f|$. The same holds for massive s-holomorphic functions $F^{(m),\delta}$ defined on the square grid~$\delta\mathbb{Z}^2$. Thus, it seems plausible that the constant prefactor in~\eqref{eq:max-principle} is unnecessary at least as~$\delta\to 0$. However, we do not know a proof of such a statement for irregular~$\Lambda^\delta$.
\end{remark}

\subsection{Full-plane branching discrete kernels and their asymptotics}
\label{sub:kernels-def}
In our paper we very often use functions~$H[X_1,X_2]$ constructed from two spinors~$X_1,X_2$ satisfying the propagation equation~\eqref{eq:3-terms}, which are (locally) defined on slightly \emph{different} double covers~$\Upsilon^\times_{[v]}(G)$ and~$\Upsilon^\times_{[u]}(G)$, where~$v\in G^\bullet$ and~$u\in G^\circ$ are neighboring vertices. Let the corner~$c\in\Upsilon(G)$ be adjacent to both~$u$ and~$v$. Note that these two double covers (non-branching over~$u$ and non-branching over~$v$) can be naturally identified with each other everywhere except the two lifts of $c$.  Let~$z^{\pm}\in\diamondsuit(G)$ be two quads adjacent to the edge~$(uv)$ of~$\Lambda(G)$ so that~$u$ is the next vertex to~$v$ when going around~$z^+$ counterclockwise; see { Fig.~\ref{fig:Uu=Uv}}. In what follows we assume that
\begin{equation}
\label{eq:Uu=Uv-convention}
\begin{array}{l}
\text{the lifts of~$c=(uv)$ onto $\Upsilon^\times_{[v]}$ and $\Upsilon^\times_{[u]}$ are identified in such a way that}\\
\text{the structure of these double covers around the quad~$z^+$ is the same.}
\end{array}
\end{equation}

\begin{lemma} \label{lem:XG-monodromy}
In the setup described above, let us define increments of the function~$H[X_1,X_2]:\Lambda^\delta\to\R$ via the formula~\eqref{eq:HX1X2-def}. Then,~$H[X_1,X_2]$ has an additive monodromy~$2X_1(c)X_2(c)$ when going around~$c$ counterclockwise. In particular, if~$X_{1,2}=X_{1,2}^{(m),\delta}$ are defined on an isoradial grid~$\Lambda^\delta$ and the massive s-holomorphic functions~$F_{1,2}^{(m),\delta}$ are constructed from~$X^{(m),\delta}_{1,2}$ according to Definition~\ref{def:shol-def-mass}, then
\begin{equation}
\label{eq:XG-monodromy}
2X^{(m),\delta}_1(c)X^{(m),\delta}_2(c)\ =\ \frac{1}{2}\oint^{[(m),\delta]} \Im\big[F^{(m),\delta}_1(z)F^{(m),\delta}_2(z)dz\big],
\end{equation}
where the discrete integral along a closed contour surrounding~$c=(uv)$ in the right-hand side is understood in the sense of Lemma~\ref{lem:H=intImF2} and Remark~\ref{rem:H=intImF2}.
\end{lemma}

\begin{figure}
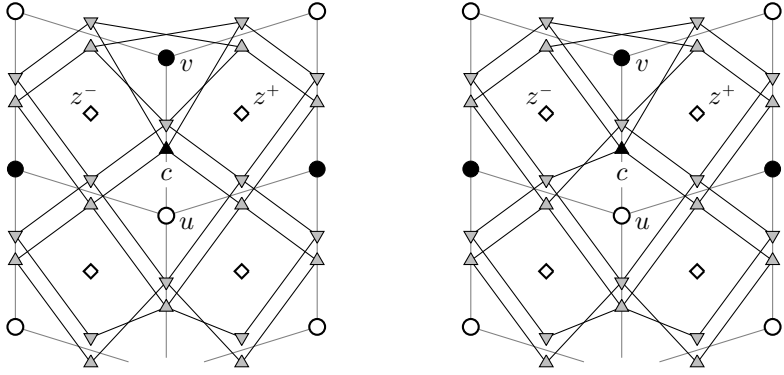

\begin{center}\begin{tikzpicture}[scale=0.21]
\input{BranchingU.txt}
\end{tikzpicture}
\hskip 48pt
\begin{tikzpicture}[scale=0.21]
\input{BranchingV.txt}
\end{tikzpicture}\end{center}
\caption{{ Local structure of the double covers~$\Upsilon^\times_{[u]}$, $u=u(c)$ (on the left) and~$\Upsilon^\times_{[v]}$, $v=v(c)$ (on the right). These double covers can be identified with each other except at lifts of~$c$. When choosing such a lift (shown as a black triangular node) on both double covers \emph{simultaneously} we use the convention~\eqref{eq:Uu=Uv-convention} that their structure around the quad~$z^+$ lying to the right of~$(uv)$ is the same.}}
\label{fig:Uu=Uv}
\end{figure}

\begin{proof} Under the convention~\eqref{eq:Uu=Uv-convention}, the increments of the function~$H[X_1,X_2]$ sum up to zero when going around all quads~$z\in\diamondsuit(G)$ except~$z=z^-$. The sum around~$z^-$ (when going counterclockwise) is~$2X_1(c)X_2(c)$ since it would vanish if we used the opposite convention on the identification of~$\Upsilon^\times_{[v]}(X_1)$ and~$\Upsilon^\times_{[u]}(G)$ in~\eqref{eq:Uu=Uv-convention}.
\end{proof}
\begin{remark}
Note that a similar statement holds if, say, the spinor~$X_1$ is (locally) defined on the double cover~$\Upsilon^\times(G)$ while~$X_2$ lives on~$\Upsilon^\times_{ [v(a),u(a)]}(G)$, { $a\in\Upsilon^\times(G)$}. { If~$X_2=\rG_{ (a)}$} is the massive Cauchy kernel on~$\Lambda^\delta$, this gives a Cauchy-type formula for massive s-holomorphic functions;
see { Fig.~\ref{fig:U=Ua}, definition~\eqref{eq: def_chi_chi} of~$\rG_{(a)}$, asymptotics~\eqref{eq: asymp_chi_chi},} and~\cite[Section~A.2]{Park2021Fermionic} for more details.
\end{remark}

In our paper we typically apply the formula~\eqref{eq:XG-monodromy} in a situation when~$F^{(m),\delta}_2$ is a concrete massive s-holomorphic function, namely an analogue of the 'discrete $z^{-\frac12}$ kernel on~$\mathbb{Z}^2$' that was used in~\cite{ChelkakHonglerIzyurov} (and of its massive analogue constructed in~\cite{park2018massive} for~$m\le 0$ and~$w\in\Gamma^{\bullet,\delta}$). Similarly to the discrete Cauchy formula mentioned above, this formula provides a tool to reconstruct the value of~$X_1$ (or~$F_1$) right near the branching point~$v$ from its values on an arbitrary contour surrounding~$v$.

We construct and analyze the aforementioned branching kernels in Section~\ref{sec:asymptotics}; the main result of this analysis is given by the following theorem. We assume that $m\le 0$ but formulate the result for \emph{both}~$w=v\in\Gamma^{\bullet,\delta}$ and $w=u\in\Gamma^{\circ,\delta}$; a similar result for~$m\ge 0$ follows by the duality.
\begin{theorem}
\label{thm:G-asymp-mass}
Let~$m\le 0$ and let an isoradial grid~$\Lambda^\delta$ satisfy the bounded angles property~$\BAP$. For each~$w\in\Lambda^\delta$ there exists a real-valued spinor~$\rG^{(m),\delta}_{[w]}$ defined on the double cover~$\Upsilon^\times_{[w]}(\Lambda^\delta)$ and satisfying the propagation equation~\eqref{eq:3-terms}, such that
\begin{equation}
\label{eq:G(c)=pi/2}
{ \rG^{(m),\delta}_{[w]}(c)\,=\,1}\ \ \text{for~all~corners}\ \ c\sim w
\end{equation}
and that the massive s-holomorphic spinor~$\cG^{(m),\delta}_{[w]}$ constructed from~$\rG^{(m),\delta}_{[w]}$ according to Definition~\ref{def:shol-def-mass} { has the following asymptotics as~$\delta\to 0$:}
\begin{align*}
\cG^{(m),\delta}_{[v]}(z)\ &=\ e^{-i\frac{\pi}{4}}\cdot { \biggl(\frac{2}{\pi}\biggr)^{\!\!\frac12}}\frac{e^{2m|z-v|}}{\sqrt{z-v}}+\delta\cdot { R^{(m)}_{\,\bullet}}(z,v)+{ O(\delta^2)}\ \ \text{if}\ \ w=v\in\Gamma^{\bullet,\delta},\\
\cG^{(m),\delta}_{[u]}(z)\ &=\ e^{i\frac{\pi}{4}}\cdot { \biggl(\frac{2}{\pi}\biggr)^{\!\!\frac12}}\frac{e^{-2m|z-u|}}{\sqrt{z-u}}+\delta\cdot  { R^{(m)}_{\,\circ}}(z,u)+{ O(\delta^2)}\ \ \text{if}\ \ w=u\in\Gamma^{\circ,\delta},
\end{align*}
where the sub-leading terms $R^{(m)}_{\,\bullet}$ and $R^{(m)}_{\,\circ}$ are uniformly bounded and uniformly Lipschitz in the second (i.\,e., $v$ or~$u$) argument provided that~$|z-w|$ and~$|z-w|^{-1}$ are uniformly bounded, and the error terms are uniform under the same assumptions.
\end{theorem}
\begin{proof}
See Section~\ref{sec:asymptotics}: { the spinors~$\rG_{[w]}$ are defined in~(\ref{eq: def_chi_mu}--\ref{eq: def_chi_sigma}), the identity~\eqref{eq:G(c)=pi/2} is checked in Proposition~\ref{prop: kernels_values}, and the asymptotics of~$\rG_{[w]}$ are given by~(\ref{eq: asymp_chi_mu}--\ref{eq: asymp_chi_sigma}). Note that these asymptotics do  \emph{not} contain corrections of order~$\delta$; in other words, the terms~$R^{(m)}_{\,\circ}$,~$R^{(m)}_{\,\bullet}$ appear only along the reconstruction of~$\cG_{[w]}$ from~$\rG_{[w]}$ via Definition~\ref{def:shol-def-mass} and thus can be written explicitly. In particular, the fact that they are bounded and Lipschitz in the second variable trivially follows from~(\ref{eq: asymp_chi_mu}--\ref{eq: asymp_chi_sigma}).}
\end{proof}
\begin{remark}\label{rem:G-asymp} (i) Informally speaking, these kernels can be thought of as properly re-scaled infinite-volume correlators~$\langle\chi_c\mu_v\sigma_\infty\rangle$ or~$\langle\chi_c\mu_\infty\sigma_u\rangle$, respectively. This interpretation can be made rigorous if~$w=v\in\Gamma^{\bullet,\delta}$ (and $m<0$) by considering a limit of finite-volume correlators; { cf. the proof of Lemma~\ref{lem:E[sigma]=cst} given below. Note that the exponential decay at infinity characterizes such a kernel uniquely up to a multiplicative normalization. { (In fact, it is not hard to deduce from Lemma~\ref{lem:maxH-principle} and Remark~\ref{rem:maxH-principle} that already the estimate $O(R^{-1/2-\varepsilon})$ at infinity implies such a uniqueness property.)} Moreover, 
the fact that~\eqref{eq:G(c)=pi/2} simultaneously holds for all~$c\sim v$} is nothing but a re-statement of the universality of the magnetization in the sub-critical model. However, a similar interpretation for~$w=u\in\Gamma^{\circ,\delta}$ (or~$w=v\in\Gamma^{\bullet,\delta}$ and~$m>0$) is less transparent since { in this case the kernel \emph{grows} when~$|z-u|\to\infty$ and thus does not admit a straightforward characterization.}

\smallskip

\noindent (ii) In the critical case~$m=0$ the asymptotics simplify to
\begin{equation}
\label{eq:G-asymp-crit}
\cG^\delta_{[w]}(z)\ =\ e^{\mp i\frac{\pi}{4}}\cdot { \big(\tfrac{2}{\pi}\big)^{\frac{1}{2}}(z-w)^{-\frac12}}\ +\ O(\delta^2\cdot |z-w|^{-\frac 52}),
\end{equation}
where the constant prefactor is~$e^{-i\frac\pi4}$ if~$w=v\in\Gamma^{\bullet,\delta}$ and~$e^{i\frac\pi4}$ if~$w=u\in\Gamma^{\circ,\delta}$; see Section~\ref{sub:kernels-crit} for details.
\end{remark}

\smallskip

\noindent (iii) In Section~\ref{sub:spin-sub} we also rely upon the following statement. Let~$\delta=1$, \mbox{$ q=\frac{1}{2}m<0$}, and~$w=v\in\Gamma^\bullet$. Then, the spinor~$\rG_{[v]}=\rG^{(\delta),m}_{[v]}$ constructed in Theorem~\ref{thm:G-asymp-mass} \mbox{satisfies} the uniform bound
\begin{equation}
\label{eq:G-exp-decay-sub}
|\rG_{[v]}(c)|\ =\ O(\exp(-\beta(q,\theta_0)\cdot|c-v|)) \ \ \text{as}\ \ |c-v|\to\infty,
\end{equation}
where a constant $\beta(q,\theta_0)>0$ does not depend on~$\Lambda$. We discuss this estimate (note that it does not directly follow from Theorem~\ref{thm:G-asymp-mass}) in Section~\ref{sub:exp-asymp}.

\subsection{S-embeddings of the massive isoradial Ising model} \label{sub:s-emb}
In this section we discuss the link of the massive Ising model on isoradial graphs considered in our paper with a general framework of \emph{s-embeddings} recently developed in~\cite{chelkak2020ising}. From a certain perspective, the material presented in this section can be viewed as an illustration of the general construction from~\cite{chelkak2020ising}. However, note that we also rely upon this link -- namely, upon Proposition~\ref{prop:Fm=FS} -- when giving an (alternative to~\cite{Park2021Fermionic}) proof of the a priori regularity of massive s-holomorphic functions on isoradial grids and of the discrete maximum principle for such functions.

Recall that an s-embedding~$\cS=\cS_\cX$ of a given planar graph carrying a nearest-neighbor Ising model is constructed out of a complex-valued solution~$\cX$ of the propagation equation~\eqref{eq:3-terms} or, equivalently, out of two (linear independent) real-valued solutions~$\cX_\rr$ and~$\cX_\ri$ of this equation such that~$\cX=\varsigma\cdot(\cX_\rr-i\cX_\ri)$; recall that~$\varsigma=e^{i\frac\pi 4}$. This construction boils down to the definition~\eqref{eq:HX-def} applied to~$\cX$. Namely, one has
\[
\mathrm{Re}\,\cS_\cX =\ 2H[\cX_\rr,\cX_\ri]\,,\quad \mathrm{Im}\,\cS_\cX =\ H_{\cX_\rr}\!- H_{\cX_\ri}\,,\quad
\cQ_\cX\ =\ H_{\cX_\rr}\!+ H_{\cX_\ri}\,,
\]
see~\cite{chelkak2020ising} for more details on the definition of the auxiliary function~$\cQ_\cX=H[\cX,\overline{\cX}]$ .

In the context of the massive Ising model on an isoradial grid~$\Lambda^\delta$, the two spinors~$\cX_\rr^\delta$, $\cX_\ri^\delta$ that one uses for a definition of an s-embedding~$\cS^\delta$ also give rise to massive s-holomorphic functions~$\cF_\rr^\delta$, $\cF_\ri^\delta$ via Definition~\ref{def:shol-def-mass}. Moreover, due to Lemma~\ref{lem:H=intImF2} (and using the notation of Remark~\ref{rem:H=intImF2}) we have
\begin{align}
\notag \mathrm{Re}\,\cS^\delta\ &=\ \textstyle \frac{1}{2}\int^{[(m),\delta]}\Im\big[\,2\cF_\rr^\delta(z)\cF_\ri^\delta(z)dz\,\big]\,,\\[2pt]
\label{eq:SembQ=int} \mathrm{Im}\,\cS^\delta &=\ \textstyle \frac{1}{2}\int^{[(m),\delta]}\Im\big[((\cF_\rr^\delta(z))^2-(\cF_\ri^\delta(z))^2)dz\big]\,,\\[2pt]
\notag \cQ^\delta\ &=\ \textstyle \frac{1}{2}\int^{[(m),\delta]}\Im\big[((\cF_\rr^\delta(z))^2+(\cF_\ri^\delta(z))^2)dz\big]\,.
\end{align}
\begin{remark}
It is worth emphasizing that there is an enormous freedom in choosing~$\cF_\rr^\delta$ and~$\cF_\ri^\delta$ { and not only those constructed in the forthcoming Theorem~\ref{thm:F1Fi-asymp}.} If these functions converge (as~$\delta\to 0$, on compact subsets of~$\C$) to certain functions $f_\rr$,~$f_\ri$, then so do~$\cS^\delta$ and~$\cQ^\delta$. Moreover, the limit of~\eqref{eq:SembQ=int} is nothing but a conformal { (or isothermal)} parametrization of a constant curvature surface in the Minkowski space~$\R^{2+1}$ by two massive holomorphic (i.\,e., satisfying the equation~\eqref{eq:mass-hol}) functions~$f_\rr$, $f_\ri$; see~\cite[Section~2.7]{chelkak2020ising} for a discussion.
\end{remark}


\begin{theorem} \label{thm:F1Fi-asymp}
Given~$m\le 0$, on each isoradial grid~$\Lambda^\delta$ satisfying the property~$\BAP$ there exist massive s-holomorphic functions~$\cF_\rr^\delta$, $\cF_\ri^\delta$ such that the following asymptotics hold uniformly on compact sets as~$\delta\to 0$:
\[
\cF_\rr^\delta(z)\ =\ \exp(-2m\Im z)+O(\delta),\qquad \cF_\ri^\delta(z)\ =\ i\exp(2m\Im z)+O(\delta).
\]
With a proper choice of additive constants in their definitions, the corresponding s-embeddings~$\cS^\delta$ and the function~$\cQ^\delta$ have the following asymptotics as~$\delta\to 0$:
\begin{equation}
\label{eq:SembQ-asymp}
\begin{array}{rl}
\cS^\delta(z)&=\ \Re z +\tfrac{i}{4m}\sinh (4m\Im z)+O(\delta),\\[4pt]
\cQ^\delta(z)&=\ \tfrac{1}{4m}(1-\cosh (4m\Im z))+O(\delta),
\end{array}
\end{equation}
also uniformly on compact subsets of~$\C$.
\end{theorem}
\begin{proof} { See Section~\ref{sub: def_kernels} for the construction and Sections~\ref{sub:exp-asymp},~\ref{sub: massive_asymptotics_proof} for the asymptotic analysis of~$\cF_\rr^\delta$ and~$\cF_\ri^\delta$.} The asymptotics of~$\cS^\delta$ and~$\cQ^\delta$ easily follow from~\eqref{eq:SembQ=int}.
\end{proof}

In order to apply the results of~\cite[Section~2]{chelkak2020ising} we also need to check that~$\cS^\delta$ are \emph{proper} s-embeddings, i.\,e., that no edge intersections arise when we re-embed the isoradial grids~$\Lambda^\delta$ into the complex plane using~$\cS^\delta$. For the purposes of this paper it is enough to consider this re-embedding procedure on compact subsets only. Recall that we denote by~$\Lambda^\delta_R$ the discretization of the box~$[-R,R]\times [-R,R]$ on~$\Lambda^\delta$.

\begin{proposition} For each $R>0$ there exist $\delta_0=\delta_0(R,\theta_0)>0$ such that the following holds for all~$\delta\le\delta_0$ and all isoradial grids satisfying the property~$\BAP$:

\smallskip
\noindent $\cS^\delta$ is a proper embedding of the box~$\Lambda^\delta_R$ satisfying the condition~$\textsc{Unif}(\delta)$, i.\,e., all the lengths~$|\cS^\delta(v)-\cS^\delta(u)|$, $u\sim v$, of edges in~$\cS^\delta(\Lambda^\delta_R)$ are uniformly comparable to~$\delta$ and all the angles of quads in~$\cS^\delta(\Lambda^\delta_R)$ are uniformly bounded away from~$0$.
\end{proposition}

\begin{proof} Let~$u\sim v\in\Lambda^\delta_R$ and~$c\in\Upsilon(\Lambda^\delta)$ be adjacent to both~$u$ and~$v$. It follows from Definition~\ref{def:shol-def-mass} and a trivial estimate~$\widehat{\eta}_{c,z}=\eta_c+O(\delta)$ that the edge length
\begin{align*}
|\cS^\delta(v)-\cS^\delta(u)|\ &=\ |\cX^\delta(c)|^2\ =\ \big(|\cX_\rr^\delta(c)|^2+|\cX_\ri^\delta(c)|^2\big)\\
&=\ \delta\cdot\big[\big(\Pr\big[\cF_\rr^\delta(z);\eta_c\R\big]\big)^2+\big(\Pr\big[\cF_\ri^\delta(z);\eta_c\R\big]\big)^2+O(\delta)\big]
\end{align*}
is uniformly comparable to~$\delta$ provided that~$\delta\le\delta_0$ due to asymptotics of functions~$\cF_\rr^\delta$ and~$\cF_\ri^\delta$ given in Theorem~\ref{thm:F1Fi-asymp}. Similarly, if~$c$ and~$c'$ correspond to two adjacent edges of a quad~$z$ in~$\cS^\delta$, then it is easy to see that the quantity
\begin{align*}
\delta^{-1}\Im[\,\cX^\delta(c)\overline{\cX^\delta(c')}\,]\ &=\ \delta^{-1}\cdot \big[\cX_\rr^\delta(c)\cX_\ri^\delta(c')-\cX_\ri^\delta(c)\cX_\rr^\delta(c')\big]\\
&=\ |\cF_\rr^\delta(z)||\cF_\ri^\delta(z)|\cdot\Im[\,\eta_{c'}\overline{\eta}{}_c\,]+O(\delta)
\end{align*}
is uniformly bounded away from~$0$ and thus all the angles of quads in~$\cS^\delta(\Lambda^\delta_R)$ are uniformly bounded from below as required.

Thus, it remains to check that~$\cS^\delta$ is a proper embedding of~$\Lambda^\delta_R$. The computation given above also ensures that the increments~$\cS^\delta(v)-\cS^\delta(u)$ around a given vertex of~$\cS^\delta$ (or around a given quad) are cyclically ordered in the same way as the increments~$\eta_c^{-2}=-i\cdot (v-u)$ on the original isoradial grid~$\Lambda^\delta$. In other words, all quads in~$\cS^\delta$ are oriented in the same way as in~$\Lambda^\delta$ and quads surrounding a given vertex do not overlap with each other. Now note that this local property implies the discrete argument principle: the number of times that~$\cS^\delta$ covers a point in~$\C$ is equal the winding number of the image in~$\cS^\delta$ of a big contour surrounding this point in~$\Lambda^\delta$. This winding number is equal to~$1$ due to asymptotics~\eqref{eq:SembQ-asymp} provided that~$\delta$ is small enough, which completes the proof.
\end{proof}

Given an s-embedding~$\cS^\delta:\Lambda^\delta\to\C$ and a spinor~$X^\delta$ (locally) defined on $\Upsilon^\times(\Lambda^\delta)$ one can construct an \emph{s-holomorphic on~$\cS^\delta$} function~$F_{\cS^\delta}$ by requiring that
\begin{equation}
\label{eq:FS-def}
X^\delta(c)\ =\ \Re\big[\,\overline\varsigma\cX^\delta(c)\cdot F_{\cS^\delta}(z)\,]
\end{equation}
(see~\cite[Proposition~2.5]{chelkak2020ising}). The next proposition provides an explicit formula linking the function $F_{\cS^\delta}$ and the massive s-holomorphic function~$F^{(m),\delta}$ constructed from the same spinor~$X^\delta$ via Definition~\ref{def:shol-def-mass}.
\begin{proposition} \label{prop:Fm=FS} Let a spinor~$X^\delta$ locally satisfy the propagation equation~\eqref{eq:3-terms} on~$\Upsilon^{\times}(\Lambda^\delta)$, the massive s-holomorphic (on~$\Lambda^\delta$) function~$F^{(m),\delta}$ be defined according to~\eqref{eq:mass-hol}, and the function~$F_{\cS^\delta}$ be defined by~\eqref{eq:FS-def}. The following identity holds:
\begin{equation}
\label{eq:Fm=FS}
F^{(m),\delta}(z)\ = \cD^\delta_+(z)\cdot F_{\cS^\delta}(z)+\cD^\delta_-(z)\cdot \overline{F_{\cS^\delta}(z)}\,,
\end{equation}
where the coefficients are given by~$\cD^\delta_\pm(z)=\tfrac{1}{2}(\cF_\rr^\delta(z)\mp i\cF_\ri^\delta(z))$.
\end{proposition}

\begin{proof} For shortness, denote $\nu_c:=\widehat{\eta}_{z,c}$. The formulas~\eqref{eq:mass-hol} and~\eqref{eq:FS-def} imply that
\[
\delta^{\frac12}\cdot \big[\,\overline{\nu}_cF^{(m),\delta}(z)+\nu_c\overline{F^{(m),\delta}(z)}\,\big]\ =\ 2X(c)\ =\ \overline{\varsigma}\cX^\delta(c)F_{\cS^\delta}(z)+\varsigma\overline{\cX^\delta(c)}\overline{F_{\cS^\delta}(z)}.
\]
for all~$c\sim z$; note that these equations uniquely define the value~$F^{(m),\delta}$. Therefore, the formula~\eqref{eq:Fm=FS} is equivalent to the identity
\[
\delta^{\frac12}\cdot\big[\,\overline{\nu}_c\cD^\delta_+(z)+\nu_c\overline{\cD^\delta_-(z)}\,\big]= \overline{\varsigma}\cX^\delta(c)=\cX_\rr^\delta(c)-i\cX_\ri^\delta(c),\quad c\sim z.
\]
It remains to note that
$\cX_\rr^\delta(c)=\frac12\delta^{\frac12}\cdot[\,\overline{\nu}_c\cF_\rr^\delta(z)+\nu_c\overline{\cF_\rr^\delta(z)}\,]$
and similarly for~$\cX_\ri^\delta(c)$, again due to Definition~\ref{def:shol-def-mass}.
\end{proof}

\section{Proofs of the main results}
\label{sec: Proofs}
\setcounter{equation}{0}

In this section we prove the main results of our paper. In particular, Theorem~\ref{thm:intro-mass}, i.\,e., the convergence of the re-scaled infinite-volume correlations~$\delta^{-\frac14}\E_{\Lambda^\delta}^{(m)}[\sigma_u\sigma_w]$ to a universal (i.\,e., independent of~$\Lambda^\delta$) rotationally invariant limit~$\mathcal{C}_\sigma^2\cdot \Xi(|u-w|,m)$ is proven in Theorem~\ref{thm:spin-univ-crit} for~$m=0$ and in Corollary~\ref{cor:spin-univ-mass} for~$m\ne 0$. Along the way, we also prove convergence of spin-spin correlations in discrete approximations \mbox{$\Omega^\delta\subset\Gamma^\delta$} of finite $C^1$-smooth domains~$\Omega\subset\C$: see Corollary~\ref{cor:conv-omega-crit} and Theorem~\ref{thm:conv-omega-mass} for the critical and massive cases, respectively.

We start our exposition by giving a proof of Baxter's formula (Proposition~\ref{prop:intro-sub}) for the magnetization in the infinite-volume sub-critical model on isoradial graphs; note that already this proof contains two important ideas that we also use later: gluing isoradial grids to each other via a procedure discussed in Section~\ref{sub:star-ext} and the reconstruction of values of spinor observables near their (non-)branching points via the explicit kernels from Section~\ref{sub:kernels-def}. Then, we analyze the spin-spin correlations in the critical model in Sections~\ref{sub:spin-obs-crit}--\ref{sub:univ-crit}. Let us repeat that this analysis considerably simplifies the arguments used, e.\,g., in~\cite{ChelkakHonglerIzyurov}. The massive model is discussed in Sections~\ref{sub:spin-mass-def},~\ref{sub:conv-mass} basing upon a similar strategy.

\subsection{Baxter's formula for the magnetization in the sub-critical model} \label{sub:spin-sub}
Throughout this section, $\delta=1$ and $ q=\frac{1}{2}m<0$ are fixed.
\begin{lemma}
\label{lem:E[sigma]=cst}
Let $m<0$ and~$\Lambda$ be an infinite isoradial grid of mesh $\delta=1$ satisfying the property~$\BAP$. Then, the magnetization~$\E^+_\Lambda[\sigma_u]$ does not depend on $u\in\Gamma^\circ$.
\end{lemma}

\begin{proof} Let $z=(v_0u_0v_1u_1)$ be a rhombus on $\Gamma$; see Fig.~\ref{fig:notation} for the notation. Consider the pointwise limit
\[
X_{[v_0]}(c)\ :=\ \lim\nolimits_{ R\to\infty}\E^\wired_{\Lambda_R}[\chi_c\mu_{v_0}\sigma_{u_\mathrm{out}}],\quad c\in\Upsilon^\times_{[v_0]},
\]
which can also be thought of as the infinite-volume correlator $\E_\Lambda[\chi_c\mu_{v_0}\sigma_\infty]$. The existence of a limit follows from the fact that for { each} given~$c$, one can fix a disorder line $\gamma$ with $\pa \gamma=\{v(c),v_0\}$ and write
\[
\E^\wired_{\Lambda_R}[\chi_c\mu_{v_0}\sigma_{u_\mathrm{out}}] \textstyle\ =\ \E^\wired_{\Lambda_R}\left[\sigma _{u(c)}\sigma_{u_\mathrm{out}}\prod_{u\sim w:(uw)\cap \gamma\neq \emptyset}e^{-2\beta^\circ J^\circ_{(uw)^*}\sigma_{u}\sigma_{w}}\right],
\]
{ which is a finite linear combination of multipoint spin expectations since
\[
e^{-2\beta^\circ J^\circ_{(uw)^*}\sigma_{u}\sigma_{w}}\ =\ \cosh(-2\beta^\circ J^\circ_{(uw)^*})+\sigma_u\sigma_w \sinh(-2\beta^\circ J^\circ_{(uw)^*})\,.
\]
Each of these expectations is decreasing as~$R\to\infty$ due to the FKG inequality, which guarantees the convergence of $\E^\wired_{\Lambda_R}[\chi_c\mu_{v_0}\sigma_{u_\mathrm{out}}]$.}

\begin{figure}
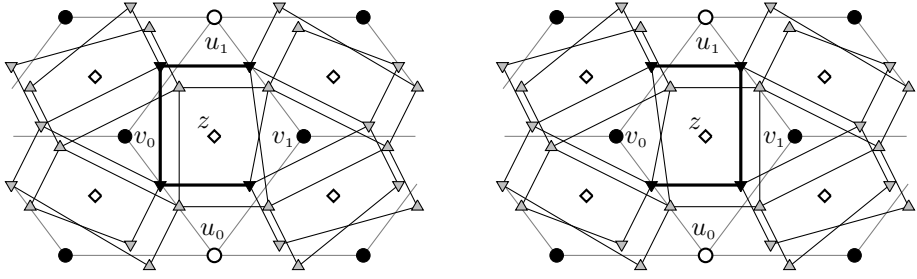

\begin{center}\begin{tikzpicture}[scale=0.18]
\input{BranchingV0.txt}
\end{tikzpicture}
\hskip 24pt
\begin{tikzpicture}[scale=0.18]
\input{BranchingV1.txt}
\end{tikzpicture}\end{center}
\caption{{ If~$v_0,v_1\in\Gamma^\bullet$ be two vertices of a quad~$z=(v_0u_0v_1u_1)$, the double covers~$\Upsilon^\times_{[v_0]}$ (shown on the left) and~$\Upsilon^\times_{[v_1]}$ (shown on the right) can be identified with each other except at lifts of corners~$c_{pq}=(v_pu_q)$,~$p=0,1$, $q=0,1$, surrounding~$z$. In the proof of Lemma~\ref{lem:E[sigma]=cst} we choose these lifts so that the incidence relations~\eqref{eq:Uv0=Uv1-choice-of-c} hold; these relations are shown by thick black lines in the figure.}}
\label{fig:Uv0=Uv1}
\end{figure}

Note that
\[
X_{[v_0]}(\ccc{00})\ =\ \lim\nolimits_{ R\to\infty}\E^\wired_{\Lambda_R}[\sigma_{v_0}\sigma_{u_\mathrm{out}}]\ =\
\E^+_\Lambda[\sigma_{u_0}]\quad \text{and}\quad X_{[v_0]}(\ccc{01})\ =\ \E^+_\Lambda[\sigma_{u_1}].
\]
Let us now consider the function~$H:=H[X_{[v_0]},\rG_{[v_1]}]$, where $\rG_{[v_1]}$ is the explicit infinite-volume kernel discussed in Theorem~\ref{thm:G-asymp-mass} and { Remark~\ref{rem:G-asymp}(iii).} Since the two spinors~$X_{[v_0]}$ and~$\rG_{[v_1]}$ are defined on slightly \emph{different} double-covers $\Upsilon^\times_{[v_0]}$ and~$\Upsilon^\times_{[v_1]}$, the function~$H$ has an additive monodromy~$Z$ around the quad~$(v_0u_0v_1u_1)$. Similarly to the proof of Lemma~\ref{lem:XG-monodromy}, it is easy to see that
\begin{align*}
\tfrac{1}{2}Z\ &=\ -X_{[v_0]}(\ccc{00})\rG_{[v_1]}(\ccc{00})+X_{[v_0]}(\ccc{10})\rG_{[v_1]}(\ccc{10})\\
& =\ -X_{[v_0]}(\ccc{11})\rG_{[v_1]}(\ccc{11})+X_{[v_0]}(\ccc{01})\rG_{[v_1]}(\ccc{01}),
\end{align*}
where we assume that the double covers $\Upsilon^\times_{[v_0]}$ and $\Upsilon^\times_{[v_1]}$ are identified with each other away from the quad $(v_0u_0v_1u_1)$ and that { (see Fig.~\ref{fig:Uv0=Uv1})}
\begin{equation}
\label{eq:Uv0=Uv1-choice-of-c}
\ccc{11}\sim\ccc{01}\sim \ccc{00}\sim \ccc{01}\ \text{on}\ \Upsilon^\times_{[v_0]}\quad \text{whilst}\quad
\ccc{00}\sim\ccc{10}\sim\ccc{11}\sim\ccc{01}\ \text{on}\ \Upsilon^\times_{[v_1]}.
\end{equation}
On the other hand, the uniform boundedness of~$X_{[v_0]}(c)$ and the exponential decay of $\rG_{[v_1]}$ (as $|c-v_1|\to\infty$; see { Remark~\ref{rem:G-asymp}(iii)} and Section~\ref{sub:exp-asymp}) imply that $Z=0$ { as we can move the integration contour in~\eqref{eq:XG-monodromy} far away from the branching. Thus, }
\[
X_{[v_0]}(\ccc{00})\rG_{[v_1]}(\ccc{00})\ =\ X_{[v_0]}(\ccc{10})\rG_{[v_1]}(\ccc{10}).
\]
Using the propagation equation~\eqref{eq:3-terms} to write the value~$\rG_{[v_1]}(\ccc{00})$ as a linear combination of~$\rG_{[v_1]}(\ccc{10})$ and~$\rG_{[v_1]}(\ccc{11})$, and the value~$X_{[v_0]}(\ccc{10})$ as a linear combination of~$X_{[v_0]}(\ccc{00})$ and~$X_{[v_0]}(\ccc{01})$, one can rewrite the last identity as
\[
X_{[v_0]}(\ccc{00})\rG_{[v_1]}(\ccc{11})\ =\ X_{[v_0]}(\ccc{01})\rG_{[v_1]}(\ccc{10}).
\]
Recall that $\rG_{[v_1]}(\ccc{11})=\rG_{[v_1]}(\ccc{10})$ due to the explicit construction of the kernel~$\rG_{[v_1]}$. This implies the desired identity~$\E^+_\Lambda[\sigma_{u_0}]=X_{[v_0]}(\ccc{00})=X_{[v_0]}(\ccc{01})=\E^+_\Lambda[\sigma_{u_1}]$.
\end{proof}

\begin{proposition} \label{prop:E[sigma-sigma]=cst}
We have $\mathbb{E}_\Lambda[\sigma_u\sigma_w]\to \sccorr$ as $|u-w|\to\infty$ uniformly over isoradial lattices~$\Lambda$ and positions of points $u,w\in\Gamma^\circ$.
\end{proposition}

\begin{proof} Denote by $D:=|u-w|$ the (Euclidean) distance between $u$ and~$w$. The uniform RSW estimates~\eqref{eq:RSW-sigma-sigma} imply that for each $\varepsilon>0$ one can find~$A=A(\varepsilon)>1$ such that
\[
\E^\wired_{\Lambda_{AD}(u)}[\sigma_{u}\sigma_{w}]-\varepsilon\ \le\ \E_\Lambda[\sigma_u\sigma_{w}]\ \le\ \mathbb \E^\wired_{\Lambda_{AD}(u)}[\sigma_{u}\sigma_{w}]
\]
on each isoradial grid containing the box~$\Lambda_{AD}(u)$ (and satisfying~$\BAP$).

Let us now replace~$\Lambda$ by the `star extension' $[\Lambda_{AD}(u)]^\star$ of the box~$\Lambda_{AD}(u)$ constructed in Section~\ref{sub:star-ext}. For shortness, below we use the notation $\Lambda^\star:=[\Lambda_{AD}(u)]^\star$; note that the { the same reasoning applied to~$\Lambda^\star$ instead of~$\Lambda$} implies that
\begin{equation}
\label{eq:change-to-glued}
\big|\,\E_\Lambda[\sigma_u\sigma_{w}]-\E_{\Lambda^\star}[\sigma_u\sigma_{w}]\,\big|\ \le\ \varepsilon.
\end{equation}
According to the construction, on the modified lattice~$\Lambda^\star$ we can find a `white' vertex~$u'$ such that~$|u'-u|\le CAD$ and that the box~$\Lambda^\star_{AD}(u')$ is a piece of a rectangular grid; the constant~$C$ depends on~$\theta_0$ only. Denote by~$w'$ a `white' vertex of~$\Lambda^\star$ lying at distance~$D$ from~$u'$ and such that $w'-u'$ is one of the axial directions of this rectangular grid.

We claim that, { for each grid~$\Lambda^\star$ satisfying the bounded angles property~$\BAP$,}
\begin{equation}
\label{eq:Euu-Eww=O(exp)}
\big|\,\E_{\Lambda^\star}[\sigma_u\sigma_w]-\E_{\Lambda^\star}[\sigma_{u'}\sigma_{w'}]\,\big|\ =\ O(A(\varepsilon)D^2\exp(-\beta D)),
\end{equation}
where a universal constant~$\beta=\beta(q,\theta_0)>0$ comes from the uniform exponential decay of the branching kernels~$\rG_{[v]}$; see~\eqref{eq:G-exp-decay-sub} and Section~\ref{sub:exp-asymp}. The proof of~\eqref{eq:Euu-Eww=O(exp)} goes along the same lines as the proof of Lemma~\ref{lem:E[sigma]=cst}. Namely, let~$z=(v_0u_0v_1u_1)$ be a quad lying at distance~$D$ from the second `white' vertex~$w$ under consideration. As in the proof of Lemma~\ref{lem:E[sigma]=cst}, let us { consider} a pointwise limit
\[
X_{[v_0,w]}(c)\ :=\ \lim\nolimits_{ R\to\infty}\E^\wired_{\Lambda^\star_R}[\chi_c\mu_{v_0}\sigma_{w}],\quad c\in\Upsilon^\times_{[v_0,w]},
\]
and a function~$H:=H[X_{[v_0,w]},\rG_{[v_1]}]$ defined in the vicinity of~$z$ of radius~$D$. (Note that now the spinor~$X_{[v_0,w]}$ has the second branching at~$w$ while~$\rG_{[v_1]}$ does not branch there, this is why~$H$ is not defined on the whole grid~$\Lambda^\star$.) Let~$Z$ be the additive monodromy of the function~$H$ around~$z$. The uniform boundedness of the fermionic observable~$X_{[v_0,w]}$ and the uniform exponential decay of the kernel~$\rG_{[v_1]}$ imply that~$Z=O(D \exp(-\beta D))$ since we can compute this monodromy along a contour running at distance~$D$ from~$z$. Then, computing the monodromy~$Z$ similarly to the proof of Lemma~\ref{lem:E[sigma]=cst} we obtain the estimate
\begin{equation}
\label{eq:sub-step-on-glued}
\big|\,\E_{\Lambda^\star}[\sigma_{u_0}\sigma_{w}]-\E_{\Lambda^\star}[\sigma_{u_1}\sigma_{w}]\,\big|\ =\ O(D\exp(-\beta D)).
\end{equation}
Moving a pair of points $u,w$ step by step to $u',w'$ so that they remain at distance (at least)~$D$ from each other we obtain the desired estimate~\eqref{eq:Euu-Eww=O(exp)}.

Finally, let~$\Lambda^\mathrm{rect}$ be a rectangular grid that extends the rectangular part~$\Lambda^\star_{AD}(u')$ of~$\Lambda^\star$. Similarly to~\eqref{eq:change-to-glued}, one has~$\big|\,\E_{\Lambda^\star}[\sigma_{u'}\sigma_{w'}]-\E_{\Lambda^\mathrm{rect}}[\sigma_{u'}\sigma_{w'}]\,|\le \varepsilon$.
Combining~\eqref{eq:change-to-glued},~\eqref{eq:Euu-Eww=O(exp)} and this estimate, we see that
\begin{equation}
\label{eq:sub-iso-to-rect}
\big|\,\E_\Lambda[\sigma_u\sigma_w]-\E_{\Lambda^\mathrm{rect}}[\sigma_{u'}\sigma_{w'}]\,\big|\ \le\ 2\varepsilon + O(A(\varepsilon)D^2\exp(-\beta D)).
\end{equation}
We can now rely upon explicit computations of the `horizontal' spin-spin correlations~$\E_{\Lambda^\mathrm{rect}}[\sigma_{u'}\sigma_{w'}]$ on rectangular grids; see~\cite[Section~X.4]{mccoy2014two} or~\cite[Theorem~3.6]{chelkak_Hongler_Mahfouf}. In particular, it is well known (see also Remark~\ref{rem:reduction-to-square} below) that
\begin{equation}
\label{eq:sub-rect-lim}
\E_{\Lambda^\mathrm{rect}}[\sigma_{u'}\sigma_{w'}]\ \to\ \sccorr\ \ \text{as}\ \ |u'-w'|\to\infty.
\end{equation}
Choosing first~$\varepsilon$ small enough and then~$D$ big enough in~\eqref{eq:sub-iso-to-rect} allows us to conclude that~$\E_\Lambda[\sigma_u\sigma_w]\to \sccorr$ as~$|u-w|\to\infty$, uniformly with respect to isoradial grids~$\Lambda$ satisfying the bounded angles condition~$\BAP$.
\end{proof}
\begin{remark}
\label{rem:reduction-to-square} A careful reader could notice that above we used the -- a priori nontrivial -- fact that the limit in~\eqref{eq:sub-rect-lim} does not depend on the (unknown) aspect ratio of the rectangular lattice~$\Lambda^\mathrm{rect}$. This can be avoided by making another comparison of correlations on~$\Lambda^\mathrm{rect}$ and those on the \emph{square lattice}~$\mathbb Z^2$; recall that one can easily glue~$\Lambda^\mathrm{rect}$ and~$\mathbb Z^2$ staying inside the family of isoradial grids (see Fig.~\ref{fig:star-extension}). The estimate~\eqref{eq:sub-iso-to-rect} and a similar estimate for~$\Lambda^\mathrm{rect}$ and~$\mathbb{Z}^2$ imply that
\[
\big|\,\E_\Lambda[\sigma_u\sigma_w]-\E_{\mathbb Z^2}[\sigma_{u'}\sigma_{w'}]\,\big|\ \le\ 4\varepsilon + O(A(\varepsilon)D^2\exp(-\beta D)).
\]
In particular, this gives a self-contained proof of the existence of the universal (among isoradial grids) limit~$\lim_{|u-w|\to\infty}\E_\Lambda[\sigma_u\sigma_w]$. Moreover, on the square lattice one can work with `diagonal' spin-spin correlations instead of `horizontal' ones, which considerably simplifies the computation of the limit; e.\,g., see~\cite[Section~3]{Chelkak2016ECM}.
\end{remark}

\begin{corollary} For all isoradial grids~$\Lambda$ satisfying the condition~$\BAP$ and all~$u\in \Gamma^\circ$ the magnetization $\E^+_\Lambda[\sigma_u]$ is equal to~$\magnet$
\end{corollary}
\begin{proof} It follows from Lemma~\ref{lem:E[sigma]=cst} that, given~$\Lambda$, the magnetization~$\E_\Lambda^+[\sigma_u]$ does not depend on the position of~$u\in\Gamma^\circ$; let us denote this common value by~$M(\Lambda)$.

First, note that Proposition~\ref{prop:E[sigma-sigma]=cst} and the FKG inequality imply that
\[
\sccorr\ =\ \lim_{R\to\infty} \E^\wired_{\Lambda_R}[\sigma_u\sigma_{u'}]\
  \ge\ \lim_{R\to\infty} \E^+_{\Lambda_R}[\sigma_u]\cdot \E^+_{\Lambda_R}[\sigma_{u'}]\ =\ (M(\Lambda))^2
\]
and hence~$M(\Lambda)\le \magnet$ for all~$\Lambda$. To prove the inverse inequality note that, for large enough~$R$, the FKG inequality also implies that
\[\label{eq: corr_magn_magn}
\E^\wired_{\Lambda_R}[\sigma_u\sigma_{u'}]\ \le\ \E^+_{\Lambda_D(u)}[\sigma_u]\cdot \E^+_{\Lambda_{D'}(u')}[\sigma_{u'}]\quad \text{if}\ \ \Lambda_D(u)\cap \Lambda_{D'}(u')=\emptyset.
\]

Let us first consider the particular case when~$\Lambda=\Lambda^\mathrm{rect}$ is a rectangular lattice. In this case, for each~$\varepsilon>0$ one can choose~$D=D'$ big enough so that
\[
\E^+_{\Lambda_D(u)}[\sigma_u]\ =\ \E^+_{\Lambda_D(u')}[\sigma_{u'}]\ \le\ M(\Lambda^\mathrm{rect})+\varepsilon;
\]
note that we used the translation invariance of~$\Lambda^\mathrm{rect}$. By first letting~$R\to\infty$ in \eqref{eq: corr_magn_magn} and then~$\varepsilon\to 0$ one concludes that~$M(\Lambda^\mathrm{rect})=\magnet$.

Let us now consider a general case. As above, given~$\varepsilon>0$ and a vertex $u\in\Gamma^\circ$ near~$0$, one can choose~$D=D(\varepsilon)$ large enough so that
\[
\E^+_{\Lambda_D(u)}[\sigma_u])\ \le\ M(\Lambda)+\varepsilon.
\]
Let~$\Lambda^\star:=[\Lambda_D(u)]^\star$ be the star extension of~$\Lambda_D(u)$. By construction (see Fig.~\ref{fig:star-extension}), $\Lambda^\star$ contains arbitrary large pieces of a certain rectangular grid~$\Lambda^\mathrm{rect}$ (whose aspect ratio can depend on~$\Lambda^\star$, which in its turn depend on~$\varepsilon$). One can now choose $D'=D'(\varepsilon)$ large enough so that
\[
\E^+_{\Lambda^\star_{D'}(u')}[\sigma_{u'}]\ \le\ M(\Lambda^\mathrm{rect})+\varepsilon\ =\ \magnet+\varepsilon
\]
for a certain `white' vertex $u'$ of $\Lambda^\star$. Moreover, once~$D'$ is chosen, the vertex $u'$ can be taken arbitrarily far from~$u$. Passing to the limit $R\to\infty$, then $|u-u'|\to \infty$ and finally~$\varepsilon\to 0$ in the inequality
\[
\E^\wired_{\Lambda_R}[\sigma_u\sigma_{u'}]\ \le\ (M(\Lambda)+\varepsilon)(\magnet+\varepsilon)
\]
we see that~$M(\Lambda) \ge \magnet$, which completes the proof.
\end{proof}

\subsection{Critical model: convergence of normalized observables $\bm{\E^\wired_{\Omega^\delta}[\chi_c\mu_v\sigma_w]}$} \label{sub:spin-obs-crit}
In this section we consider the critical Ising model in bounded discrete approximations~$\Omega^\delta\subset\Gamma^\delta$ of a bounded simply connected domain~$\Omega\subset\C$ with a $C^1$-smooth boundary.
The convergence results discussed below can be viewed as generalizations of similar results obtained in~\cite{ChelkakHonglerIzyurov} for the square grid~$\Lambda^\delta=\delta\mathbb Z^2$; under an additional smoothness assumption on~$\partial\Omega$. It is worth noting that we use this smoothness assumption only in order to give a simple proof of Proposition~\ref{prop:bc-are-OK} following the paper~\cite{park2018massive}, and that it can be removed by using techniques from~\cite{ChelkakSmirnov2} or from~\cite{CHI_Mixed} instead; see Remark~\ref{rem:CHI-is-OK} below. Let us emphasize that the proofs given in this section do \emph{not} simply mimic those from~\cite{ChelkakHonglerIzyurov}; on the contrary, we considerably improve the strategy used in~\cite{ChelkakHonglerIzyurov} even in the case~$\Lambda^\delta=\delta\mathbb Z^2$ in what concerns the analysis of spinor observables near the branching points. In particular, we do \emph{not} rely upon  explicit \mbox{`$z^{1/2}$-type'} kernels (see~\cite[Lemma~2.17]{ChelkakHonglerIzyurov}).

Let~$\Omega\subset\C$ be a bounded simply connected domain and $\Omega^\delta$ be discrete approximation of~$\Omega$ on isoradial grids~$\Lambda^\delta$ with $\delta\to 0$. (As always in our paper we also assume that~$\Lambda^\delta$ satisfy the bounded angles property~$\BAP$.) For simplicity, we also assume that the boundary~$\partial\Omega$ of~$\Omega$ is $C^1$-smooth and that~$\Omega^\delta$ approximate~$\Omega$ in the Hausdorff sense so that~$\mathrm{dist}(\partial\Omega^\delta;\partial\Omega)=O(\delta)$. However, it is worth noting that one can drop these regularity assumption by repeating the arguments developed in~\cite{ChelkakSmirnov2} in what concerns the near-to-the-boundary analysis of fermionic observables in rough domains~$\Omega$ under the Carath\'eodory convergence~$\Omega^\delta\to\Omega$; see also Remark~\ref{rem:CHI-is-OK} below.

Let~$v,w$ be distinct inner points of~$\Omega$; for the sake of shortness we use the same notation for discrete approximations~$v=v^\delta\in \Gamma^{\bullet,\delta}$ and $w=w^\delta\in\Gamma^{\circ,\delta}$ of these points. Also, let~$u=u^\delta\in\Gamma^{\circ,\delta}$ be one of the neighboring `white' vertices of~$v$. Following~\cite{ChelkakHonglerIzyurov}, we consider the normalized real-valued fermionic observable
\begin{equation}
\label{eq:Xvw-def}
X^\delta_{[\Omega^\delta;v,w]}(c)\ :=\ \frac{\E^\wired_{\Omega^\delta}[\chi_{c}\mu_{v}\sigma_{w}]}{\E^\wired_{\Omega^\delta}[\sigma_{u}\sigma_{w}]}, \quad c\in\Upsilon^\times_{[v,w]}(\Omega^\delta),
\end{equation}
and denote by~$F^\delta_{[\Omega^\delta;v,w]}$ the complex-valued observables { constructed from~$X^\delta_{[\Omega^\delta;v,w]}$} according to~\eqref{eq:mass-hol}; recall that~$F^\delta_{[\Omega^\delta;v,w]}$ is an s-holomorphic spinor in~$\Omega^\delta$ branching over~$v$ and~$w$. Also,
\begin{equation}
\label{eq:X(cuv)=1}
X^\delta_{[\Omega^\delta;v,w]}(c_{uv})\ =\ 1,
\end{equation}
where the corner~$c_{uv}$ is adjacent to both~$u$ and~$v$.

In the forthcoming Theorem~\ref{thm:conv-Fwv-crit} (which can be viewed as a generalization to isoradial grids of a particular case~$n=2$ of~\cite[Theorem~2.16]{ChelkakHonglerIzyurov}; see also Remark~\ref{rem:CHI-is-OK}) we prove the convergence of these discrete observables as~$\delta\to 0$. The limit is a holomorphic spinor~$f_{[\Omega;v,w]}$ on the double cover of $\Omega$ ramified over~$v$ and~$w$ that is uniquely defined by the following conditions:
\begin{itemize}
\item $f_{[\Omega;v,w]}$ is continuous in~$\overline{\Omega}\smallsetminus\{v,w\}$ and satisfies the Riemann-type boundary conditions $\Im[f_{[\Omega;v,w]}(\zeta)(\tau(\zeta))^{1/2}]=0$ for all~$\zeta\in\partial\Omega$, where~$\tau(\zeta)$ denotes the tangent vector to~$\partial\Omega$ at the point~$\zeta$ oriented counterclockwise;
\smallskip
\item the following asymptotics holds:
\begin{equation}
\label{eq:f-asymp-near-v}
f_{[\Omega;v,w]}(z)\ =\ e^{-i\frac{\pi}{4}}(z-v)^{-\frac12}+O(|z-v|^{\frac12})\ \ \text{as}\ \ z\to v;
\end{equation}
and there exists a (a priori unknown) constant~$\cB_\Omega(v,w)\in\R$ such that
\begin{equation}
\label{eq:f-asymp-near-w}
f_{[\Omega;v,w]}(z)\ =\ e^{i\frac{\pi}{4}}\cB_\Omega(v,w)\cdot (z-w)^{-\frac12}+O(|z-w|^{\frac12})\ \ \text{as}\ \ z\to w.
\end{equation}
\end{itemize}

\begin{remark} \label{rem:f-uniq}
(i) Since~$f_{[\Omega;v,w]}$ branches over the point~$w$, the coefficient~$\cB_\Omega(v,w)$ is a priori defined only up to the sign. We fix this sign by requiring that~$\cB_\Omega(v,w)\ge 0$.

\smallskip

\noindent (ii) The uniqueness of the solution to this boundary value problem easily follows from considering the harmonic function~$h:=\int \Im[(f_1(z)-f_2(z))^2dz]$: if~$f_1$ and~$f_2$ were two distinct solutions, then the function~$h$ would \emph{not} have a singularity at the point \mbox{$z=v$}, would behave like~$b\log|z-w|+O(1)$ with~$b\ge 0$ as~$z\to w$, and would have negative outer normal derivative along~$\partial\Omega$, the sign which contradicts to the Green formula; see also~\cite[Section~2.5]{ChelkakHonglerIzyurov}.

\smallskip

\noindent (iii) If~$\phi:\Omega\to\Omega'$ is a conformal map, then it is easy to see that
\begin{equation}
\label{eq:fvw-conf-cov}
f_{[\Omega;v,w]}(z)=f_{[\Omega';\phi(v),\phi(v)]}(\phi(z))\cdot(\phi'(z))^{\frac 12}.
\end{equation}
The solution in the upper-half plane~$\Omega=\mathbb H$ can be written explicitly (see~\cite[Section~2.7]{ChelkakHonglerIzyurov}); in particular, this can be used to justify the existence of~$f_{[\Omega;v,w]}$. Alternatively, one can get the existence of~$f_{[\Omega;v,w]}$ directly from Theorem~\ref{thm:conv-Fwv-crit}, not relying upon explicit formulas and/or the conformal covariance property~\eqref{eq:fvw-conf-cov}.
\end{remark}

Let us denote by~$H^\delta$ the function constructed from the fermionic observable~\eqref{eq:Xvw-def} via~\eqref{eq:HX-def}, where the additive constant in its definition is chosen so that~$H^\delta$ satisfies the Dirichlet boundary conditions. Fix a small enough~$r_0>0$ and denote
\begin{equation}
\label{eq:MdOd-def}
M^\delta\ :=\ \max\nolimits_{\Omega(r_0)} |H^\delta|,\ \ \text{where}\ \ \Omega(r_0):=\Omega\smallsetminus (B(v,r_0)\cup B(w,r_0)).
\end{equation}
The proof of the following estimate considerably simplifies the strategy used in~\cite{ChelkakHonglerIzyurov}.
\begin{proposition} \label{prop:F=O(1)}
Assume that the estimate~$M^\delta=O(1)$ holds as~$\delta\to 0$. Then,
\begin{equation}
\label{eq:F=O(1)}
F^\delta_{[\Omega^\delta;v,w]}=O(1)\ \ \text{as}\ \ \delta\to 0\ \ \text{uniformly on compact subsets of}~\Omega\smallsetminus\{v,w\}.
\end{equation}
\end{proposition}
\begin{proof} The a priori regularity estimates of s-holomorphic functions via the associated functions $H_F$ (see Proposition~\ref{prop:regularity}), in particular, imply that
\[
F^\delta_{[\Omega^\delta;v,w]}=O(1)\ \ \text{as}\ \ \delta\to 0\ \ \text{uniformly~on~compact~subsets~of}\ \Omega(r_0).
\]
Thus, to prove~\eqref{eq:F=O(1)} we need to control the behavior of~$F^\delta$ near~$v$ and~$w$.

Consider now the explicit full-plane kernel~$\rG^\delta_{[v]}$ discussed in Section~\ref{sub:kernels-def} and a spinor (defined, e.\,g., in a $3r_0$-vicinity of~$v$)
\begin{equation}
\label{eq:X-dag-def}
X^{\delta,\dagger}(\cdot)\ :=\ X^\delta_{[\Omega^\delta;v,w]}(\cdot)-\rG^\delta_{[v]}(\cdot)
\end{equation}
and let~$F^{\delta,\dagger}$ and~$H^{\delta,\dagger}$ be constructed from~$X^{\delta,\dagger}$ via~\eqref{eq:mass-hol} and~\eqref{eq:HX-def}, respectively. Due to~\eqref{eq:X(cuv)=1} and~\eqref{eq:G(c)=pi/2}, we have~$X^{\delta,\dagger}(c_{uv})=0$ and hence the function~$H^{\delta,\dagger}$ satisfies \emph{both} the maximum and the minimum principle near~$v$ (see Remark~\ref{rem:maxH-principle}). Also, we know that~$F^{\delta,\dagger}=O(1)$ and hence~$H^{\delta,\dagger}=O(1)$ near the circle \mbox{$\{z:|z-w|=3r_0\}$}, { for an appropriate choice of the additive constants in the definition of~$H^{\delta,\dagger}$.} Therefore,
\[
H^{\delta,\dagger}=O(1)\ \ \text{as}\ \ \delta\to 0\ \ \text{uniformly in the \emph{disc}}~B(v,3r_0):=\{z:|z-v|< 2r_0\}.
\]
It follows from Proposition~\ref{prop:regularity} that
\[
F^{\delta,\dagger}(z)=O(|z-v|^{-\frac{1}{2}})\ \ \text{and so}\ \ F^\delta_{[\Omega^\delta,v,w]}(z)=O(|z-v|^{-\frac12})\ \ \text{for}\ z\in B(v,2r_0)
\]
due to the explicit asymptotics of~$\cG^\delta_{[v]}$ (see Theorem~\ref{thm:G-asymp-mass}). In particular, $F^\delta$ remain uniformly bounded as~$\delta\to 0$ on compact subsets of~$\Omega\smallsetminus (\{v\}\cup B(w,r_0))$.

A similar though slightly more involved argument can be applied near the second branching point~$w\in \Gamma^{\circ,\delta}$; a complication is caused by the fact that now we do not have a prescribed value similar to~\eqref{eq:X(cuv)=1} near~$w$. However, it is not hard to see that these values remain uniformly bounded as~$\delta\to 0$. Indeed, denote by $c_w$ one of the corners adjacent to~$w$ and consider the function~$H\big[X^\delta_{[\Omega^\delta;v,w]}\,,\rG^\delta_{[v(c_w)]}\big]$. By Lemma~\ref{lem:XG-monodromy}, this function has the additive monodromy
\begin{align*}
2X^\delta_{[\Omega^\delta;v,w]}(c_w)\rG^\delta_{[v(c_w)]}(c_w)\ &=\ { 2X^\delta_{[\Omega^\delta;v,w]}(c_w)}\\
&=\ \frac{1}{2}\oint^{[\delta]}_{z:|z-w|=2r_0}\Im\big[ F^\delta_{[\Omega^\delta;v,w]}(z)\cG^\delta_{[v(c_w)]}dz\big].
\end{align*}
Since the (complex-valued) kernels~$\cG^\delta_{[v(c_w)]}$ remain uniformly bounded at a definite distance from~$w$ as~$\delta\to 0$ (see Theorem~\ref{thm:G-asymp-mass}), we obtain the uniform estimate
\begin{equation}
\label{eq:X(cw)=O(1)}
X^\delta_{[\Omega^\delta;v,w]}(c_w)\ =\ O(1)\ \ \text{as}\ \ \delta\to 0\,.
\end{equation}
We can now repeat the arguments given above considering the spinor
\[
X^\delta_{[\Omega^\delta;v,w]}(\cdot)-X^\delta_{[\Omega^\delta;v,w]}(c_w)\cdot \rG^\delta_{[w]}(\cdot)
\]
near the point~$w$ to prove that $F^\delta_{[\Omega^\delta;v,w]}(z)=O(|z-w|^{-\frac12})$ for~\mbox{$z\in B(w,2r_0)$.}
\end{proof}
Let us for a while take for granted the assumption $M^\delta=O(1)$ made in Proposition~\ref{prop:F=O(1)}. It follows from the estimate~\eqref{eq:F=O(1)} and from the a priori regularity of s-holomorphic functions (see Proposition~\ref{prop:regularity}) that the functions $F^\delta_{[\Omega^\delta;v,w]}$ are also equicontinuous on compact subsets of~$\Omega\smallsetminus\{v,w\}$. Thus, one can apply the Arzel\`a--Ascoli theorem and find a subsequential limit
\begin{equation}
\label{eq:F-subseq-lim}
F^\delta_{[\Omega^\delta;v,w]}(z)\ \to\ g_{[\Omega;v,w]}(z)\ \ \text{as}\ \ \delta=\delta_k\to 0,
\end{equation}
where the convergence is uniform on compact subsets of~$\Omega\smallsetminus\{v,w\}$. Trivially, each such a subsequential limit~$g_{[\Omega;v,w]}$ is a spinor branching over~$v$ and~$w$. Moreover,~$g_{[\Omega;v,w]}$ is holomorphic due to Corollary~\ref{cor:mass-hol}. Let
\begin{equation}
\label{eq:h-subseq}
\textstyle h\ :=\ \frac{1}{2}\int\Im[(g_{[\Omega;v,w]}(z))^2dz];
\end{equation}
note that~$h$ is a harmonic function in the punctured domain~$\Omega\smallsetminus\{v,w\}$ defined up to an additive constant.

As functions~$H^\delta$ satisfy the Dirichlet boundary conditions (see Remarks~\ref{rem:H-Dirichlet-bc} and Remark~\ref{rem:bdry-trick}), it is natural to expect that the same holds for their subsequential limits~\eqref{eq:h-subseq}. In the critical case~$m=0$, one can prove this fact without assuming that the boundary~$\partial\Omega$ is smooth by using, e.\,g., the techniques developed in~\cite[Section~6]{ChelkakSmirnov2}. However, we prefer to quote a more straightforward argument { suggested by S.\,C.\,Park} in~\cite[Proposition~22]{park2018massive}, which works in $C^1$-smooth domains only but instead has a great advantage of admitting a straightforward generalization to the~$m<0$ case.
{ Let us emphasize that we will also rely upon Proposition~\ref{prop:bc-are-OK} when discussing the massive setup in Section~\ref{sub:conv-mass}; namely, in the proof of Theorem~\ref{thm:conv-omega-mass}.}

\begin{proposition} \label{prop:bc-are-OK}
{ Let~$m\le 0$} and assume that the boundary of~$\Omega$ is $C^1$-smooth and that~$\Omega^\delta$ approximate~$\Omega$ in the Hausdorff sense so that~$\mathrm{dist}(\partial\Omega^\delta;\partial\Omega)=O(\delta)$ as~$\delta\to 0$. Provided that the estimate~$M^\delta=O(1)$ holds, the following are fulfilled:

\smallskip

\noindent (i) the unform bound~\eqref{eq:F=O(1)} holds up to the boundary of discrete domains~$\partial\Omega^\delta$.

\smallskip

\noindent (ii) for each subsequential limit~$g_{[\Omega;v,w]}$, the function~$h$ defined by~\eqref{eq:h-subseq} is continuous in~$\overline{\Omega}\smallsetminus\{v,w\}$ and satisfies Dirichlet boundary conditions at~$\partial\Omega$;

\smallskip

\noindent (iii) moreover,~$g_{[\Omega;v,w]}$ is continuous up to the boundary of~$\Omega$ and satisfies Riemann-type boundary conditions~$\Im[g_{[\Omega;v,w]}(\zeta)(\tau(\zeta))^{1/2}]=0$ for all~$\zeta\in \partial\Omega$, where~$\tau(\zeta)$ denotes the tangent vector to~$\partial\Omega$ at the point~$\zeta$ oriented counterclockwise.
\end{proposition}

\begin{proof} Recall that the functions~$H^\delta$ are \emph{sub}-harmonic on~$\Gamma^{\bullet,\delta}$ due to Proposition~\ref{prop:H-sub} (or directly due to~\cite[Propostion~3.6]{ChelkakSmirnov2} if~$m=0$) and that this property remains true if one performs the `boundary modification trick' from~\cite[Section~3.6]{ChelkakSmirnov2}; see also Remark~\ref{rem:bdry-trick} above. Therefore, a comparison with the discrete harmonic measure and the estimate~$M^\delta=O(1)$ imply that~$H^\delta(v)=O(\delta)$ at near-to-boundary vertices~$v\in\Gamma^{\bullet,\delta}$ and hence~$F^\delta_{[\Omega^\delta;v,w]}(z)=O(1)$ for all~$z\in\partial\Omega^\delta$.

The discrete maximum principle for~$|F^\delta|$ (which holds up to a universal multiplicative constant due to Lemma~\ref{lem:max-principle} or directly due to results of~\cite[Section~2.5]{chelkak2020ising} if~$m=0$) together with the uniform estimate~\eqref{eq:F=O(1)} in the bulk of~$\Omega^\delta$ imply that
\[
F^\delta_{[\Omega^\delta;v,w]}=O(1)\ \ \text{as}\ \ \delta\to 0\ \ \text{uniformly in}\ \ \Omega^\delta\smallsetminus(B(v,r_0)\cup B(w,r_0)),
\]
including at the points close to the boundary~$\partial\Omega^\delta$.

Now one easily sees that each subsequential limit~\eqref{eq:F-subseq-lim} is uniformly bounded up to the boundary of~$\Omega$ (i.\,e., on compact subsets of~$\overline{\Omega}\smallsetminus\{v,w\}$). In particular, the primitive~$h=\frac12\int\Im[(g_{[\Omega;v,w]}(z))2dz]$ is continuous in~$\overline{\Omega}\smallsetminus\{v,w\}$ and satisfies the Dirichlet boundary conditions at~$\partial\Omega$.

Moreover, $\Delta h=4m^2|g_{[\Omega;u,v]}|=O(1)$ up to the boundary of a $C^1$-smooth domain~$\Omega$. Due to standard estimates, this implies that~$h$ is continuously differentiable and hence~$g_{[\Omega;u,v]}$ is continuous \emph{up to the boundary of~$\Omega$,} which allows one to speak about its boundary values. Finally, as pointed out in the proof of~\cite[Proposition~22]{park2018massive}, the discrete integration by parts argument used in~\cite[Remark~6.3]{ChelkakSmirnov2} only relies upon the sub-harmonicity of function~$H^\delta$ on~$\Gamma^{\bullet,\delta}$ and thus can be applied verbatim to control the sign of the normal derivative of~$h$ provided that~$m\le 0$.
\end{proof}

We are now in the position to prove the main result of this section. Note that we do \emph{not} assume the bound~$M^\delta=O(1)$ as~$\delta\to 0$ anymore.
\begin{theorem} \label{thm:conv-Fwv-crit} The following holds uniformly on compact subsets of~$\Omega\smallsetminus\{v,w\}$:
\[
F^\delta_{[\Omega^\delta;v,w]}(z)\ \to\ (\tfrac2\pi)^{\frac 12}f_{[\Omega;v,w]}(z)\ \ \text{as}\ \ \delta\to 0.
\]
Moreover, this convergence is also uniform with respect to positions of~$v$ and~$w$ provided that $v,w,z$ stay at definite distance from each other and from~$\partial\Omega$.
\end{theorem}
\begin{proof} We first prove the required result provided that~$M^\delta=O(1)$ as~$\delta\to 0$ and then rule out the impossible scenario~$M^\delta\to\infty$ as~$\delta=\delta_k\to\infty$.

If~$M^\delta=O(1)$ as~$\delta\to 0$ then we only need to prove that each subsequential limit~$g_{[\Omega;v,w]}$ of s-holomorphic observables~$F^\delta_{[\Omega^\delta;v,w]}$ solves the same boundary value problem as~$f_{ [\Omega;v,w]}$, up to the multiple~$(2/\pi)^{1/2}$ in the asymptotics~\eqref{eq:f-asymp-near-v}. (Recall that the existence of these subsequential limits~$g_{[\Omega;v,w]}$ follows from Proposition~\ref{prop:F=O(1)} and the Arzel\`a--Ascoli theorem.) Proposition~\ref{prop:bc-are-OK} guarantees that~$g_{[\Omega;v,w]}$ satisfies the required Riemann-type boundary conditions at~$\partial\Omega$. Therefore, it remains to analyze the asymptotics of~$g_{[\Omega;v,w]}(z)$ as $z\to v$ and as~$z\to w$.

It follows from the convergence results for s-holomorphic spinors~$F^\delta_{[\Omega^\delta;v,w]}$ and~$\cG^\delta_{[v]}$ that the functions~$H^{\delta,\dagger}$ (constructed from the fermionic observable~\eqref{eq:X-dag-def} in a usual way) converge to a function
\[
\textstyle h^\dagger\ :=\ \frac{1}{2}\int\Im [(g_{[\Omega;v,w]}(z)-\big(\tfrac2\pi\big)^{\frac12}e^{-i\frac\pi4}(z-v)^{-\frac 12})^2dz]
\]
on compact subsets of a \emph{punctured} disc $B(v,2r_0)\smallsetminus\{w\}$. However, the functions~$H^{\delta,\dagger}$ remain uniformly bounded in the whole disc~$B(v,2r_0)$ due to the discrete maximum principle. Therefore, the \emph{harmonic} function~$h^\dagger$ has a removable singularity at~$v$ and
\[
g_{[\Omega;v,w]}(z)-\big(\tfrac2\pi\big)^{\frac12}e^{-i\frac\pi4}(z-v)^{-\frac 12}\ =\ O(|z-v|^{\frac 12})\ \ \text{as}\ \ z\to v
\]
(the right-hand side automatically improves from~$O(1)$ to~$O(|z-v|^{\frac 12})$ since~$g_{[\Omega;v,w]}$ has a square-root-type branching at~$v$). A similar argument applies near the second branching point~$w$. More precisely, using the estimate~\eqref{eq:X(cw)=O(1)} and passing to a subsequence once more, we can assume that
\begin{equation}
\label{eq:X(cw)->B}
X^\delta_{[\Omega^\delta;v,w]}(c_w)\ \to\ B\in\R\ \ \text{as}\ \ \delta=\delta_k\to 0.
\end{equation}
Then, the same argument as above implies that
\[
g_{[\Omega;v,w]}(z) - B\cdot \big(\tfrac2\pi\big)^{\frac12}e^{i\frac\pi4}(z-w)^{-\frac 12}\ =\ O(|z-w|^{\frac12})\ \ \text{as}\ \ z\to w.
\]

Let us now rule out the scenario when~$M^\delta\to\infty$ as~$\delta=\delta_k\to 0$. In this hypothetical situation one can consider the re-scaled observables
\[
\widetilde{F}{}^\delta_{[\Omega^\delta;v,w]}\ :=\ (M^\delta)^{-\frac{1}{2}}\cdot F^\delta_{[\Omega^\delta;v,w]}
\]
and repeat the arguments given above relying upon the identity~$\max_{\Omega(r_0)}|\widetilde{H}{}^\delta|=1$ for the corresponding functions~$\widetilde{H}{}^\delta$. Each subsequential limit~$\widetilde{g}_{[\Omega;v,w]}$ obtained in this way solves the same boundary value problem as~$f_{[\Omega;v,w]}$ except that
\[
\widetilde{g}_{[\Omega;v,w]}(z)\ =\ O(|z-v|^{\frac12})\ \ \text{as}\ \ z\to v
\]
since the re-scaled functions~$(M^\delta)^{-\frac 12}\cG^\delta_{[v]}$ vanish (on compact subsets of~$\C\smallsetminus\{v\}$) in the limit~$\delta\to 0$. Similarly to Remark~\ref{rem:f-uniq}(ii), this boundary value problem does not admit a non-trivial solution, i.\,e.,~$\widetilde{g}_{[\Omega;v,w]}(z)=0$ for all~$z\in\Omega$. However, together with the uniform estimate $\widetilde{F}{}^\delta_{ [\Omega^\delta;v,w]}=O(1)$ near the boundary of~$\Omega^\delta$ provided by Proposition~\ref{prop:F=O(1)}, this implies that
\[
\widetilde{H}{}^\delta\ \to\ 0\ \ \text{as}\ \ \delta\to 0\ \ \text{on compact subsets of$\ \ \overline{\Omega}\smallsetminus\{v,w\}$},
\]
which contradicts to the normalization $\max_{\Omega(r_0)}|\widetilde{H}{}^\delta|=1$.

Finally, note that we never used the fact that the positions of points~$v^\delta=v$ and~$w^\delta=w$ are fixed; in all the arguments given above it is sufficient to assume that~$u^\delta\to u$ and $w^\delta\to w$ as $\delta\to 0$. By compactness, this implies that all the convergence statements discussed above are uniform in~$u$ and $w$ provided that these points stay at definite distance from each other and from~$\partial\Omega$.
\end{proof}

\subsection{Critical model: convergence of ratios of spin-spin correlations} \label{sub:conv-crit}
We now discuss several corollaries of Theorem~\ref{thm:conv-Fwv-crit} following the scheme designed in~\cite{ChelkakHonglerIzyurov}. However, let us note that we give more straightforward proofs than those from~\cite{ChelkakHonglerIzyurov}.
The following result generalizes~\cite[Theorem~1.7]{ChelkakHonglerIzyurov} to isoradial grids.
\begin{corollary} \label{cor:conv-to-B-crit}
In the same setup as above, let~$b\in \Omega^{\bullet,\delta}$ be adjacent to~$w\in\Omega^{\circ,\delta}$ and recall that~$v\sim u$. Then,
\begin{equation}
\label{eq:conv-to-B-crit}
\frac{\E^\wired_{\Omega^\delta}[\mu_v\mu_b]}{\E^\wired_{\Omega^\delta}[\sigma_u\sigma_w]}\ =\ \frac{\E^\free_{\Omega^{*,\delta}}\,[\sigma_v\sigma_b]}{\E^\wired_{\Omega^\delta}[\sigma_u\sigma_w]}\ \to\ \cB_\Omega(v,w)\ \ \text{as}\ \ \delta\to 0\,.
\end{equation}
\end{corollary}
\begin{remark} Recall that, according to the Kramers--Wannier duality, the disorder-disorder correlator~$\E^\wired_{\Omega^\delta}[\mu_v\mu_b]$ is equal to the spin-spin correlation $\E^\free_{\Omega^{*,\delta}}[\sigma_v\sigma_b]$ in the \emph{dual} Ising model on~$\Omega^{*,\delta}$ with free boundary conditions.
\end{remark}
\begin{proof} Let~$c_w\in\Upsilon^\times_{[v,w]}(\Omega^\delta)$ be adjacent to both~$w$ and~$b$. By definition of the observable~$X^\delta_{[\Omega^\delta;v,w]}$ (see~\eqref{eq:Xvw-def}), we have ${\E^\wired_{\Omega^\delta}[\mu_v\mu_b]}/{\E^\wired_{\Omega^\delta}[\sigma_u\sigma_w]}=\pm X^\delta_{[\Omega^\delta;v,w]}(c_w)$. In fact, the convergence of the values $X^\delta_{[\Omega^\delta;v,w]}(c_w)$ has been already implicitly proven along the proof of Theorem~\ref{thm:conv-Fwv-crit}. Indeed, each subsequential limit~$B$ in~\eqref{eq:X(cw)->B} has to be equal to~$\cB_\Omega(v,w)$ since the (a priori, unknown) coefficient in the asymptotics~\eqref{eq:f-asymp-near-w} is uniquely determined by the boundary problem itself.

Alternatively, one can consider the function~$H\big[X^\delta_{[\Omega^\delta;v,w]},\rG^\delta_{[b]}\big]$ defined in a vicinity of~$w$ and note that its additive monodromy around~$c_w$ is equal to
\begin{align*}
{ 2X^\delta_{[\Omega^\delta;v,w]}(c_w)\ }&=\ 2X^\delta_{[\Omega^\delta;v,w]}(c_w)\rG^\delta_{[b]}(c_w)\\
& =\ \frac{1}{2}\oint^{[\delta]}_{z:|z-w|=2r_0}\Im\big[F^\delta_{[\Omega^\delta;v,w]}(z)\cG^\delta_{[b]}(z)dz\big];
\end{align*}
see Lemma~\ref{lem:XG-monodromy}. Using the convergence of~$F^\delta_{[\Omega^\delta;v,w]}(z)$ and~$\cG^\delta_{[b]}(z)$ as~$\delta\to 0$ near the circle~$\{z:|z-w|=2r_0\}$ we easily see that
\begin{align*}
{ X^\delta_{[\Omega^\delta;v,w]}(c_w)\ }\to\ { \frac{1}{2\pi}}\oint^{[\delta]}_{z:|z-w|=2r_0}\Im\big[f_{[\Omega;v,w]}(z)(z-w)^{-\frac12}dz\big]\ =\ \pm \cB_{\Omega}(v,w),
\end{align*}
where we used the asymptotics~\eqref{eq:f-asymp-near-w} in the last equation. (The potential ambiguity in the sign is caused by the fact that we work with spinors branching over~$w$ rather than with single-valued holomorphic functions.) The $\pm$ sign can be easily fixed by noting that both~${\E^\wired_{\Omega^\delta}[\mu_v\mu_{b}]}/{\E^\wired_{\Omega^\delta}[\sigma_u\sigma_w]}$ and~$\cB_\Omega(v,w)$ are positive quantities.
\end{proof}
Following~\cite{ChelkakHonglerIzyurov}, let us introduce the quantity~$\cA_\Omega(v,w)\in\C$ as the sub-leading coefficient in the asymptotics of the spinor~$f_{[\Omega;v,w]}$ at the point~$z=v$:
\begin{equation}
\label{eq:cA-def}
f_{[\Omega;v,w]}(z)\ =\ e^{-\frac\pi 4}(z-v)^{-\frac 12}\cdot\big(1+2\cA_\Omega(v,w)(z-v)+O(|z-v|^2)\big).
\end{equation}
Further, define the two-point spin correlation function $\langle\sigma_{u_1}\sigma_{u_2}\rangle^\wired_\Omega$ as the exponential of the primitive of the following (closed) differential form:
\begin{equation}
\label{eq:sigma-sigma-def}
\langle\sigma_{u_1}\sigma_{u_2}\rangle^\wired_\Omega\ :=\ \exp \int^{(u_1,u_2)}\Re\big[\cA_\Omega(v,w)dv+\cA_\Omega(w,v)dw],
\end{equation}
normalized so that~$\langle\sigma_{u_1}\sigma_{u_2}\rangle^\wired_\Omega\sim |u_2-u_1|^{-\frac 14}$ as~$u_2\to u_1$. For simply connected domains~$\Omega$, the function~$\langle\sigma_{u_1}\sigma_{u_2}\rangle^\wired_\Omega$ can be written explicitly (e.\,g., see~\cite[Eq.~(1.2)]{ChelkakHonglerIzyurov}); it is also worth noting that the existence of such a primitive can be deduced from the convergence results discussed in Corollaries~\ref{cor:conv-to-A-crit},~\ref{cor:u'w'/uw-conv-crit} and Theorem~\ref{thm:spin-univ-crit} below. Also, let us define
\begin{equation}
\label{eq:sigma-sigma-free-def}
\langle\sigma_{u_1}\sigma_{u_2}\rangle^\free_\Omega\ :=\ \cB_\Omega(u_1,u_2)\cdot \langle\sigma_{u_1}\sigma_{u_2}\rangle^\wired_\Omega\,.
\end{equation}
It is not hard to deduce from~\eqref{eq:fvw-conf-cov} that, for conformal maps~$\varphi:\Omega\to\Omega'$,
\begin{equation}
\label{eq:sigma-sigma-cov}
\langle\sigma_{u_1}\sigma_{u_2}\rangle^\mathfrak{b}_\Omega\ =\ \langle\sigma_{\varphi(u_1)}\sigma_{\varphi(u_2)}\rangle^\mathfrak{b}_{\Omega'}\cdot |\varphi'(u_1)|^{\frac18}|\varphi'(u_2)|^{\frac18}.
\end{equation}
for both wired ($\mathfrak{b}=\wired$) and free ($\mathfrak{b}=\free$) boundary conditions.

The following result generalizes~\cite[Theorem~1.5]{ChelkakHonglerIzyurov} to the isoradial setup; it is worth noting that the proof given below considerably simplifies the arguments used in~\cite{ChelkakHonglerIzyurov} even in the square grid case~$\Lambda^\delta=\delta\mathbb Z^2$.

\begin{corollary}\label{cor:conv-to-A-crit} In the same setup as above, let~$u_0,u_1\in \Gamma^{\circ,\delta}$ be two neighboring `white' vertices adjacent to~$v\in \Gamma^{\bullet,\delta}$. Then, the following asymptotics hold:
\[
\frac{\E^\wired_{\Omega^\delta}[\sigma_{u_1}\sigma_w]}{\E^\wired_{\Omega^\delta}[\sigma_{u_0}\sigma_w]}\ =\ 1+\Re\big[(u_1-u_0)\cdot \cA_\Omega(v,w)\big]+o(\delta),
\]
where the error term is uniform with respect to~$v,w\in\Omega$ provided that they remain at a definite distance from each other and from~$\partial\Omega$.
\end{corollary}
\begin{proof} Let~$v(c_{0,1})=v$,~$u(c_{0,1})=u_{0,1}$, and assume that~$c_{0,1}\in\Upsilon^\times_{[v;w]}(\Omega^\delta)$ are chosen to be adjacent on the double cover. Consider the observable~\eqref{eq:Xvw-def} with~$u=u_0$ and note that
\[
\frac{\E^\wired_{\Omega^\delta}[\sigma_{u_1}\sigma_w]}{\E^\wired_{\Omega^\delta}[\sigma_{u_0}\sigma_w]}-1\ =\ X^\delta_{[\Omega^\delta;v,w]}(c_1)-X^\delta_{[\Omega^\delta;v,w]}(c_0).
\]
Using Lemma~\ref{lem:XG-monodromy} as in the proof of Corollary~\ref{cor:conv-to-B-crit} we see that
\begin{align*}
X^\delta_{[\Omega^\delta;v,w]}(c_1)&-X^\delta_{[\Omega^\delta;v,w]}(c_0)\\
& =\ {\frac{1}{4}}\oint^{[\delta]}_{z:|z-v|=2r_0}\Im \big[ F^\delta_{[\Omega^\delta;v,w]}(z)\cdot (\cG^\delta_{[u_1]}(z)-\cG^\delta_{[u_0]}(z))dz\big].
\end{align*}
Note that the asymptotics~\eqref{eq:G-asymp-crit} contain no $O(\delta)$ term. Therefore,
\[
\cG^\delta_{[u_1]}(z)-\cG^\delta_{[u_0]}(z)\ =\ \tfrac{1}{2}(u_1-u_0)\cdot e^{i\frac\pi4}{ \big(\tfrac{2}{\pi}\big)^{\frac{1}{2}}}(z-v)^{-\frac32}+O(\delta^2)\ \ \text{as}\ \ \delta\to 0.
\]
Since we also know that~$F^\delta_{[\Omega^\delta;v,w]}(z)\to \big(\tfrac{2}{\pi}\big)^{\frac12}f_{[\Omega;v,w]}(z)$ as~$\delta\to 0$ (uniformly near the circle~$\{z:|z-v|=2r_0\}$, see Theorem~\ref{thm:conv-Fwv-crit}), this gives
\begin{align*}
\frac{\E^\wired_{\Omega^\delta}[\sigma_{u_1}\sigma_w]}{\E^\wired_{\Omega^\delta}[\sigma_{u_0}\sigma_w]}-1\
&=\ \frac{1}{4\pi}\oint_{z:|z-v|=2r_0}\Im\biggl[(u_1-u_0)\cdot\frac{e^{i\frac\pi 4}f_{[\Omega;v,w]}(z)dz}{(z-v)^{\frac32}}\biggr]+o(\delta)\\[2pt]
&=\ \Re\big[(u_1-u_0)\cdot \cA_\Omega(v,w)\big]+o(\delta),
\end{align*}
where we used definition~\eqref{eq:cA-def} of the coefficient~$\cA_\Omega(v,w)$ in the last equation.
\end{proof}

\begin{corollary} \label{cor:u'w'/uw-conv-crit}
In the same setup as above, the following holds:
\[
\frac{\E^\wired_{\Omega^\delta}[\sigma_{u'}\sigma_{w'}]}{\E^\wired_{\Omega^\delta}[\sigma_u\sigma_w]}\ \to\ \frac{\langle\sigma_{u'}\sigma_{w'}\rangle^\wired_\Omega}{\langle\sigma_u\sigma_w\rangle^\wired_\Omega}\ \ \text{as}\ \ \delta\to 0,
\]
uniformly with respect to~$u,u',w,w'\in\Omega$ provided that all these four points stay at a definite distance~$\rho>0$ from~$\partial\Omega$ and that~$|u-w|\ge \rho$ and~$|u'-w'|\ge \rho$.
\end{corollary}
\begin{proof} Note that the asymptotics provided by Corollary~\ref{cor:conv-to-A-crit} can be equivalently written as
\[
\log \frac{\E^\wired_{\Omega^\delta}[\sigma_{u_1}\sigma_w]}{\E^\wired_{\Omega^\delta}[\sigma_{u_0}\sigma_w]}\ =\ \Re\big[(u_1-u_0)\cdot \cA_\Omega(v,w)\big]+o(\delta).
\]
Summing these expression along appropriate paths connecting~$u$ to~$u'$ and~$w$ to~$w'$ inside~$\Omega^\delta$ and using definition~\eqref{eq:sigma-sigma-def} of the continuous correlation functions~$\langle\sigma_u\sigma_w\rangle^\wired_\Omega$ one easily gets the result.
\end{proof}

\subsection{Critical model: universality of the full-plane spin-spin correlations} \label{sub:univ-crit}
In this section we consider the critical Ising model on infinite isoradial grids~$\Lambda$ with~$\delta=1$ and prove that the two-point correlations~$\E_{\Lambda}[\sigma_u\sigma_w]$ have \emph{universal} (i.\,e., independent of~$\Lambda$) rotationally invariant asymptotics as~$|u-w|\to \infty$.
\begin{theorem} \label{thm:spin-univ-crit}
Let~$\mathcal{C}_\sigma=2^{\frac 16}e^{\frac32\zeta'(-1)}$. The following asymptotics hold:
\begin{equation}
\label{eq:spin-univ-crit}
\E_{\Lambda}[\sigma_u\sigma_w]\ \sim\ \mathcal{C}_\sigma^2\cdot |u-w|^{-\frac14}\ \ \ \text{as}\ \ |u-w|\to\infty,
\end{equation}
uniformly with respect to~$\Lambda$ provided that it satisfies the property~$\BAP$.
\end{theorem}
\begin{proof} The proof is similar to the proof of Proposition~\ref{prop:E[sigma-sigma]=cst} for the subcritical model, see also Remark~\ref{rem:reduction-to-square}. More precisely, it goes in the following three steps:
\renewcommand\theenumi{\roman{enumi}}
\begin{enumerate}
\item prove that the asymptotics~\eqref{eq:spin-univ-crit} holds if~$\Lambda=\mathbb{Z}^2$ is the square grid;
\item prove that~\eqref{eq:spin-univ-crit} holds uniformly over the class of rectangular grids~$\Lambda^\mathrm{rect}$;
\item prove that the same asymptotics is fulfilled for all isoradial grids.
\end{enumerate}
The proofs of steps (ii) and (iii) are similar to each other and are based upon gluing large pieces of~$\Lambda^\mathrm{rect}$ and~$\mathbb Z^2$ in step (ii), and the `star extension' construction discussed in Section~\ref{sub:star-ext}, which allows to glue a large piece of~$\Lambda$ with a piece of a rectangular grid in step (iii). Thus, we only discuss steps (i) and (iii) below.

Step (i) follows from (ia) the explicit computation of `diagonal' correlations on the square grid due to Wu~\cite[Section~XI.4]{mccoy2014two}, see also~\cite[Appendix]{chelkak_Hongler_Mahfouf} for a short proof of this result; and (ib) the fact that these asymptotics are rotationally invariant,] see~\cite[Remark~2.6]{ChelkakHonglerIzyurov} (and~\cite{pinson} were this result was proven by other techniques).

To prove (iib), let us fix~$\varepsilon>0$ and let~$A=A(\varepsilon)\gg 1$ be chosen so that the RSW estimates~\eqref{eq:RSW-sigma-sigma} hold. Denote~$\rho=\rho(\varepsilon):=1/A(\varepsilon)$ and  $\delta:=2\rho\cdot |u-w|^{-1}$. Let the scaled grid~$\delta\mathbb{Z}^2$ be also shifted so that~$u$ and~$w$ become the { (approximations of)} points~$\pm \alpha \rho$ on~$\delta\mathbb{Z}^2$, where~$|\alpha|=1$. Denote by~$\mathbb{D}^\delta$ the discretization of the unit disc~$\mathbb{D}:=\{z:|z|<1\}$ on~$\delta\mathbb{Z}^2$. It follows from~\eqref{eq:RSW-sigma-sigma} that
\[
\frac{\E_{\mathbb{D}^\delta}^\wired[\sigma_{\alpha\rho}\sigma_{-\alpha\rho}]} {\E_{\delta\mathbb{Z}^2}[\sigma_{\alpha\rho}\sigma_{-\alpha\rho}]}=1+O(\varepsilon),\ \ \text{uniformly in~$\delta$.}
\]
At the same time, Corollary~\eqref{cor:u'w'/uw-conv-crit} and the rotational invariance of continuous correlation functions~$\langle\sigma_{\alpha\rho}\sigma_{-\alpha\rho}\rangle_{\mathbb D}^\wired=\langle\sigma_{\rho}\sigma_{-\rho}\rangle_{\mathbb D}^\wired$ imply that
\[
\frac{\E_{\mathbb{D}^\delta}^\wired[\sigma_{\alpha\rho}\sigma_{-\alpha\rho}]}
{\E_{\mathbb{D}^\delta}^\wired[\sigma_{\rho}\sigma_{-\rho}]}=1+o_{\delta\to 0}(1)\ \ \text{for each fixed~$\rho(\varepsilon)$.}
\]
Using~\eqref{eq:RSW-sigma-sigma} once again to compare~$\E_{\mathbb{D}^\delta}^\wired[\sigma_{\rho}\sigma_{-\rho}]$ and~$\E_{\delta\mathbb{Z}^2}[\sigma_{\rho}\sigma_{-\rho}]$ we see that
\[
\frac{\E_{\delta\mathbb{Z}^2}[\sigma_{\alpha\rho}\sigma_{-\alpha\rho}]}
{\E_{\delta\mathbb{Z}^2}[\sigma_{\rho}\sigma_{-\rho}]}\ =\ 1+O(\varepsilon)+o_{\delta\to 0}(1)
\]
Choosing first~$\varepsilon$ small enough and then~$|u-w|=2\rho(\varepsilon)\cdot \delta^{-1}$ big enough one obtains the required rotational invariance property (ib).

We now discuss step (iii), recall that step (ii) can be done following exactly the same lines by replacing the pair $\Lambda$, $\Lambda^\mathrm{rect}$ of isoradial grids by $\Lambda^\mathrm{rect}$ and~$\delta\mathbb{Z}^2$ in the argument given below; cf. Remark~\ref{rem:reduction-to-square}. Let~$A=A(\varepsilon)\gg 1$, $\rho=\rho(\varepsilon):=1/A(\varepsilon)$ and $\delta:=2\rho\cdot |u-w|^{-1}$ be chosen as above. Also, let~$\Lambda^\delta$ be the grid~$\Lambda$ scaled by the factor~$\delta$, rotated and shifted so that~$u$ and~$w$ become the points~$\pm \rho$ on~$\Lambda^\delta$.

Further, let~$\Lambda^\delta_1=\Lambda^\delta_1(0)$ be the discretization of the box $[-1,1]\times[-1,1]$ on~$\Lambda^\delta$ and denote by \mbox{$\Lambda^{\star,\delta}:=[\Lambda^\delta_1]^\star$} the `star extension' of this box (see Section~\ref{sub:star-ext}). Recall that~$\Lambda^{\star,\delta}$ contains pieces of rectangular lattices located at $O(1)$ distance { (uniformly in~$\delta$)} from the origin. Therefore, there we can find a point $z\in\C$ such that \mbox{$|z|\le C=C(\theta_0)$} and { that} the box~$\Lambda^{\star,\delta}_1(z)$ of size~$2\times 2$ centered at~$z$ is a subset of a rectangular grid~$\Lambda^{\mathrm{rect},\delta}$.

Finally, denote by~$\Omega\subset\C$ the disc of radius~$C+1$ centered at~$\frac{1}{2}z$. Corollary~\ref{cor:u'w'/uw-conv-crit} implies the convergence
\[
\frac{\E^\wired_{\Omega^\delta}[\sigma_\rho\sigma_{-\rho}]}{\E^\wired_{\Omega^\delta}[\sigma_{z+\rho}\sigma_{z-\rho}]}\ \to\ \frac{\langle \sigma_{\rho}\sigma_{-\rho}\rangle_\Omega^\wired}{\langle \sigma_{z+\rho}\sigma_{z-\rho}\rangle_\Omega^\wired}\ =\
1\ \ \text{as}\ \ \delta\to 0,
\]
where~$\Omega^\delta$ stands for the discretization of the disc~$\Omega$ on the modified grid~$\Lambda^{\star,\delta}$.

At the same time, the RSW estimates~\eqref{eq:RSW-sigma-sigma} give
\begin{align*}
\E^\wired_{\Omega^\delta}[\sigma_{\rho}\sigma_{-\rho}]\ &\le\ \E^\wired_{\Lambda^\delta_1(0)}[\sigma_{\rho}\sigma_{-\rho}]\ \le\ (1\!-\!\varepsilon)^{-1}\cdot \E_{\Lambda^\delta}[\sigma_{\rho}\sigma_{-\rho}],\\
\E^\wired_{\Omega^\delta}[\sigma_{\rho}\sigma_{-\rho}]\ &\ge\ (1-\varepsilon)\cdot \E^\wired_{\Lambda^\delta_1(0)}[\sigma_{\rho}\sigma_{-\rho}]\ \ge\ (1\!-\!\varepsilon)\cdot \E_{\Lambda^\delta}[\sigma_{\rho}\sigma_{-\rho}]
\end{align*}
and similarly for~$\E_{\Omega^\delta}^\wired[\sigma_{z-\rho}\sigma_{z+\rho}]$ and~$\E_{\Lambda^{\mathrm{rect},\delta}}[\sigma_{z+\rho}\sigma_{z-\rho}]$. Therefore,
\[
\frac{\E_{\Lambda^\delta}[\sigma_{\rho}\sigma_{-\rho}]}{\E_{\Lambda^{\mathrm{rect},\delta}}[\sigma_{z+\rho}\sigma_{z-\rho}]}\ =\ 1+O(\varepsilon)+o_{\delta\to 0}(1),
\]
where the error term~$O(\varepsilon)$ is uniform in~$\delta$. Moreover, it follows from step (ii) applied to the re-scaled grid~$\delta^{-1}\Lambda^{\mathrm{rect},\delta}$ with mesh size~$1$ that
\[
\E_{\Lambda^{\mathrm{rect},\delta}}[\sigma_{z+\rho}\sigma_{z-\rho}]\ =\ \mathcal{C}_\sigma^2\cdot (2\rho\delta^{-1})^{-\frac14} \cdot (1+o_{\delta\to 0}(1)).
\]
Since~$\delta=2\rho\cdot|u-w|^{-1}$, we arrive at the asymptotics
\[
\E_{\Lambda}[\sigma_u\sigma_w]\ =\ \E_{\Lambda^\delta}[\sigma_\rho\sigma_{-\rho}]\ =\ \mathcal{C}_\sigma^2\cdot|u-w|^{-\frac14}\cdot (1+O(\varepsilon)+o_{|u-w|\to\infty}(1)),
\]
with the uniform (in $u,w$) error term~$O(\varepsilon)$. By first choosing~$\varepsilon$ small enough and then~$|u-w|$ large enough one obtains the required asymptotics~\eqref{eq:spin-univ-crit}.
\end{proof}

It is easy to see that Theorem~\ref{thm:spin-univ-crit} allows to pass from the convergence of \emph{ratios} of two-point spin correlations in finite domains discussed in Corollary~\ref{cor:u'w'/uw-conv-crit} to the convergence of these correlations themselves; see also Remark~\ref{rem:CHI-is-OK} below.

\begin{corollary} \label{cor:conv-omega-crit}
Let~$\Omega\subset\C$ be a bounded simply connected domain with~$C^1$-smooth boundary and assume that~$u,w\in\Omega$, $u\ne w$. Let isoradial grids~$\Lambda^\delta$ satisfy the property~$\BAP$ and discrete domains~$\Omega^\delta\subset\Gamma^\delta$ approximate~$\Omega$ in the Hausdorff sense so that~$\mathrm{dist}(\partial\Omega^\delta;\partial\Omega)=O(\delta)$. Then,
\[
\begin{array}{lcl}
\delta^{-\frac14}\E_{\Omega^\delta}^\wired[\sigma_u\sigma_w]&\to&\mathcal{C}_\sigma^2\cdot\langle\sigma_u\sigma_w\rangle^\wired_{\Omega}\,,\\[2pt] \delta^{-\frac14}\E_{\Omega^\delta}^\free[\sigma_u\sigma_w]&\to& \mathcal{C}_\sigma^2\cdot\langle\sigma_u\sigma_w\rangle^\free_{\Omega}
\end{array}\quad \text{as}\ \ \delta\to 0.
\]
\end{corollary}
\begin{proof} Let~$\rho>0$ be small enough. It follows from Corollary~\ref{cor:u'w'/uw-conv-crit} that
\[
\frac{\E^\wired_{\Omega^\delta}[\sigma_u\sigma_w]}{\E^\wired_{\Omega^\delta}[\sigma_u\sigma_{u+\rho}]}\cdot \frac{\langle\sigma_u\sigma_{u+\rho}\rangle^\wired_\Omega}{\langle\sigma_u\sigma_w\rangle^\wired_\Omega}\ =\ 1+o_{\delta\to 0}(1)\ \ \text{for each fixed~$\rho>0$.}
\]
Recall that the multiplicative normalization of the continuous correlation functions~\eqref{eq:sigma-sigma-def} is fixed so that
\[
\langle\sigma_u\sigma_{u+\rho}\rangle^\wired_\Omega\ =\ \rho^{-\frac14}\cdot(1+o_{\rho\to 0}(1)).
\]
At the same time, the RSW estimates~\eqref{eq:RSW-sigma-sigma} imply that
\[
\frac{\E^\wired_{\Omega^\delta}[\sigma_u\sigma_{u+\rho}]}{\E_{\Lambda^\delta}[\sigma_u\sigma_{u+\rho}]}\ =\ 1+o_{\rho\to 0}(1)\ \ \text{uniformly in~$\delta\le \delta_0(\rho)$.}
\]
Finally, due to Theorem~\ref{thm:spin-univ-crit} we have the asymptotics
\[
\E_{\Lambda^\delta}[\sigma_u\sigma_{u+\rho}]\ =\ \mathcal{C}_\sigma^2\cdot (\delta^{-1}\rho)^{-\frac14}\cdot (1+o_{\delta\to 0}(1))\ \ \text{for each fixed~$\rho>0$.}
\]
Combining these asymptotics together, first choosing~$\rho$ small enough, and then~$\delta$ small enough we obtain the required convergence result for~$\delta^{-\frac14}\E_{\Omega^\delta}^\wired[\sigma_u\sigma_w]$.

Recall that the continuous correlation functions $\langle\sigma_u\sigma_w\rangle_\Omega^\free$ are defined by~\eqref{eq:sigma-sigma-free-def}. The convergence of spin-spin correlations with free boundary conditions follows from the Kramers--Wannier duality, Corollary~\ref{cor:conv-to-B-crit}, and the convergence of spin-spin correlations in the dual isoradial model with wired boundary conditions.
\end{proof}

\begin{remark}
\label{rem:CHI-is-OK} (i) Recall that we used the smoothness assumption on~$\partial\Omega$ only in Proposition~\ref{prop:bc-are-OK}, which claims the correct boundary conditions of subsequential limits~\eqref{eq:F-subseq-lim} of discrete fermionic observables. In the critical setup, this smoothness assumption can be dropped, e.\,g., by using techniques from~\cite[Section~6]{ChelkakSmirnov2} instead.

\smallskip

\noindent (ii) The simplification/generalization of the proofs from~\cite{ChelkakHonglerIzyurov} discussed above is not restricted to the two-point correlations and applies to all results of that paper. In particular, for all~$n\ge 1$ the convergence
\[
\delta^{-\frac{n}8}\E_{\Omega^\delta}^\wired[\sigma_{u_1}\ldots\sigma_{u_n}]\ \to\ \mathcal{C}_\sigma^n\cdot\langle\sigma_{u_1}\ldots\sigma_{u_n}\rangle_\Omega^\wired \ \ \text{as}\ \ \delta\to 0
\]
holds uniformly over the class of isoradial discretizations~$\Omega^\delta\subset\Gamma^\delta$, at least provided that $\Lambda^\delta$ satisfy the uniformly bounded angles property~$\BAP$.

\smallskip

\noindent (iii) We refer the interested reader to a recent paper~\cite{CHI_Mixed} for a discussion of further generalizations, e.\,g., (a) convergence results for more involved correlations and (b) techniques that can be used to control the boundary values of fermionic observables.
\end{remark}

\subsection{Massive model: definitions of correlation functions in continuum} \label{sub:spin-mass-def} We now move to the analysis of the massive model. In this section we { define} relevant two-point spin correlation functions \emph{in continuum}, assuming that~$m<0$ and that $\Omega$ is a $C^1$-smooth domain. Note that, contrary to the conformally invariant case~$m=0$ discussed above, this smoothness assumption plays a much more important role in what follows. In particular, it greatly simplifies the discussion of the \emph{uniqueness} of solutions of the corresponding boundary value problems (which can be avoided if~$m=0$ due to the conformal covariance identity~\eqref{eq:fvw-conf-cov}). We refer the interested reader to a recent paper~\cite{Park2021Fermionic} by S.~C.~Park, where certain techniques allowing to treat Riemann-type boundary value problems for massive fermionic observables in rough domains~$\Omega$ were developed.

The definitions given below generally follow those from~\cite{park2018massive}. However, note that we avoid the explicit use of Bessel functions in order to keep the discussion as light as possible; we refer the interested reader to~\cite{park2018massive} for more information on massive holomorphic functions and their singularities. It is worth noting that below we use a slightly different notation~$\cA^{(m)}_\Omega(v,w)$ instead of~$\cA_\Omega(v,w\,|\,m)$ etc used in~\cite{park2018massive}.

Given~$m<0$, let us denote by~$f^{(m)}_{[\Omega;v,w]}$ a spinor on the double cover $[\Omega;v,w]$ of a smooth domain~$\Omega$ ramified over $v$ and $w$ such that
\begin{itemize}
\item $f^{(m)}_{[\Omega;v,w]}$ satisfies the massive holomorphicity equation~\eqref{eq:mass-hol} in~$\Omega\smallsetminus\{v,w\}$;
\smallskip
\item $f^{(m)}_{[\Omega;v,w]}$ is continuous in~$\overline{\Omega}\smallsetminus\{v,w\}$ and satisfies the Riemann-type boundary conditions $\Im[f^{(m)}_{[\Omega;v,w]}(\zeta)(\tau(\zeta))^{1/2}]=0$ for all~$\zeta\in\partial\Omega$, where~$\tau(\zeta)$ denotes the tangent vector to~$\partial\Omega$ at the point~$\zeta$ oriented counterclockwise;
\smallskip
\item the following asymptotics near the branching points~$v,w$ hold:
\begin{align}
\label{eq:f-near-v-mass}
\qquad  f^{(m)}_{[\Omega;v,w]}(z)\ &=\ e^{-i\frac\pi 4}\frac{e^{2m|z-v|}}{\sqrt{z-v}}+O(|z-v|^{\frac12}),\qquad z\to v,\\
\qquad  f^{(m)}_{[\Omega;v,w]}(z)\ &=\ e^{i\frac\pi 4}\cB^{(m)}_\Omega(v,w)\cdot \frac{e^{-2m|z-w|}}{\sqrt{z-w}}+O(|z-w|^{\frac12}),\quad z\to w,
\label{eq:f-near-w-mass}
\end{align}
where~$\cB^{(m)}_\Omega(v,w)\in\R$ is an (a priori unknown) coefficient defined up to the sign, which we fix by requiring that~$\cB^{(m)}_\Omega(v,w)\ge 0$.
\end{itemize}

Recall that the massive holomorphicity condition~$\overline{\partial}f+im\overline{f}=0$ implies that the primitive $h:=\int\Im[(f(z))^2dz]$ is well-defined and that~$\Delta h:=-8m|f|^2$. In particular, $h$ is sub-harmonic provided that~$m<0$. Thus, the argument given in Remark~\ref{rem:f-uniq}(ii) applies and gives the \emph{uniqueness} of a solution~$f^{(m)}_{[\Omega;v,w]}$ to this boundary value problem; see~\cite[Proposition~17]{park2018massive}. However, contrary to the critical case, for~$m<0$ one cannot claim the \emph{existence} of this solution via the conformal covariance and an explicit formula in a reference domain. Instead,~$f^{(m)}_{[\Omega;v,w]}$ can be obtained as a subsequential limit of discrete observables; see Theorem~\ref{thm:conv-Fwv-mass} below.
\begin{lemma}\label{lem:B->1} Let~$\Omega$ be a smooth simply connected domain, $m\le 0$. Then, one has
\begin{equation}
\label{eq:B->1-mass}
\cB^{(m)}_\Omega(v,w)\to 1\ \ \text{as}\ \ w\to v,\quad \text{for each\ \ $v\in\Omega$.}
\end{equation}
\end{lemma}
\begin{proof} Let us emphasize that this fact is \emph{not} trivial if~$m\ne 0$ because of the lack of conformal invariance and explicit formulas. However, one can analyze the behavior of the coefficient~$\cB^{(m)}_\Omega(v,w)$ by comparing it with~$\cB_\Omega(v,w)=\cB^{(0)}_\Omega(v,w)$, which is explicit. We refer the reader to~\cite[Lemmas~28 and~29]{park2018massive} for the proof of~\eqref{eq:B->1-mass}.
\end{proof}

We are now in the position to define the coefficients~$\cA^{(m)}_\Omega(v,w)\in\C$. Similarly to~\eqref{eq:cA-def}, they are introduced by requiring that
\begin{equation}
\label{eq:cA-def-mass}
f^{(m)}_{[\Omega;v,w]}(z)\ =\ e^{-i\frac\pi 4}\frac{e^{2m|z-v|}}{(z-v)^{\frac12}}\cdot\big(1+2\cA^{(m)}_\Omega(v,w)\cdot(z-v)+O(|z-v|^2)\big)
\end{equation}
as~$z\to v$. Let us emphasize that the \emph{existence} of such an asymptotic expansion for~$m<0$ is much less trivial than~\eqref{eq:cA-def} for~$m=0$. As discussed, e.\,g., in~\cite{park2018massive} (see Eq.~(2.16) and~(2.17) of that paper), spinor solutions to the massive holomorphicity equation~\eqref{eq:mass-hol} considered near their branching points admit expansions via half-integer modified Bessel functions~$e^{\pm i\nu\arg z}I_\nu(2|mz|)$. In their turn, the function~$I_\nu(r)$ are power series in~$r$; see~\cite[Eq.~10.25.2]{NIST:DLMF}. It remains to check that the last factor in the expansion~\eqref{eq:cA-def-mass} cannot contain other linear terms than~$\mathrm{cst}\cdot(z-v)$, which we leave to the reader as a simple exercise.

The next step is to define for~$m<0$, similarly to~\eqref{eq:sigma-sigma-def} and~\eqref{eq:sigma-sigma-free-def},
\begin{align}
\label{eq:sigma-sigma-def-mass}
\langle\sigma_{u_1}\sigma_{u_2}\rangle_\Omega^{(m),\wired}\ &:=\ \exp\int^{(u_1,u_2)}\Re\big[\cA^{(m)}_\Omega(v,w)dv+\cA^{(m)}_\Omega(w,v)dw\big],\\
\langle\sigma_{u_1}\sigma_{u_2}\rangle^{(-m),\free}_\Omega\ &:=\ \cB^{(m)}_\Omega(u_1,u_2)\cdot \langle\sigma_{u_1}\sigma_{u_2}\rangle^{(m),\wired}_\Omega\,.
\label{eq:sigma-sigma-free-def-mass}
\end{align}
Again, a priori it is not obvious that the differential form in~\eqref{eq:sigma-sigma-def-mass} is exact. However, this can be immediately derived from an analogue of Corollary~\ref{cor:conv-to-A-crit} in the massive case; see Corollary~\ref{cor:conv-to-A-mass} below. Still, there remains a question of fixing the multiplicative normalization in~\eqref{eq:sigma-sigma-def-mass}, which we do by requiring that
\begin{equation}
\label{eq:sigma-sigma-norm-mass}
\langle\sigma_{u_1}\sigma_{u_2}\rangle_\Omega^{(m),\wired}\ \sim\ |u_2-u_1|^{-\frac14}\ \sim\ \langle\sigma_{u_1}\sigma_{u_2}\rangle_\Omega^{(-m),\free}\ \ \text{as}\ \ u_2\to u_1.
\end{equation}
The existence of such a normalization (for~$m<0$) follows from the monotonicity of spin-spin correlations in discrete
\[
\E^{(m),\wired}_{\Omega^\delta}[\sigma_{u_1}\sigma_{u_2}] \,\ge\, \E^{\wired}_{\Omega^\delta}[\sigma_{u_1}\sigma_{u_2}]\,,\qquad \E^{(-m),\free}_{\Omega^{*,\delta}}[\sigma_{v_1}\sigma_{v_2}]\,\le\,\E^{\free}_{\Omega^{*,\delta}}[\sigma_{v_1}\sigma_{v_2}]\,,
\]
convergence results~\eqref{eq:conv-to-B-mass},~\eqref{eq:conv-to-B-crit}, and from Lemma~\ref{lem:B->1} (applied for both~$m<0$ and~$m=0$); see the proof of Theorem~\ref{thm:conv-omega-mass}.

\smallskip

We define the full-plane correlation functions as
\[
\langle\sigma_{u_1}\sigma_{u_2}\rangle_\C^{(m)}\,:=\, \lim\nolimits_{\Omega\uparrow\C}\,\langle\sigma_{u_1}\sigma_{u_2}\rangle_\Omega^{(m),\wired},\quad \langle\sigma_{u_1}\sigma_{u_2}\rangle_\C^{(-m)}\,:=\, \lim\nolimits_{\Omega\uparrow\C}\,\langle\sigma_{u_1}\sigma_{u_2}\rangle_\Omega^{(-m),\free},
\]
where the limits are taken along arbitrary sequences of smooth domains exhausting the complex plane. These limits exist and are non-trivial since for $\Omega\subset\Omega'$ one has
\begin{align*}
\langle\sigma_{u_1}\sigma_{u_2}\rangle_\Omega^{(m),\wired}\ \ge\ \langle\sigma_{u_1}\sigma_{u_2}\rangle_{\Omega'}^{(m),\wired}\ &\ge\ \langle\sigma_{u_1}\sigma_{u_2}\rangle_{\C}\ =\ |u_2-u_1|^{-\frac14},\\
\langle\sigma_{u_1}\sigma_{u_2}\rangle_\Omega^{(-m),\free}\ \le\ \langle\sigma_{u_1}\sigma_{u_2}\rangle_{\Omega'}^{(-m),\free}\ &\le\ \langle\sigma_{u_1}\sigma_{u_2}\rangle_{\C}\ =\ |u_2-u_1|^{-\frac14}
\end{align*}
due to Theorem~\ref{thm:conv-omega-mass} and similar inequalities in discrete. The rotational invariance of the full-plane correlation functions
\[
\langle\sigma_{u_1}\sigma_{u_2}\rangle_\C^{(m)}\ =\ \Xi(|u_2-u_1|,m),\ \ m\in\R,
\]
follows from the rotational invariance of finite-volume correlation functions~\eqref{eq:sigma-sigma-def-mass}, \eqref{eq:sigma-sigma-free-def-mass} and of the corresponding boundary value problems in large discs. Finally, we have $\Xi(r,m)\sim r^{-\frac 14}$ as~$r\to 0$ due to~\eqref{eq:sigma-sigma-norm-mass} and the aforementioned monotonicity with respect to~$\Omega$, which allows to pass to the limit~$\Omega\uparrow\C$ in these asymptotics.

\subsection{Massive model: convergence results} \label{sub:conv-mass} We now discuss generalizations of the results from Section~\ref{sub:conv-crit} to the massive setup. Throughout this section we assume that~$\Omega\subset\C$ is a $C^1$-smooth bounded simply connected domain and~$m<0$.

As in Section~\ref{sub:conv-crit}, let~$v,w$ be distinct inner points of~$\Omega$; recall that we use the same notation for their discrete approximations. Also, let~$u=u^\delta\in\Omega^{\circ,\delta}$ be such that~$u\sim v=v^\delta\in \Omega^{\bullet,\delta}$ and~$b=b^\delta\in\Omega^{\bullet,\delta}$ be such that~$b\sim w=w^\delta\in \Omega^{\circ,\delta}$. Consider the normalized real-valued fermionic observable
\begin{equation}
\label{eq:Xvw-def-mass}
X^{(m),\delta}_{[\Omega^\delta;v,w]}(c)\ :=\ \frac{\E^{(m),\wired}_{\Omega^\delta}[\chi_{c}\mu_{v}\sigma_{w}]}{\E^{(m),\wired}_{\Omega^\delta}[\sigma_{u}\sigma_{w}]}, \quad c\in\Upsilon^\times_{[v,w]}(\Omega^\delta),
\end{equation}
and let~$F^{(m),\delta}_{[\Omega^\delta;v,w]}$ be the corresponding massive s-holomorphic spinor in~$\Omega^\delta$ branching over~$v$ and~$w$; see Definition~\ref{def:shol-def-mass}.

\begin{theorem} \label{thm:conv-Fwv-mass} Let~$m<0$ and~$\Omega$ be a $C^1$-smooth bounded simply connected domain. For each $v,w\in\Omega$, $v\ne w$, the following holds uniformly on compact subsets of~$\Omega\smallsetminus\{v,w\}$:
\begin{equation}
\label{eq:fvw-conv-mass}
F^{(m),\delta}_{[\Omega^\delta;v,w]}(z)\ \to\ (\tfrac2\pi)^{\frac 12}f^{(m)}_{[\Omega;v,w]}(z)\ \ \text{as}\ \ \delta\to 0,
\end{equation}
where a massive holomorphic spinor~$f^{(m)}_{[\Omega;v,w]}$ solves the boundary value problem described in Section~\ref{sub:spin-mass-def}. The convergence is also uniform with respect to~$v,w$ provided that $v,w,z$ stay at definite distance from each other and from~$\partial\Omega$.
Moreover,
\begin{equation}
\label{eq:conv-to-B-mass}
\frac{\E^{(m),\wired}_{\Omega^\delta}[\,\mu_{v}\mu_b\,]}{\E^{(m),\wired}_{\Omega^\delta}[\sigma_{u}\sigma_{w}]}\ =\ \frac{\E^{(-m),\free}_{\Omega^{*,\delta}}[\sigma_{v}\sigma_b]}{\E^{(m),\wired}_{\Omega^\delta}[\sigma_{u}\sigma_{w}]}\ \to\ \cB^{(m)}_\Omega(v,w)\quad \text{as}\ \ \delta\to 0.
\end{equation}
\end{theorem}
\begin{proof}
The proof of~\eqref{eq:fvw-conv-mass} repeats the proof of Theorem~\ref{thm:conv-Fwv-crit}, including that of Proposition~\ref{prop:bc-are-OK}, which goes through for $C^1$-smooth domains and $m<0$ as pointed out by Park in~\cite[Proposition~22]{park2018massive} in the square grid context $\Lambda^\delta=\delta\mathbb{Z}^2$. Working with irregular isoradial grids $\Lambda^\delta$ instead, it is worth noting that
\begin{itemize}
\item the a priori regularity estimates of massive s-holomorphic functions~$F^{(m),\delta}$ via the corresponding functions~$H^\delta$ are given by Proposition~\ref{prop:regularity};
\item sub-sequential limits of massive s-holomorphic functions satisfy the massive holomorphicity equation~\eqref{eq:mass-hol} due to Corollary~\ref{cor:mass-hol};
\item the proof of Proposition~\ref{prop:bc-are-OK} relies upon
\begin{itemize}
\item the sub-harmonicity of the function $H^{\bullet,\delta}$ discussed in Proposition~\ref{prop:H-sub} (see also Remark~\ref{rem:bdry-trick});
\item the maximum principle for the function~$|F^{(m),\delta}|$, which holds true (up to a multiplicative constant) due to Lemma~\ref{lem:max-principle}.
\end{itemize}
\end{itemize}
The only small difference with the critical case~$m=0$ is the analysis near the branching point~$w$ and the proof of~\eqref{eq:conv-to-B-mass}. Recall that the asymptotics of explicit full-plane kernels~$\cG^{(m),\delta}_{[w]}$ are given in Theorem~\ref{thm:G-asymp-mass}. Let~$c_w$ be adjacent to both~$w$ and~$b$. Then, we can still use Lemma~\ref{lem:XG-monodromy} to see that
\begin{equation}
\label{eq:x-X(cw)=}
{ 2X^{(m),\delta}_{[\Omega^\delta;v,w]}(c_w)\ }=\ \frac12\oint^{[(m),\delta]}_{z:|z-w|=2r_0}\Im\big[F^{(m),\delta}_{[\Omega^\delta;v,w]}(z)\cG^{(m),\delta}_{[b]}(z)dz\big],
\end{equation}
where the discrete contour integral in the right-hand side is understood in the sense of Remark~\ref{rem:H=intImF2}. (In particular, note that one can easily estimate the value $X^{(m),\delta}(c_w)$ via the values of~$F^{(m),\delta}$ near the contour $\{z:|z-w|=2r_0\}$; in other words, the proof of convergence~\eqref{eq:fvw-conv-mass} literally mimics that of Theorem~\ref{thm:conv-Fwv-crit}.)

In order to prove the convergence~\eqref{eq:conv-to-B-mass} note that, for each fixed~$r_0>0$, the discrete contour integral in the identity~\eqref{eq:x-X(cw)=} converges as~$\delta\to 0$:
\begin{align*}
{ X^{(m),\delta}_{[\Omega^\delta;v,w]}(c_w)\ }\ \mathop{\to}\limits_{\delta\to 0}\ \ &{ \frac{1}{2\pi}}\oint_{z:|z-w|=2r_0}\Im\biggl[f^{(m)}_{[\Omega;v,w]}(z)\cdot e^{-i\frac\pi4}\frac{e^{2m|z-w|}}{\sqrt{z-w}}dz\biggr]\\
&= { \frac{1}{2\pi}}\oint_{z:|z-w|=2r_0}\Im\biggl[\biggl(\pm \frac{\cB^{(m)}_\Omega(v,w)}{z-w}+O(1)\biggr)dz\biggr]\\[2pt] &=\ { \pm\cB^{(m)}_\Omega(v,w)}+O(r_0),
\end{align*}
where we used the asymptotic expansion~\eqref{eq:f-near-w-mass}. Though, as compared to the case~$m=0$, we do not immediately rule out a possible additional error term~$O(r_0)$ in the right-hand side, this term -- if exists -- has to be independent of~$\delta$. Choosing first~$r_0$ small enough and then $\delta$ small enough one obtains the convergence
\[
X^{(m),\delta}_{[\Omega^\delta;v,w]}(c_w)\ =\ \frac{\E^{(m),\wired}_{\Omega^\delta}[\,\mu_{v}\mu_b\,]}{\E^{(m),\wired}_{\Omega^\delta}[\sigma_{u}\sigma_{w}]}\ \to\ \cB^{(m)}_\Omega(v,w)\ \ \text{as}\ \ \delta\to 0,
\]
(the $\pm$ sign is fixed by the fact that both sides of~\eqref{eq:conv-to-B-mass} are positive quantities).
\end{proof}
\begin{remark} A careful reader could have noticed that the proof of Theorem~\ref{thm:conv-Fwv-mass} also relies upon the following fact: if $f^{(m)}$ is a massive holomorphic spinor in a punctured vicinity of a branching point~$w$ (or, similarly,~$v$), then the boundedness of the function~$h:=\int\Im[(f^{(m)}(z))^2dz]$ near $w$ implies that~$f^{(m)}(z)=O(|z-w|^{\frac 12})$. To prove this fact, one can, e.\,g., first argue that $f^{(m)}(z)=O(|z-w|^{-\frac12})$ as~$z\to w$ due to standard a priori estimates of~$f^{(m)}$ via $h=O(1)$, and then improve this estimate to~$O(|z-w|^{\frac 12})$ by using Bers' similarity principle (e.\,g., see~\cite[Lemma~14]{park2018massive}).
\end{remark}

Similarly to the critical case, it is not hard to deduce from Theorem~\ref{thm:conv-Fwv-mass} the following analogues of Corollary~\ref{cor:conv-to-A-crit} and Corollary~\ref{cor:u'w'/uw-conv-crit}.
\begin{corollary} \label{cor:conv-to-A-mass}
(i) In the same setup as above, let~$u_{0,1}\in \Gamma^\delta$ be two neighboring `white' vertices adjacent to~$v\in \Gamma^{\bullet,\delta}$. Then, the following asymptotics hold:
\begin{equation}
\label{eq:conv-to-A-mass}
\frac{\E^{(m),\wired}_{\Omega^\delta}[\sigma_{u_1}\sigma_w]}{\E^{(m),\wired}_{\Omega^\delta}[\sigma_{u_0}\sigma_w]}\ =\ 1+\Re\big[(u_1-u_0)\cdot \cA^{(m)}_\Omega(v,w)\big]+o(\delta)\ \ \text{as}\ \ \delta\to 0,
\end{equation}
where the error term is uniform with respect to~$v,w\in\Omega$ provided that they remain at a definite distance from each other and from~$\partial\Omega$.

\smallskip

\noindent (ii) Let $\rho>0$ and $u_1,u_2,u'_1,u'_2\in\Omega$ be such that $|u_2-u_1|\ge \rho$, $|u'_2-u'_1|\ge \rho$ and assume also that all these four points are at least~$\rho$-away from~$\partial\Omega$. Then,
\[
\log \frac{\E^{(m),\wired}_{\Omega^\delta}[\sigma_{u'_1}\sigma_{u'_2}]}{\E^{(m),\wired}_{\Omega^\delta}[\sigma_{u_1}\sigma_{u_2}]}\ \to\ \int_{(u_1,u_2)}^{(u'_1,u'_2)}\Re\big[\,\cA^{(m)}_\Omega(v,w)dv+\cA^{(m)}_\Omega(w,v)dw\,\big] \ \ \text{as}\ \ \delta\to 0,
\]
where the integral can be computed along any smooth path $(v_t,w_t)\in\Omega\times\Omega$ such that, for all~$t$, the points~$v_t,w_t$ stay at distance at least~$\rho$ from each other and from~$\partial\Omega$.
\end{corollary}
\begin{proof} (i) Similarly to the proof of Corollary~\ref{cor:conv-to-A-crit} and the proof of Theorem~\ref{thm:conv-Fwv-mass}(ii), for each~$r_0>0$ we can write the identity
\begin{align*}
\frac{\E^{(m),\wired}_{\Omega^\delta}[\sigma_{u_1}\sigma_w]}{\E^{(m),\wired}_{\Omega^\delta}[\sigma_{u_0}\sigma_w]}-1\ &=\ X^{(m),\delta}_{[\Omega^\delta;v,w]}(c_1)-X^{(m),\delta}_{[\Omega^\delta;v,w]}(c_0)\\
& =\ { \frac{1}{4}}\oint_{z:|z-v|=2r_0}^{[(m),\delta]}\Im \big[ F^{(m),\delta}_{[\Omega^\delta;v,w]}(z)\cdot (\cG^{(m),\delta}_{[u_1]}(z)-\cG^{(m),\delta}_{[u_0]}(z))dz\big]\,,
\end{align*}
where~$c_{0,1}$ is adjacent to both~$v\in\Omega^{\bullet,\delta}$ and~$u_{0,1}\in\Omega^{\circ,\delta}$, and the discrete contour integral in the right-hand side is understood in the sense of Remark~\ref{rem:H=intImF2}.
Since~$R(z,u)$ in the asymptotics given in Theorem~\ref{thm:G-asymp-mass} is a Lipshitz function of the second argument, for each fixed~$r_0>0$ we have a (uniform in~$z$) asymptotics
\begin{align*}
\cG^{(m),\delta}_{[u_1]}(z)\ -\ \cG^{(m),\delta}_{[u_0]}(z)\ =\ e^{i\frac\pi 4}{ \biggl(\frac{2}{\pi}\biggr)^{\!\!\frac{1}{2}}} \biggl[\frac{e^{-2m|z-u_1|}}{\sqrt{z-u_1}}-\frac{e^{-2m|z-u_0|}}{\sqrt{z-u_0}}\biggr]&+{ O(\delta^2)}\\
=\ e^{i\frac\pi4}{ \biggl(\frac{2}{\pi}\biggr)^{\!\!\frac{1}{2}}}\frac{e^{-2m|z-v|}}{\sqrt{z-v}}\cdot\biggl[\frac{u_1-u_0}{2(z-v)}+ \frac{2m\Re[(z-v)(\overline{u}_1-\overline{u}_0)]}{|z-v|}\biggr]&+{ O(\delta^2)}\,.
\end{align*}
At the same time, Theorem~\ref{thm:conv-Fwv-mass}(i) and the definition~\eqref{eq:cA-def-mass} 
imply that
\[
F^\delta_{[\Omega^\delta;v,w]}(z)\ =\ e^{-i\frac{\pi}{4}}\biggl(\frac{2}{\pi}\biggr)^{\!\!\frac{1}{2}}\frac{e^{2m|z-v|}}{\sqrt{z-v}}\cdot \big[1+2\cA^{(m)}_\Omega(v,w)(z-v)+O(r_0^2)\big]+o_{\delta\to 0}(1),
\]
uniformly in~$z$, where the error term~$O(r_0^2)$ does not depend on~$\delta$.

It is easy to see that the contour integrals $\oint\ldots dz$ of functions $(z-v)^{-2}$, $|z-v|^{-1}$ and $|z-v|(z-v)^{-2}$ along a circle centered at~$v$ vanish. Therefore,
\begin{align*}
\delta^{-1}\cdot\biggl[&\frac{\E^\wired_{\Omega^\delta}[\sigma_{u_1}\sigma_w]}{\E^\wired_{\Omega^\delta}[\sigma_{u_0}\sigma_w]}-1\biggr]\\
 &=\ \frac{1}{2\pi}\oint_{z:|z-v|=2r_0}\Im\biggl[\biggl(\frac{\delta^{-1}(u_1\!-\!u_0)\cA^{(m)}_\Omega(v,w)}{z-v}+O(1)\biggr)dz\biggr]+o_{\delta\to 0}(1)\\[2pt]
&=\ \Re\big[\delta^{-1}(u_1\!-\!u_0)\cA^{(m)}_\Omega(v,w)\big]+o_{\delta\to 0}(1)
\end{align*}
since the~$O(1)$ term in the integrand gives a contribution~$O(r_0)$, which does not depend on~$\delta$ and can be made as small as needed before choosing~$\delta$ small enough.

\smallskip

\noindent (ii) The proof repeats the proof of Corollary~\ref{cor:u'w'/uw-conv-crit} and boils down to multiplying the asymptotics~\eqref{eq:conv-to-A-mass} along discrete paths going from $(u_1,u_2)$ to $(u_1',u_2')$.
\end{proof}

The next theorem, in particular, provides an analogue of Corollary~\ref{cor:conv-omega-crit} in the massive case. Note that we change the order of statements as compared to Section~\ref{sub:conv-crit} and prove this result \emph{before} the analogue of Theorem~\ref{thm:spin-univ-crit}. This shortcut is possible due to the fact that we can now use the critical model in order to control the multiplicative normalization of spin-spin correlations in~$\Omega^\delta$ and do not need to consider the full-plane limit first (as it was in the proof of Corollary~\ref{cor:conv-omega-crit}).

\begin{theorem} \label{thm:conv-omega-mass} (i) Let~$m<0$ and~$\Omega$ be a $C^1$-smooth bounded simply connected domain. Then,
one can define a function~$\langle\sigma_{u_1}\sigma_{u_2}\rangle_\Omega^{(m),\wired}$ according to~\eqref{eq:cA-def-mass} such that the asymptotics~$\langle\sigma_{u_1}\sigma_{u_2}\rangle_\Omega^{(m),\wired}\sim |u_2-u_1|^{-\frac 14}$ as~$u_2\to u_1$ holds for each~$u_1\in\Omega$.

\smallskip

\noindent (ii) Let discrete domains~$\Omega^\delta\subset\Gamma^\delta$ approximate~$\Omega$ in the Hausdorff sense so that $\mathrm{dist}(\partial\Omega^\delta;\partial\Omega)=O(\delta)$. Then, for each~$u,w\in\Omega$, $u_1\ne u_2$, one has the convergence
\[
\begin{array}{lcl}
\delta^{-\frac14}\E_{\Omega^\delta}^{(m),\wired}[\sigma_u\sigma_w]&\to&\mathcal{C}_\sigma^2\cdot\langle\sigma_u\sigma_w\rangle^{(m),\wired}_{\Omega}\,,\\[2pt] \delta^{-\frac14}\E_{\Omega^\delta}^{(-m),\free}[\sigma_u\sigma_w]&\to& \mathcal{C}_\sigma^2\cdot\langle\sigma_u\sigma_w\rangle^{(-m),\free}_{\Omega}
\end{array}\quad \text{as}\ \ \delta\to 0,
\]
where the limits (continuous correlation functions) are defined by~\eqref{eq:sigma-sigma-def-mass}--\eqref{eq:sigma-sigma-norm-mass}, and the universal constant~$\mathcal{C}_\sigma=2^{\frac 16}e^{\frac 32\zeta'(-1)}$ does not depend neither on~$m$ nor on~$\Lambda^\delta$.
\end{theorem}
\begin{proof} (i) It immediately follows from Corollary~\ref{cor:conv-to-A-mass} that the differential form
\[
\cL_\Omega^{(m)}(v,w)\ :=\ \Re\big[\cA^{(m)}_\Omega(v,w)dv+\cA^{(m)}_\Omega(w,v)dw\big],
\]
defined on the set~$\{(v,w)\in\Omega\times\Omega:v\ne w\}$, is exact. Now note that
\begin{equation}
\label{eq:x-m-vs-0}
1\,\le\,\frac{\E_{\Omega^\delta}^{(m),\wired}[\sigma_{u_1}\sigma_{u_2}]}{\E_{\Omega^\delta}^{\wired}[\sigma_{u_1}\sigma_{u_2}]}\,\le\,
\frac{\E_{\Omega^\delta}^{(m),\wired}[\sigma_{u_1}\sigma_{u_2}]}{\E_{\Omega^{*,\delta}}^{(-m),\free}[\sigma_{v_1}\sigma_{v_2}]}\cdot \frac{\E_{\Omega^{*,\delta}}^{\free}[\sigma_{v_1}\sigma_{v_2}]}{\E_{\Omega^\delta}^{\wired}[\sigma_{u_1}\sigma_{u_2}]}\, \mathop{\to}\limits_{\delta\to 0}\,\frac{\cB_\Omega(u_1,u_2)}{\cB^{(m)}_\Omega(u_1,u_2)}
\end{equation}
due to the monotonicity of spin-spin correlations with respect to the interaction parameters, Theorem~\ref{thm:conv-Fwv-mass}(ii) and Corollary~\ref{cor:conv-to-B-crit}. Together with Lemma~\ref{lem:B->1} this gives the estimate
\[
\limsup_{\delta\to0}\biggl|\,\log \frac{\E_{\Omega^\delta}^{(m),\wired}[\sigma_{u'_1}\sigma_{u'_2}]}{\E_{\Omega^\delta}^{(m),\wired}[\sigma_{u_1}\sigma_{u_2}]} -\log \frac{\E_{\Omega^\delta}^{\wired}[\sigma_{u'_1}\sigma_{u'_2}]}{\E_{\Omega^\delta}^{\wired}[\sigma_{u_1}\sigma_{u_2}]}\,\biggr|\ =\ o_{u_2\to u_1}(1)+o_{u'_2\to u'_1}(1)
\]
and hence, by passing to the limit~$\delta\to 0$ and applying Corollary~\ref{cor:conv-to-A-mass}(ii) and Corollary~\ref{cor:conv-to-A-crit}, we have
\[
\int_{(u_1,u_2)}^{(u_1',u_2')} \big(\cL_\Omega^{(m)}(v,w)-\cL_\Omega(v,w)\big)\ =\ o_{u_2\to u_1}(1)+o_{u'_2\to u'_1}(1),
\]
where~$\cL_\Omega(v,w):=\cL_\Omega^{(0)}(v,w)=d\langle\sigma_v\sigma_w\rangle_\Omega^\wired$. Therefore, one can choose a primitive
\[
\textstyle \cR_\Omega^{(m)}\ :=\ { \exp}\int\big(\cL_\Omega^{(m)}-\cL_\Omega\big)
\]
so that it is continuous up to the diagonal~$u_1=u_2$ and, moreover,~$\cR_\Omega^{(m)}(u,u)=1$ for all~$u\in\Omega$. We can now define~$\langle\sigma_{u_1}\sigma_{u_2}\rangle_\Omega^{(m),\wired}:=\cR_\Omega^{(m)}(u_1,u_2)\cdot \langle\sigma_{u_1}\sigma_{u_2}\rangle_\Omega^\wired$.

\smallskip

\noindent (ii) Given~$u,w\in\Omega$, let us also consider another pair points~$u',w'\in\Omega$ such that the distance~$|w'-u'|$ is small. Combining Corollary~\ref{cor:conv-to-A-mass} and the estimate~\eqref{eq:x-m-vs-0} applied to~$u'$ and~$w'$ we see that
\begin{align*}
\frac{\langle\sigma_u\sigma_w\rangle^{(m),\wired}_{\Omega}}{\langle\sigma_{u'}\sigma_{w'}\rangle^{(m),\wired}_{\Omega}}\ &\le\ \liminf_{\delta\to 0}\frac{\E^{(m),\wired}_{\Omega^\delta}[\sigma_u\sigma_w]}{\E^\wired_{\Omega^\delta}[\sigma_{u'}\sigma_{w'}]}\\
&\le\ \limsup_{\delta\to 0}\frac{\E^{(m),\wired}_{\Omega^\delta}[\sigma_u\sigma_w]}{\E^\wired_{\Omega^\delta}[\sigma_{u'}\sigma_{w'}]}\ \le\ \frac{\langle\sigma_u\sigma_w\rangle^{(m),\wired}_{\Omega}}{\langle\sigma_{u'}\sigma_{w'}\rangle^{(m),\wired}_{\Omega}}\cdot (1+o_{w'\to u'}(1))
\end{align*}
Using the convergence $\delta^{-\frac 14}\E^\wired_{\Omega^\delta}[\sigma_{u'}\sigma_{w'}]\to \mathcal{C}_\sigma^2\cdot \langle\sigma_{u'}\sigma_{w'}\rangle_\Omega^\wired$ (see Corollary~\ref{cor:conv-omega-crit}) and the fact that~$\langle\sigma_{u'}\sigma_{w'}\rangle_\Omega^{(m),\wired}\sim \langle\sigma_{u'}\sigma_{w'}\rangle_\Omega^\wired$ as~$w'\to u'$, the previous estimate can be written as
\begin{align*}
\mathcal{C}_\sigma^2\cdot \langle\sigma_u\sigma_w\rangle^{(m),\wired}_{\Omega}\cdot (1-o_{w'\to u'}(1))\ &\le\ \liminf_{\delta\to 0}\delta^{-\frac14}{\E^{(m),\wired}_{\Omega^\delta}[\sigma_u\sigma_w]}\\
\le\ \limsup_{\delta\to 0}\delta^{-\frac14}{\E^{(m),\wired}_{\Omega^\delta}[\sigma_u\sigma_w]}\ &\le\ \mathcal{C}_\sigma^2\cdot \langle\sigma_u\sigma_w\rangle^{(m),\wired}_{\Omega}\cdot (1+o_{w'\to u'}(1)).
\end{align*}
By choosing first~$w'$ close enough to~$u'$ and then~$\delta$ small enough this implies the convergence~$\delta^{-\frac14}{\E^{(m),\wired}_{\Omega^\delta}[\sigma_u\sigma_w]}\to\mathcal{C}_\sigma^2\cdot \langle\sigma_u\sigma_w\rangle^{(m),\wired}_{\Omega}$.

Finally, a similar result for~$\delta^{-\frac14}{\E^{(-m),\free}_{\Omega^\delta}[\sigma_u\sigma_w]}$ follows from the Kramers--Wannier duality, convergence~\eqref{eq:conv-to-B-mass}, and the convergence of spin-spin correlations in the dual isoradial model with wired boundary conditions.
\end{proof}

We conclude this section by deducing the convergence of the spin-spin correlations in the full plane from Theorem~\ref{thm:conv-omega-mass}.

\begin{corollary}\label{cor:spin-univ-mass} Let~$u,w\in\C$ and~$m\in\R\smallsetminus\{0\}$. Then, uniformly with respect to isoradial grids $\Lambda^\delta$ satisfying the uniformly bounded angles property~$\BAP$, we have the convergence
\[
\delta^{-\frac14 }\E_{\Lambda^\delta}^{(m)}[\sigma_u\sigma_w]\ \to\ \mathcal{C}_\sigma^2\cdot \Xi(|u-w|,m)\ \ \text{as}\ \ \delta\to 0,
\]
where
\[
\Xi(|u-w|,m)\ =\ \langle\sigma_u\sigma_w\rangle^{(m)}_\C\ =\ \begin{cases}\lim_{\Omega\uparrow\C}\langle\sigma_u\sigma_w\rangle^{(m),\wired}_\Omega &  \text{if~$m<0$},\\[2pt]
\lim_{\Omega\uparrow\C}\langle\sigma_u\sigma_w\rangle^{(m),\free}_\Omega & \text{if~$m>0$}.
\end{cases}
\]
\end{corollary}
\begin{proof} Let~$B^\delta_R\subset\Lambda^\delta$ denote the discretization of the disc~$B_R:=B(\frac{1}{2}(u+w),R)\subset\C$. For~$m<0$, it follows from the RSW estimates~\eqref{eq:RSW-sigma-sigma} that
\[
\E^{(m)}_{\Lambda^\delta}[\sigma_u\sigma_w]\ =\ \E^{(m),\wired}_{B^\delta_R}[\sigma_u\sigma_w]\cdot (1+o_{R\to\infty}(1))\ \ \text{as}\ \ R\to\infty,
\]
uniformly in~$\delta$. By definition of the infinite-volume correlations in continuum,
\[
\langle\sigma_u\sigma_w\rangle^{(m)}_\C\ =\ \langle\sigma_u\sigma_w\rangle^{(m)}_{B_R}\cdot (1+o_{R\to\infty}(1)).
\]
Finally, for each fixed~$R$, Theorem~\ref{thm:conv-omega-mass}(ii) provides the asymptotics
\[
\delta^{-\frac 14}\E^{(m),\wired}_{B^\delta_R}[\sigma_u\sigma_w]\ =\ \mathcal{C}_\sigma^2\cdot\langle\sigma_u\sigma_w\rangle^{(m)}_{B_R}\cdot (1+o_{\delta\to 0}(1)).
\]
Choosing first~$R$ big enough and then $\delta$ small enough we obtain the required convergence of the infinite-volume correlations in the case~$m<0$.

The proof for~$m>0$ is similar and relies upon the convergence of the finite-volume correlations
$\delta^{-\frac 14}\E^{(m),\free}_{B^\delta_R}[\sigma_u\sigma_w]$, which is also given by Theorem~\ref{thm:conv-omega-mass}(ii).
\end{proof}

\section{Construction and asymptotic analysis of full-plane kernels}\label{sec:asymptotics}

In this section, we construct and analyze massive s-holomorphic functions
\begin{itemize}
\item $\cF_\rr,\cF_\ri$ (discrete analogues of~$e^{-2m\Im z}$ and~$ie^{2m\Im z}$; see Theorem~\ref{thm:F1Fi-asymp}), which were used to construct an s-embedding in Section~\ref{sub:s-emb};

\smallskip

\item $\cG_{[v]}$, $\cG_{[u]}$ (discrete analogues of~$e^{\mp i\pi/4}\cdot z^{-1/2}e^{\pm 2m|z|}$; see Theorem~\ref{thm:G-asymp-mass}), which are the main tool used in our paper;

\smallskip

\item massive Cauchy kernels~$\cG_{(a)}$, where~$a\in\Upsilon^\times$ is a (lift onto~$\Upsilon^\times$ of a) given edge of~$\Lambda$; see Section~\ref{sub: def_kernels} and, in particular, Remark~\ref{rem: Gaa-def} for more details.
\end{itemize}
Although not strictly necessary for the present paper, the latter kernel~$\cG_{(a)}$ can be used to establish the regularity of massive s-holomorphic functions (see \cite{Park2021Fermionic} and the proof of Proposition~\ref{prop:regularity}) and also to prove the convergence of energy correlations; we thus include its construction and analysis for reference purposes.

For shortness,
\emph{from now onwards we omit the superscripts~$(m),\delta$ in the notation.}
Also, in this section we prefer to work with $\qtoem>0$ in order to keep the moduli $k,K(k),K'(k)$ real; for $q<0$ it suffices to apply the Kramers--Wannier duality.

{
\begin{remark} \label{rem:duality-asymp}
Recall that this duality amounts to exchanging the lattices \mbox{$\Gamma^{\bullet,\delta}\leftrightarrow\Gamma^{\circ,\delta}$} and, simultaneously, changing the sign of~$q$ and~$m$. In order to keep the definition of~$\eta_c$ and of massive s-holomorphic functions invariant under this procedure, one also needs to simultaneously replace the global prefactor $\varsigma$ in~\eqref{eq:eta-c-def} by~$\pm i\varsigma$; note that this also leaves the equation \mbox{$\overline\partial f+ \varsigma^2m\overline{f}=0$} unchanged. In order to keep this dependence on~$\varsigma$ transparent, we do \emph{not} rely upon the explicit convention~$\varsigma=e^{i\frac\pi 4}$ in what follows and formulate all results in a slightly more invariant way.
\end{remark}}

The results given below are formulated in terms of real-valued spinors $\cX_\rr$, $\cX_\ri$, $\rG_{[v]}$, $\rG_{[u]}$, $\rG_{(a)}$ satisfying the three-terms identity~\eqref{eq:3-terms} rather than in terms of massive s-holomorphic functions~$\cF_\rr$, $\cF_\ri$, $\cG_{[v]}$, $\cG_{[u]}$, $\cG_{(a)}$ themselves; recall that the correspondence between the two is provided by Definition~\ref{def:shol-def-mass}.

\subsection{Discrete exponentials}\label{sub:discrete-exp}

We heavily rely upon the existence
of particular solutions to the three-terms equation~\eqref{eq:3-terms} on isoradial
graphs, the \emph{discrete exponentials}. At criticality, they were first
introduced by Mercat~\cite{mercat2001discrete} and Kenyon~\cite{kenyon2002laplacian}, and for~$q\ne 0$ by Boutillier, de Tili\`ere and Raschel~\cite{boutillier2017z}.
In this section, we review their construction, with slight modifications
made in order to fit our setup.

Given~$c\in\Upsilon^\times$ (i.\,e., a lift onto the double cover of the mid-point of an edge $(u(c)v(c))$ of the rhombic lattice~$\Lambda$ with~$u(c)\in\Gamma^\circ$ and $v(c)\in\Gamma=\Gamma^\bullet$), let
\[
\ea{\alpha}_c\ :=\ \arg(v(c)-u(c)).
\]
Note that these angles are typically denoted by~$\overline{\alpha}_c$ in~\cite{boutillier2017z,boutillier2019z,deTiliere-Zinv}, similarly to the notation~$\overline\theta_e$ for the half-angles of rhombi that we used above. However, we prefer to simply write~$\alpha_c$ and~$\theta_e:=\overline{\theta}_e$ throughout this section in order not to create a confusion with the complex conjugation. At the same time, we will use the notation
\begin{equation}
\label{eq:thetae-thetag}
\alphae_c:=\tfrac{2K}{\pi}\alpha_c\qquad \text{and}\qquad \thetae_e:=\tfrac{2K}{\pi}\theta_e
\end{equation}
for the ``elliptic'' quantities that are denoted by~$\alpha_c$ and~$\theta_e$ in~\cite{boutillier2017z,boutillier2019z,deTiliere-Zinv}.

We start by recalling the construction of discrete exponentials in the \emph{critical} Ising model, i.\,e., if~$m=0$ and~$\thetaa_e=\theta_e=\thetae_e$. Given~$\lambda\in\C\smallsetminus\{0\}$, we set
\begin{align}
\dexp{\lambda}{v(c)}c
\ :&=\ { \overline{\varsigma}}\eta_{c}\cdot(2\lambda)^{-1/2}\cdot(1+\lambda\cdot(v(c)-c)) \notag\\
&=\ (2e^{i\ea{\alpha}_{c}}\lambda)^{-1/2}\cdot(1+\tfrac{\lambda}{2}(v(c)-u(c))),\label{eq: defecrit1}\\
\dexp{\lambda}{u(c)}c
\ :&=\ {\overline\varsigma}\eta_{c}\cdot(2\lambda)^{-1/2}\cdot(1+\lambda\cdot(u(c)-c)) \notag \\
&=\ (2e^{i\ea{\alpha}_{c}}\lambda)^{-1/2}\cdot(1-\tfrac{\lambda}{2}(v(c)-u(c))),\label{eq: defecrit2}
\end{align}
and $\dexp{\lambda}c{v(c)}:=\dexp{\lambda}{v(c)}c^{-1}$, $\dexp{\lambda}c{u(c)}:=\dexp{\lambda}{u(c)}c^{-1}$, where the additional factors~${ \overline\varsigma}\eta_{c}$ and $(2\lambda)^{-1/2}$ are introduced to fit the forthcoming definition outside of criticality; see~(\ref{eq: defexpk}--\ref{eq: defexpk2}) and Remark~\ref{rem: exp_critical_massive}.
We then define
\[
\dexp{\lambda}x{x_{0}}:=\prod_{k=1}^{N}\dexp{\lambda}{x_{k}}{x_{k-1}},\quad x,x_{0}\in\Lambda\cup\Upsilon^{\times},
\]
the product is taken over an arbitrary path $x_{0}\sim x_{1}\sim\dots\sim x_{N}=x$, where $x_{k}$
and $x_{k-1}$ are adjacent points of $\Upsilon^{\times}$ and $\Lambda$ (or vice versa).
Note that, using the path of just two steps, we get
\[
\dexp{\lambda}{v(c)}{u(c)}\ =\ \frac{1+\tfrac{\lambda}{2}(v(c)-u(c))}{1-\tfrac{\lambda}{2}(v(c)-u(c))}\,,
\]
recovering the usual definition of discrete exponentials on $\Lambda$
as in \cite{kenyon2002laplacian}. This also shows that $\dexp{\lambda}x{x_{0}}$
does not depend on the path chosen. In particular,
\begin{itemize}
\item for fixed~$\lambda$ and~$x_0$, the discrete exponential $\dexp\lambda{\cdot}{x_0}$ is a well-defined function on~$\Lambda$ and a well-defined spinor on~$\Upsilon^\times$;
\item for fixed~$x,x_0$, the discrete exponential $\dexp{\cdot}x{x_0}$ is a well-defined function of $\lambda\in\C\smallsetminus\{0\}$ if both~$x,x_{0}\in\Lambda$ or both~$x,x_{0}\in\Upsilon^{\times}$, and is a spinor branching over $\lambda=0$ if one of~$x,x_{0}$ belongs to~$\Lambda$ and the other to~$\Upsilon^{\times}$.
\end{itemize}

\begin{lemma}
\noindent \label{lem: prop_exp_critical}For each~$x_{0}\in\Lambda\cup\Upsilon^{\times}$
and~$\lambda\in\C\smallsetminus\{0\}$, the complex-valued spinor $\dexp{\lambda}{\cdot}{x_{0}}:\Upsilon^{\times}\to\C$
satisfies the three-terms identity~\eqref{eq:3-terms} with~$\thetaa_e={ \theta_e}=\thetae_e$.
\end{lemma}

\begin{proof}
\noindent Denoting $X_{pq}:=\dexp{\lambda}{c_{pq}}{x_{0}}$ and $\eang{pq}:=\eang{c_{pq}}$ (see Fig.~\ref{fig:notation}),
we have
\[
X_{01}=X_{00}\cdot\frac{e^{-\frac{i\ea{\alpha}_\scc{00}}{2}}\cdot(1+\tfrac{\lambda\delta}{2}e^{i\ea{\alpha}_{00}})} {e^{-\frac{i\ea{\alpha}_{01}}{2}}\cdot(1+\tfrac{\lambda\delta}{2}e^{i\ea{\alpha}_{01}})};\quad X_{10}=X_{00}\cdot\frac{e^{-\frac{i\ea{\alpha}_{00}}{2}}\cdot(1-\tfrac{\lambda\delta}{2}e^{i\ea{\alpha}_{00}})} {e^{-\frac{i\ea{\alpha}_{10}}{2}}\cdot(1-\tfrac{\lambda\delta}{2}e^{i\ea{\alpha}_{10}})}.
\]
Taking into account that $\alpha_{01}=\alpha_{00}-2{ \theta_e}$ and $\alpha_{10}=\alpha_{00}+\pi-2{ \theta_e}$, the required identity~\eqref{eq:3-terms} boils down to an elementary identity
\[
1+\tfrac{\lambda\delta}{2}e^{i(\eang{00}-2{ \theta_e})}\ =\ e^{-i{ \theta_e}}\cdot(1+\tfrac{\lambda\delta}{2}e^{i\eang{00}})\cdot\cos{ \theta_e} \ +\ ie^{-i{ \theta_e}}\cdot(1-\tfrac{\lambda\delta}{2}e^{i\eang{00}})\cdot\sin{ \theta_e},
\]
which is straightforward to check.
\end{proof}

We now proceed to defining the discrete exponentials { outside} criticality, following~\cite{boutillier2017z,boutillier2019z,deTiliere-Zinv}; recall that in this section we assume that~$q>0$ and hence~$k,k'\in(0,1)$ and the complete elliptic integrals of the first kind~$K(k),K'(k)$ are real; the opposite case follows from the Kramers--Wannier duality.

Throughout this section to each ``elliptic'' variable (e.\,g., $\ellip{\nu}\in\T(k)$)
that can be an argument of an elliptic function corresponds a ``Euclidean''
variable $\ea{\nu}$ that can be plugged into a trigonometric function and vice versa;
the relation is (cf.~\eqref{eq:thetae-thetag})
\[
\textstyle \ellip{\nu}\ =\ \frac{2K}{\pi}\ea{\nu}.
\]

Given a value~$\frac{1}{2}\ellip{\mu}\in\T(k):=\C/(4K\Z+4iK'\Z)$ and $c\in\Upsilon^{\times}$, define
\begin{align}
\dexpk k{\ellip{\mu}}{v(c)}c\  & :=\ { (k')^{\frac14}}\cdot\sd{\tfrac{1}{2}(\ellip{\mu}-\ellipa c)}k,\label{eq: defexpk}\\
\dexpk k{\ellip{\mu}}{u(c)}c\  & :=\ { -i(k')^{-\frac14}}\cdot \cd{\tfrac{1}{2}(\ellip{\mu}-\ellipa c)}k.\label{eq: defexpk2}
\end{align}
Recall that~$\sd{\ellip w+2K}k=-\sd{\ellip w}k$ and $\cd{\ellip w+2K}k=-\cd{\ellip w}k$,
hence the quantities~(\ref{eq: defexpk}--\ref{eq: defexpk2}) change the sign under the transform $\ea{\alpha}_{c}\mapsto\ea{\alpha}_{c}+2\pi$. This is why they are defined for $c\in\Upsilon^{\times}$ (and not for $c\in\Upsilon$), similarly to~(\ref{eq: defecrit1}--\ref{eq: defecrit2}). { Let us also emphasize the fact that the functions~(\ref{eq: defexpk}--\ref{eq: defexpk2}) of~$\frac{1}{2}\mu\in\T(k)$ are \emph{not} well-defined on the torus~$\mu\in\T(k)$ itself, rather being \emph{spinors} on~$\T(k)$; see~(\ref{eq: exp_per_1}--\ref{eq: exp_per_3}).}

As in the critical case, we extend the definition of $\dexpk k{\ellip{\mu}}x{x_{0}}$
to arbitrary $x,x_{0}\in\Lambda\cup\Upsilon^{\times}$ by multiplying
along paths. In particular, this gives
\begin{equation}
\dexpk k{\ellip{\mu}}{v(c)}{u(c)}\ =\ i\sqrt{k'}\cdot\sc{\tfrac{1}{2}(\ellip{\mu}-\ellipa c)}k.\label{eq: exp_vu}
\end{equation}
as in \cite[Eq.15]{boutillier2017z}. Since $\sc{w+2K}k=\sc wk$,
the latter expression is actually independent on the choice of the lift of $c$ onto~$\Upsilon^\times$.
Moreover, the identity
\[
\sc{\ellip w+K}k\sc{\ellip w}k\,=\,{ -(k')^{-1}}
\]
(see \cite[Eq.~22.4.3]{NIST:DLMF}) guarantees that multiplying (\ref{eq: exp_vu}) around a quad
yields $1$. This proves that $\dexpk k{\ellip{\mu}}x{x_{0}}$ is independent
of the choice of the path.

\smallskip{}

For a fixed $x_{0}$ and a fixed $c\in\Upsilon^{\times}$, the periodicity
properties of the discrete exponential $\mathbf{e}(\cdot):=\dexpk k{\cdot}c{x_{0}}$,
easily read off \cite[Eq.~22.4.1]{NIST:DLMF}, are as follows:
\begin{align}
\mathbf{e}(\ellip{\mu}+4K)=\mathbf{e}(\ellip{\mu})&=  \mathbf{e}(\ellip{\mu}+4iK')\hskip 8pt \qquad \text{if }x_{0}=c_0\in\Upsilon^{\times};\label{eq: exp_per_1}\\
-\mathbf{e}(\ellip{\mu}+4K)=\mathbf{e}(\ellip{\mu})&= \mathbf{e}(\ellip{\mu}+4iK')\hskip 8pt \qquad \text{if }x_{0}=u\in\Gamma^{\circ};\label{eq: exp_per_2}\\
-\mathbf{e}(\ellip{\mu}+4K)=\mathbf{e}(\ellip{\mu})&=  -\mathbf{e}(\ellip{\mu}+4iK') \qquad \text{if }x_{0}=v\in\Gamma^{\bullet}.\label{eq: exp_per_3}
\end{align}
\begin{lemma}[{\cite[Proposition 36]{boutillier2019z}}]
\label{lem: prop_exp_massive}
The discrete exponentials $c\mapsto\dexpk k{\ellip{\mu}}c{x_{0}}$
satisfy the propagation equation~\eqref{eq:3-terms} on~$\Upsilon^\times$ provided that~$\sin\thetaa_e=\sc{\thetae_e}k$.
\end{lemma}

\begin{proof}
\noindent Similarly to the proof of Lemma \ref{lem: prop_exp_critical},
this boils down to checking the identity
\[
1\ =\ \frac{\sd{\frac{1}{2}(\ellip{\mu}-\ellip{\alpha}-2\ellip{\theta})}k}{\sd{\frac{1}{2}(\ellip\mu-\ellip\alpha)}k}\cn{\ellip\theta}k+\ \frac{\cd{\frac{1}{2}(\ellip\mu-\ellip\alpha-2\ellip\theta)}k}{\cd{\frac{1}{2}(\ellip\mu-\ellip\alpha-2\ellip K)}k}\sn{\ellip\theta}k.
\]
Denote~$\ellip w:=\frac{1}{2}(\ellip \mu-\ellip\alpha)$. Since $\cd{\ellip{w}-K}k=\sn {\ellip{w}}k$,
this is equivalent to
\[
\sn {\ellip{w}}k\ =\ \sd{{\ellip{w}}-\ellip\theta}k\cn{\ellip\theta}k\dn {\ellip{w}}k+\cd{{\ellip{w}}-\ellip\theta}k\sn{\ellip\theta}k.
\]
The latter identity follows from the fact that the two sides are co-periodic
(see \cite[Eq.~22.4.1]{NIST:DLMF}), and have the same poles and residues.
Namely, by \cite[Eq.~22.5.1]{NIST:DLMF}, the LHS has a pole at $iK'$
with residue $\frac{1}{k},$ while the RHS has a pole at $iK'$ with
residue $-i\cn{\theta}k\sd{iK'-\theta}k$, and at $K+iK'+\theta$
with residue
\[
-i(kk')^{-1}\cn{\theta}k\dn{K+iK'+\theta}k-k^{-1}\sn{\theta}k;
\]
the former expression equals to $\frac{1}{k}$ and
the latter to $0$ by \cite[Eq.~22.3.4]{NIST:DLMF}.
\end{proof}
\begin{remark}\label{rem: exp_critical_massive}
Note that our modified definition~(\ref{eq: defecrit1}--\ref{eq: defecrit2}) of discrete exponentials at criticality is nothing but a particular case of the more general elliptic construction~(\ref{eq: defexpk}--\ref{eq: defexpk2}) corresponding to~$k=0$, $k'=1$, $K=\frac{\pi}{2}$, $K'=+\infty$. More precisely, the limiting form of (\ref{eq: defexpk}--\ref{eq: defexpk2}) as~$\delta=1$ and~$k\to 0$ is
\begin{align*}
\dexpk 0{\mu}{v(c)}c\ & =\ \sin\tfrac{1}{2}(\ea{\mu}-\ea{\alpha}_{c})\ =\ \tfrac{1}{2i}e^{-\frac{i}{2}\alpha_{c}}e^{i\frac{\mu}{2}}\cdot (1-e^{-i\mu}e^{i\alpha_c}),\\
\dexpk 0{\mu}{u(c)}c\ & =\ -i\cdot\cos\tfrac{1}{2}(\ea{\mu}-\ea{\alpha}_{c})\ =\ \tfrac{1}{2i}e^{-\frac{i}{2}\alpha_{c}}e^{i\frac{\mu}{2}}\cdot (1+e^{-i\mu}e^{i\alpha_c}),
\end{align*}
which coincides with (\ref{eq: defecrit1}--\ref{eq: defecrit2}) if $e^{-i\mu}=-\frac{1}{2}\lambda$. Thus, we have the identity
\begin{equation}
\label{eq: exp_critical_massive}
\dexpk 0{\ea{\mu}}x{x_{0}}\ =\ \dexp{-2e^{-i\ea{\mu}}}x{x_{0}}\ \ \text{provided that}\ \ \delta=1.
\end{equation}
\end{remark}

\subsection{Definition and basic properties of the full-plane kernels} \label{sub: def_kernels}
In this section we introduce real-valued spinors
\begin{itemize}
\item $\cX_\rr$, $\cX_\ri$ defined on~$\Upsilon^\times$;

\smallskip
\item $\rG_{[u]}$, $\rG_{[v]}$ (where~$u\in\Gamma^\circ$ and~$v\in\Gamma^\bullet$) defined on~$\Upsilon^\times_{[u]}$ and~$\Upsilon^\times_{[v]}$, respectively;

\smallskip
\item and the massive Cauchy kernel~$\rG_{(a)}$, where $a\in\Upsilon^\times$, defined on the double cover~$\Upsilon^\times_{(a)}:=\Upsilon^\times_{[v(a),u(a)]}$, which can be naturally identified with~$\Upsilon^\times$ except the two lifts of the corner~$a$ itself; see { Fig.~\ref{fig:U=Ua} and} Remark~\ref{rem: Gaa-def} below.
\end{itemize}
These spinors satisfy (see Proposition~\ref{prop: kernels_values} below) the propagation equation~\eqref{eq:3-terms}, which allows one to construct massive s-holomorphic functions~$\cF_\rr$, $\cF_\ri$, $\cG_{[u]}$, $\cG_{[v]}$, and $\cG_{(a)}$ out of them via Definition~\ref{def:shol-def-mass}.

\smallskip

For~$w_1,w_2\in\C$ { such that~$w_1\ne w_2$}, denote
\[
\textstyle \ea{\phi}_{w_1w_2}:=\arg(w_{1}-w_{2}),\qquad \ellip{\phi}_{w_{1}w_{2}}:=\frac{2K}{\pi}\cdot\ea{\phi}_{w_{1}w_{2}}.
\]
Also, let~$o\in\Gamma^{\circ,\delta}$ be the closest to the origin vertex of~$\Gamma^{\circ,\delta}$. We define
\begin{align}
{ \delta^{-\frac 12}}\cX_\rr(c)\ & :=\ { -ik^{-1}}\cdot\dexpk k{{ \tfrac{4K}{\pi}\arg\varsigma}+2iK'}c{o}\,,\label{eq:defX1}\\
{ \delta^{-\frac 12}}\cX_\ri(c)\ & :=\ { -ik^{-1}}\cdot \dexpk k{{ \tfrac{4K}{\pi}\arg\varsigma-2K}+2iK'}c{o}\,;\label{eq:defXi}\\
\coruz uc\ & :=\ { -\frac{(k')^{\frac14}}{2\pi}}\int_{\ellip{\phi}_{cu}-2iK'}^{\ellip{\phi}_{cu}+2iK'}\dexpk k{\ellip{\mu}}cu\,d\ellip{\mu}\,,\label{eq: def_chi_mu}\\
\corvz vc\ & :=\ { -\frac{i(k')^{\frac14}}{2\pi}}\int_{\ellip{\phi}_{cv}-2K-2iK'}^{\ellip{\phi}_{cv}+2K+2iK'}\dexpk k{\ellip{\mu}}cv\,d\ellip{\mu}\,;\label{eq: def_chi_sigma}\\
\corzz{a}{c}\ & :=\ { \frac{i}{2\pi}}\int_{\ellip{\phi}_{ca}-2iK'}^{\ellip{\phi}_{ca}+2iK'}\frac{\dexpk k{\ellip{\mu}}ca}{\cd{\frac{1}{2}(\ellip{\mu}-\ellipa{a})}k\sn{\frac{1}{2}(\ellip{\mu}-\ellipa{a})}k}\,d\ellip{\mu}\,;\label{eq: def_chi_chi}
\end{align}
where the integrals in~(\ref{eq: def_chi_mu}--\ref{eq: def_chi_chi}) are computed along \emph{straight} segments: vertical in~\eqref{eq: def_chi_mu},~\eqref{eq: def_chi_chi} and `diagonal' in~\eqref{eq: def_chi_sigma}. Note that the integrands are meromorphic functions of~$\ellip\mu$, so one only has to specify the homotopy classes of the paths of integration with respect to the poles of discrete exponentials~$\dexpk{k}\cdot{c}{x_0}$, which belong to the set~$\Im\ellip\mu\in 4K'\mathbb{Z}$ due to~(\ref{eq: defexpk}--\ref{eq: exp_vu}).

It is not hard to see that symmetries of the elliptic functions involved imply that all the quantities~(\ref{eq:defX1}--\ref{eq: def_chi_chi}) are real. However, note that one can avoid checking this fact by taking the real part in the above definitions; this operation is obviously compatible with the propagation equation~\eqref{eq:3-terms} and with asymptotics~(\ref{eq: asymp_X1}--\ref{eq: asymp_chi_chi}) discussed below.

We begin by listing several important comments on the definitions~(\ref{eq: def_chi_mu}--\ref{eq: def_chi_mu}).

\begin{remark}
Because of (\ref{eq: exp_per_1}--\ref{eq: exp_per_3}) and \cite[Eq.~22.4(iii)]{NIST:DLMF}, each of the three integrals in~(\ref{eq: def_chi_mu}--\ref{eq: def_chi_chi})
is half the integral over the twice longer (vertical or `diagonal') segment, which can
be thought of as a period of the twice bigger torus $\ellip\mu\in \C/(8K\mathbb{Z}+8iK'\mathbb{Z})$ on which the integrand is naturally defined.
\end{remark}

\begin{remark} \label{rem:GuGv-def-onUuUv}
The fact that the spinors~$\rG_{[u]}(c)$, $\rG_{[v]}(c)$ are defined on~$\Upsilon^\times_{[u]}$, $\Upsilon^\times_{[v]}$, respectively (and not simply on~$\Upsilon^\times$) follows from the anti-periodicity~(\ref{eq: exp_per_2}-\ref{eq: exp_per_3}) of discrete exponentials in the horizontal direction and the fact that~$\ellip{\phi}_{cu}$ increases by $\frac{2K}\pi\cdot 2\pi = 4K$ when~$c$ makes a turn around~$u$; see also Proposition~\ref{prop: kernels_values} below.
\end{remark}

\begin{figure}
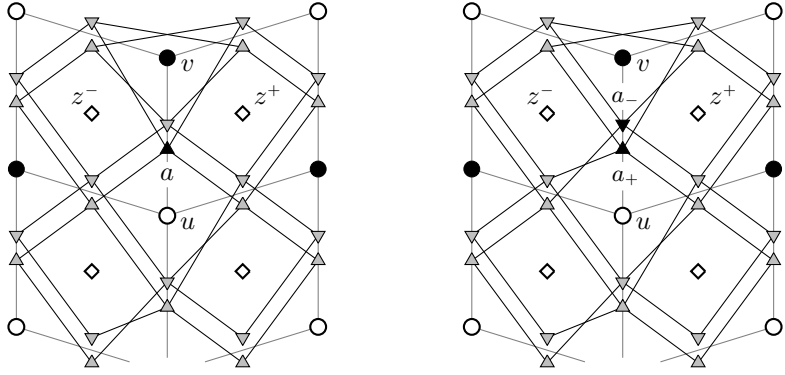

\begin{center}\begin{tikzpicture}[scale=0.21]
\input{BranchingO.txt}
\end{tikzpicture}
\hskip 48pt
\begin{tikzpicture}[scale=0.21]
\input{BranchingA.txt}
\end{tikzpicture}\end{center}
\caption{{ Given~$a\in\Upsilon^\times$ (shown as a black triangular node in the left picture), the double cover~$\Upsilon^\times_{(a)}=\Upsilon^\times_{[v(a),u(a)]}$ (shown on the right) can be identified with~$\Upsilon^\times$ except at the lifts of~$a$. We choose~$a_+\in\Upsilon^\times_{(a)}$ so that the branching structure of the two double covers around the quad~$z^+$ is the same, and similarly for~$a_-$ and~$z^-$.}}
\label{fig:U=Ua}
\end{figure}

\begin{remark} \label{rem: Gaa-def}
 Definition (\ref{eq: def_chi_chi}) of the kernel~$\rG_{(a)}(c)$ does not make sense for $c=a$;
 this corresponds to the fact that the double covers~$\Upsilon^\times$ and~$\Upsilon^\times_{(a)}$ have different branching structures near~$a$. Let~$a_+\in\Upsilon^{\times}_{(a)}$ (resp.,~$a_-$) be such that, if one identifies it with~$a\in\Upsilon^\times$, then the two double covers have the same structure around the quad~$z^\pm$ lying to the right (resp., to the left) of $(u(a)v(a))$ as in~\eqref{eq:Uu=Uv-convention}; { see Fig.~\ref{fig:U=Ua}.}  
  To define~$\rG_{(a)}(c)$ for~$c=a^\pm$ we choose~$\phi_{a_\pm a}$ so that
 \[
 \textstyle \phi_{a_+a}-\phi_{v(a)u(a)}\ \in\ (-\pi,0),
 \qquad \phi_{a_-a}-\phi_{v(a)u(a)}\ \in\ (0,\pi)
 \]
 and use the same definition~\eqref{eq: def_chi_chi} with~$\dexpk k{\ellip\mu}{a_\pm}a:=1$; it is easy to see that the result actually does not depend on the above choices of~$\ellip\phi_{a_+a}\in(\ellip{\alpha}_a-2K,\ellip{\alpha}_a)$ and $\ellip\phi_{a_-a}\in(\ellip\alpha_a,\ellip\alpha_a+2K)$. Note that $\corzz{a}{a_{+}}=-\corzz{a}{a_{-}}$ due to~\cite[Eq.~22.4(iii)]{NIST:DLMF}.
\end{remark}
\begin{proposition} \label{prop: kernels_values}
The spinors~(\ref{eq: def_chi_mu}--\ref{eq: def_chi_chi}) satisfy the propagation equation~\eqref{eq:3-terms} on~$\Upsilon^\times_{[u]}$, $\Upsilon^\times_{[v]}$ and~$\Upsilon^\times_{(a)}$, respectively; in the latter case one should use the value at~$a_+$ (resp., at~$a_-$) to recover the equation to the right (resp., to the left) of~$(u(a)v(a))$. Moreover, for all~$c\in\Upsilon^\times$ we have
\[
\coruz{u(c)}c=\corvz{v(c)}c=1\quad\text{and}\quad \corzz c{c_\pm}=\pm 1\,,
\]
independently of~$k$, with conventions $\frac{1}{2}\alpha_c=\frac{1}{2}\phi_{cu}=\frac{1}{2}(\phi_{cv}+\pi)$ used to specify one of the two layers of corners~$u\sim c\in\Upsilon^\times_{[u]}$ (resp., $v\sim c\in\Upsilon^\times_{[v]}$); cf.~Remark~\ref{rem:GuGv-def-onUuUv}.
\end{proposition}
\begin{proof} Due to Lemma~\ref{lem: prop_exp_massive}, the discrete exponentials~$c\mapsto \dexpk k{\ellip\mu}{c}x_0$ satisfy the equation~\eqref{eq:3-terms}. Therefore so do each of the functions~(\ref{eq: def_chi_mu}--\ref{eq: def_chi_chi}) provided that one can shift the contours of integration used to define its values around a given quad to the same position, not crossing the poles of the integrands. This follows from the same considerations as in~\cite[Theorem~4.2]{kenyon2002laplacian} and~\cite[Theorem~12]{boutillier2017z}.

For~$u=u(c)$, applying \cite[Eq.~22.14.7]{NIST:DLMF} we obtain the identity
\begin{align*}
\coruz{u(c)}c\ &=\ -\frac{i}{2\pi}\int_{\ellip\phi_{cu}-2iK'}^{\ellip\phi_{cu}+2iK'}\dc{\tfrac{1}{2}(\ellip{\mu}-\ellip\alpha_c)}k d\ellip{\mu}\\
&=\ -\frac{i}{2\pi}\int_{-2iK'}^{2iK'}\dc{\tfrac{1}{2}\ellip{\mu}}k d\ellip{\mu}
=-\frac{i}{\pi}\log\left(\nc{\tfrac12{\ellip{\mu}}}k+\sc{\tfrac12{\ellip{\mu}}}k\right)\Big|_{-2iK'}^{2iK'}\,. 
\end{align*}
Note that~$\nc{\pm iK'}k=0$ and~$\sc{\pm iK'}k=\pm i$; see \cite[Eq.~22.5.1, 22.4.3]{NIST:DLMF}. Thus, the value~$\coruz{u(c)}c$ belongs to the set~$1+2\mathbb{Z}$. Since it could only depend continuously on~$k$, in fact it does \emph{not} depend on~$k$ at all. The concrete value~$1$ of the answer can be justified, e.\,g., by considering the limit
\[
\int_{-2iK'}^{2iK'}\dc{\tfrac12{\ellip{\mu}}}k d\ellip{\mu}\ \to\ \int_{-i\infty}^{i\infty}\frac{dt}{\cos(t/2)}\ =\ 2\pi i\ \ \text{as}\ \ k\to 0\,.
\]
 The computation for $\corvz{v(c)}c$ is similar. For $\corzz c{c_{+}},$
we get by \cite[Eq.~22.13(i)]{NIST:DLMF}
\[
\frac{i}{2\pi}\int_{-0-2iK'}^{-0+2iK'}\dc{\tfrac12{\ellip{\mu}}}k \ns{\tfrac12{\ellip{\mu}}}k d\mu \ =\ \frac{i}{\pi}\log\sc{\tfrac12{\ellip{\mu}}}k\Big|_{-0-2iK'}^{-0+2iK'}\ \in\ 1+2\mathbb Z
\]
and, once again, the choice of the value $1$ can be justified by considering~$k\to 0$.
\end{proof}

\subsection{Asymptotics at criticality}\label{sub:kernels-crit}

Before analyzing the asymptotics of the full-plane kernels~(\ref{eq: asymp_chi_mu}--\ref{eq: asymp_chi_chi}) in the massive regime, let us first briefly discuss the degenerate case~$q=0$ and~$\delta=1$. Changing the variable~$\lambda:=-2e^{-i\mu}$ according to~\eqref{eq: exp_critical_massive} (see also~\cite[Remark~13]{boutillier2017z}), one sees that in this case~\eqref{eq: def_chi_mu} reads as
\[
\rG_{[u]}(c)\ =\ -\frac{i}{2\pi}\int_{(\overline{u}-\overline{c})\R_+}\!\frac{\dexp{\lambda}cu}{\lambda}\,d\lambda\
=\ -\frac{i}{2\pi}\int_{(\overline{u}-\overline{c})\R_+}\!\frac{\dexp{2\lambda}cu}{\lambda}\,d\lambda\,.
\]
The behaviour of these integrals for~$|c-u|\gg \delta=1$ is governed by the behavior of the integrands near~$0$ and~$\infty$ (e.\,g., see~\cite{kenyon2002laplacian} or~\cite[Appendix]{ChelkakSmirnov1}). Note that we have
\begin{align*}
\lambda^{-1}\dexp{2\lambda}{c}{u}\ &=\ \lambda^{-1}\dexp{2\lambda}{c}{u(c)}\cdot \dexp{2\lambda}{u(c)}{u}\\
&=\ 2\varsigma\overline\eta_c\cdot \lambda^{-\frac12}\cdot \exp\big[2(c-u)\lambda+O(\lambda^2)+O(|c-u|\lambda^3)\big]\ \ \text{as}\ \ \lambda\to 0.
\end{align*}
since $\dexp{2\lambda} vu = \exp\big[2\lambda(v-u)+O(\lambda^3)\big]$ for~$u\sim v$ and thus only the first factor contributes to the~$O(\lambda^2)$ term. A similar computation shows that
\[
\lambda^{-1}\dexp{2\lambda}{c}{u}\ =\ -2\overline{\varsigma}\eta_c\cdot \lambda^{-\frac32}\cdot \exp\big[2(\overline{c}-\overline{u})\lambda^{-1}+O(\lambda^{-2})+O(|c-u|\lambda^{-3})\big]
\]
as~$\lambda\to\infty$. Using the Laplace method as in~\cite{kenyon2002laplacian,ChelkakSmirnov1,DubedatDimersCR} one obtains asymptotics
\begin{align*}
\rG_{[u]}(c)\ &=\ -\tfrac{i}{\pi}\cdot \big[\varsigma\overline{\eta}_c\cdot (2(u-c))^{-\frac12}-\overline{\varsigma}\eta_c\cdot(2(\overline{u}-\overline{c}))^{-\frac12}\big]\cdot \Gamma(\tfrac12)+ O(|c-u|^{-\frac52})\\
&=\ \big(\tfrac2\pi\big)^{\frac12}\Re\big[\overline\eta_c\cdot\varsigma(c-u)^{-\frac12}\big]+O(|c-u|^{-\frac52})
\end{align*}
as~$|c-u|\to\infty$. After an appropriate scaling as~$\delta\to 0$, this reads as
\[
\cG_{[u]}(z)\ =\ \biggl(\frac2\pi\biggr)^{\!\!\frac 12}\cdot\frac{\varsigma}{\sqrt{z-u}}+O(\delta^2|z-u|^{-\frac 52}),\quad z\in\diamondsuit,
\]
in terms of the corresponding s-holomorphic functions~$\cG_{[u]}$ associated with~$\rG_{[u]}$ via Definition~\ref{def:shol-def-mass}; this is  nothing but asymptotics~\eqref{eq:G-asymp-crit} for~$w=u\in\Gamma^\circ$. A similar result for~$\rG_{[v]}$,  with the multiple~$-i\varsigma$ instead of~$\varsigma$, follows by the duality;  see Remark~\ref{rem:duality-asymp}.
\begin{remark} A similar analysis at criticality applies to discrete Cauchy kernels
\[
\rG_{(a)}(c)\ 
=\ \frac{2ie^{i\alpha_a}}{\pi}\int_{(\overline{a}-\overline{c})\R_+}\frac{\dexp{2\lambda}ca}{1-{\lambda^2}e^{2i\alpha_a}}\,d\lambda\,,
\]
a degenerate case of~\eqref{eq: def_chi_chi} for~$q=0$ and~$\delta=1$; see also~\eqref{eq: exp_critical_massive}. Using the Laplace method as above one gets the asymptotics
\begin{align*}
\rG_{(a)}(c)\ &=\ \tfrac{2i}{\pi}\cdot\big[e^{i\alpha_a}\eta_a\overline{\eta}_c\cdot(2(a-c))^{-1}-e^{-i\alpha_a}\overline\eta_a\eta_c\cdot (2(\overline{a}-\overline{c}))^{-1}\big]+O(|c-a|^{-3})\\
&=\ \tfrac 2\pi\cdot\Re\big[\overline{\eta}_c\cdot \overline{\eta}_a(-i\varsigma^2)(c-a)^{-1}\big]+O(|c-a|^{-3})
\end{align*}
as~$|c-a|\to\infty$ and~$\delta=1$. Again, this results extends to~$\delta\to 0$ by scaling arguments and reads as
\begin{equation}
\label{eq:Cauchy-crit-asymp}
{ \delta^{-\frac12}}\cG_{(a)}(z)\ =\ \frac{2}{\pi}\cdot \frac{(-i\varsigma^2)\cdot \overline{\eta}_a}{z-a}+O(\delta^2|z-a|^{-3}),\quad z\in\diamondsuit,
\end{equation}
in terms of the corresponding s-holomorphic functions; see Definition~\ref{def:shol-def-mass}.
\end{remark}

\subsection{Asymptotics and estimates of discrete exponentials} \label{sub:exp-asymp}
We start with giving a proof of the estimate~\eqref{eq:G-exp-decay-sub} for the sub-critical model (i.e.,~$k>0$ and~$\delta=1$ are fixed). We need to show that the full plane kernel~$\coruz uc$, $u\in\Gamma^\circ$, decays exponentially fast as \mbox{$|c-u|\to\infty$}; the required estimate~\eqref{eq:G-exp-decay-sub} follows by the duality. Here and below we rely upon the following fact (e.\,g., see~\cite[Lemma~17]{boutillier2011Locality}):
\begin{quote} Under condition~$\BAP$, given~$u$ and~$c$ one can find a (so-called \emph{minimal}) nearest-neighbor path $u=w_{0}\sim w_1\sim\ldots\sim w_n$ on $\Lambda$ with~$w_n\sim c$ such that all angles~$\phi_{w_{j+1}w_j}=\arg(w_{j+1}-w_j)$, as well as~$\phi_{cw_n}$, belong to a segment of length~$\pi-2\theta_0$ (which also contains the direction~$\phi_{cu}$). We denote by~$\phi^\Lambda_{cu}$ the midpoint of this segment.
\end{quote}

\begin{proof}[{\bf Proof of the estimate (\ref{eq:G-exp-decay-sub})}] By construction, the function~$\dexpk k{\cdot}cu$ does not have poles in the strip~$|\Re(\ellip\mu-{\ellip\phi}^\Lambda_{cu})|\le K-\ellip\theta_0$, which means that one can replace~$\ellip\phi_{cu}$ by~$\ellip\phi^\Lambda_{cu}$ in the definition~\eqref{eq: def_chi_mu} of~$\coruz uc$. Note that the function~$\sc{\cdot}k$ is analytic in the strip $|\Re \ellip z|\leq\frac{1}{2}K$ and periodic in the vertical direction, and that
\[
k'\cdot \big|\sc{\ellip z}k\big|^2\ =\ k'\cdot\big|\sc{\ellip z}k\sc{\ellip z\mp K}k\big|\ =\ 1\ \ \text{if}\ \ \re z=\pm \tfrac12 K
\]
due to~\cite[Eq.~22.4.3]{NIST:DLMF} and since $\ellip z\mp K=-\overline{\ellip z}$ if~$\re z=\pm\frac12 K$. Therefore, the maximum principle gives a simple estimate
\begin{equation}
\xi\ :=\sup_{\ellip z:|\re\ellip z|\leq\frac{1}{2}(K-\ellip\theta_0)}\big|i\sqrt{k'}\cdot\sc{\ellip z}k\big|<1.\label{eq:scsup}
\end{equation}
Due to~\eqref{eq: exp_vu}, we see that $\big|\dexpk k{\ellip\mu}{w_n}u\big|\le \xi^n$ if $\re\ellip\mu=\ellip\phi^\Lambda_{cu}$. Since~$|\phi_{cw_n}-\phi^\Lambda_{cu}|\le\frac{\pi}{2}$, the last factor~$\dexpk k{\ellip \mu}{c}{w_n}$ is uniformly bounded, which implies that the integrand in~\eqref{eq: def_chi_mu} is uniformly exponentially small as~$|c-u|\to\infty$.
\end{proof}

We now move on to the massive setup~$\qtoem\to 0$, $m>0$.
The asymptotic expansions of elliptic parameters via the nome~$q=\exp(-\pi K'/K)$ are, as per \cite[Eq.~22.2.1--2 and 20.2.1--4]{NIST:DLMF}:
\begin{equation}
\begin{array}{ll}
k=4q^{\frac{1}{2}}(1-4q+O(q^{2})),& K=\frac{\pi}{2}(1+4q+4q^{2}+O(q^{3})),\\[4pt]
k'=1-8q+O(q^{3}),& K'=-\frac{1}{2}\log q+O(|q\log q|).
\end{array}
\label{eq: K_k_kprime}
\end{equation}

Similarly to the critical case discussed in the previous section, the main contribution to the integrals~(\ref{eq: def_chi_mu}--\ref{eq: def_chi_chi}) comes from neighborhoods of endpoints of the corresponding segments. It is thus convenient to introduce a shifted variable
\[
\ellip\mup\ :=\ \ellip\mu-2iK';
\]
recall also that the periodicity~(\ref{eq: exp_per_1}--\ref{eq: exp_per_2}) imply that one can compute these integrals over segments lying in the strip~$|\Im\ellip\mup|\le 2K'$ instead of~$\Im\ellip\mup\in[-4K',0]$.

In terms of the variable~$\ellip\mup$, the definitions~(\ref{eq: defexpk}--\ref{eq: exp_vu}) read as (see~\cite[Eq.~22.4.3]{NIST:DLMF})
\begin{align*}
\dexpk k{\ellip{\mu}}{v(c)}c & \ =\ { (k')^{\frac 14}}\cdot \sd{\tfrac{1}{2}(\ellip{\mu}-\ellip\alpha_c)}k\ =\  ik^{-1}{ (k')^{\frac 14}}\cdot\nc{\tfrac{1}{2}(\ellip{\mup}-\ellip\alpha_c)}k,\\
\dexpk k{\ellip{\mu}}{u(c)}c & \ =\ -i{ (k')^{-\frac 14}}\cd{\tfrac{1}{2}(\ellip{\mu}-\ellip\alpha_c)}k\ =\  -ik^{-1}{ (k')^{-\frac 14}}\dc{\tfrac{1}{2}(\ellip{\mup}-\ellipa c)}k,
\end{align*}
which gives
\begin{equation}
\label{eq:x-dexpk-j+1,j=}
\dexpk k{\ellip\mu}{w_{j+1}}{w_j}\ =\ -\sqrt{k'}\cdot\nd{\tfrac{1}{2}(\ellip{\mup}-\ellip{\phi}_{w_{j+1}w_j})}k
\end{equation}
for all pairs of neighboring vertices~$w_j\sim w_{j+1}$ on $\Lambda$.

\smallskip

Denote
\[
\Pi_{q,\theta_0}\ :=\ \big\{z=\tfrac{\pi}{2K}\ellip z\in\C:\,|\Re\ellip{z}|\le \tfrac12K-\ellip\theta_0\,,\ \tfrac{1}{2} K'\le|\Im \ellip z|\le K'\big\}.
\]
In what follows, we will always have~$z=\frac{1}{2}(\nu-\phi_{**})$ for some~$\phi_{**}\in\R$. We also introduce a real variable~$y\in\R$ by requiring that
\[
\Im \ellip{z}\,=\,\tfrac{1}{2}\Im\ellip{\nu}\,=\,K'y\quad\Leftrightarrow \quad \Im z\,=\,\tfrac{1}{2}\Im\nu\,=\,(-\tfrac{1}{2}\log q)\cdot y\,.
\]

It follows from~\cite[Eq.~22.2.6 and Eq.~20.2.3--4]{NIST:DLMF} that
\[
\nd{\ellip z}k\ =\ \frac{1+2q+O(q^{4})}{1-2q+O(q^{4})}\cdot\frac{1-2q\cos(2\ea z)+O(q^2)}{1+2q\cos(2\ea z)+O(q^2)}\quad \text{if}\ \ |\Im \ellip{z}|\le K'
\]
since we have~$|\cos(2nz)|\le q^{-ny}$, which allows to bound by~$O(q^2)$ the higher (i.\,e., $n\ge 2$) terms in the expansions of the Jacobi theta functions~$\theta_{3,4}$. Therefore,
\begin{equation}
\label{eq:x-one-step-qeps}
\sqrt{k'}\,\nd{\ellip z}k\ =\ \exp\big[\!-4q\cos(2\ea z)+O(q^2)\,\big]\quad \text{if}\ \ z\in \Pi_{q,\theta_0}.
\end{equation}

Note that the map~$z\mapsto 2q\cos 2z=q\cdot(e^{2iz}+e^{-2iz})$ sends the region~$\Pi_{q,\theta_0}$
into the region $\{\zeta\in\C:|\arg\zeta|\leq\frac{\pi}{2}-\theta_0\,,\ q^{\frac 12}-q^{\frac 32}\leq |\zeta|\leq 1-q^2\}.$ As in the proof of the estimate~\eqref{eq:G-exp-decay-sub} discussed above, let us now consider a minimal path $u=w_0\sim w_1\sim \ldots \sim w_n\sim c$, where $n\asymp \delta^{-1}|c-u|$, and apply the estimate~\eqref{eq:x-one-step-qeps} to each of the factors~\eqref{eq:x-dexpk-j+1,j=}. We see that there exists a constant~$\mathrm{cst}(\theta_0)>0$ such that
\begin{equation}
\label{eq:stretched-exp}
\log \big|\dexpk k{\ellip\mu} cu\big|\ \le -\mathrm{cst}(\theta_0)\cdot q^{-\frac 12}\cdot m|c-u|
\end{equation}
as~$\qtoem\to 0$, for all~$\ellip\mu=\ellip\mup+2iK'$ such that
\[
\Re\ellip\mup=\ellip\phi^\Lambda_{cu}\quad \text{and}\quad K'\le |\Im\ellip\mup|\le 2K'.
\]
(Note also that the last factor~$\dexpk k{\ellip\mup+2iK'}c{w_n}$ is uniformly bounded for such~$\ellip\mup$.)

\smallskip

We now move to the analysis of discrete exponentials in the region~$|\Im\ellip\mup|\le K'$. Using~\cite[Eq.~22.2.6 and Eq.~20.2.3--4]{NIST:DLMF} as above, we see that
\begin{equation}
\nd{\ellip z}k\ =\ \frac{1+2q+O(q^{4})}{1-2q+O(q^{4})}\cdot\frac{1-2q\cos(2\ea z)+O(q^{4-2|y|})}{1+2q\cos(2\ea z)+O(q^{4-2|y|})}\label{eq: nd_bound_all}
\end{equation}
uniformly in the strip $|\im\ellip z|\leq \frac{1}{2}K'$ (i.\,e., $|y|\le\frac12$).
Since~$|q\cos(2z)|\le q^{1-|y|}$, this implies the asymptotics
\begin{equation}
\label{eq:x-dexp-uv-near-R}
\dexpk k{\ellip \mu}{w_{j+1}}{w_j}\ =\ -\exp\big[\!-4q\cos(\nu-\phi_{w_{k+1}w_j})+O(q^{3-3|y|})\,\big]\,,
\end{equation}
for all edges~$w_j\sim w_{j+1}$ of~$\Lambda$, uniformly in the strip~$|\Im\ellip z|\le\frac12K'$.

Let~$u=w_0\sim w_1\sim \ldots \sim w_{2n}=u(c)\sim c$, where~$w_j\in\Lambda$ and~$n=O(\delta^{-1}|c-u|)$ be a nearest-neighbor path going from~$u$ to~$c$. Using~\cite[Eq.~22.2.3--9 and 20.2.1--4]{NIST:DLMF} again, one also sees that
\[
\cd{\ellip z}k\ =\ \cos\left(\ea z\right)\cdot \exp\big[\,2q-2q\cos\left(2\ea z\right)+O(q^{2-|y|})\big],\qquad |\Im\ellip z|\le \tfrac{1}{2}K',
\]
since~$|q^{(1+\frac{1}{2})^{2}}\cos(3\ea z)/\cos\ea z|=O(q^{\frac{9}{4}-|y|})$ and thus the second (and further) terms in the expansion of the Jacobi theta function~$\theta_2$ are absorbed into the error term. 

It follows from geometric considerations that
\[
\textstyle \sum_{j=0}^{2n-1}\cos(\ea{\mup}-\ea{\phi}_{w_{j+1}w_{j}})+\tfrac{1}{2}\cos(\nu-\alpha_c)\ =\ \delta^{-1}|c-u|\cdot \cos(\ea{\mup}-\ea{\phi}_{cu})
\]
for all~$\nu\in\R$, which is therefore identically true for all~$\nu\in\C$. Multiplying asymptotics~\eqref{eq:x-dexp-uv-near-R} along the path~$u=w_0\sim w_1\sim\ldots\sim w_{2n}=u(c)$ and taking into account the last term
\begin{align*}
\dexpk k{\ellip \mu}{c}{u(c)}\ &=\ ik\cdot{ (k')^{\frac 14}}\cd{\tfrac{1}{2}(\ellip\nu-\ellip\alpha_c)}k\\
&=\ ik\cdot \cos(\tfrac{1}{2}(\nu-\alpha_c))\cdot\exp[-2q\cos(\nu-\alpha_c)+O(q^{2-|y|})]
\end{align*}
we conclude that
\begin{align}
\label{eq:dexpk-cu-asymp-near-R} \dexpk k{\ellip \mu}{c}{u} & \\
=\ ik\,\cdot&\cos(\tfrac{1}{2}(\nu-\alpha_c))\cdot \exp\big[\!-2m|c-u|\cos(\nu-\phi_{cu})+O(q^{2-3|y|})\,\big] \notag\\
=\ ik\,\cdot&\exp\big[\!-2m|c-u|\cos(\nu-\phi_{uc})\,\big]\cdot\big(\cos(\tfrac{1}{2}(\nu-\alpha_c))+O(q^{2}e^{\frac72|\Im\nu|})\big)\notag
\end{align}
as $\qtoem\to 0$, uniformly over~$|c-u|=O(1)$ and provided that~$\ellip\mu=\ellip\mup+2iK'$ is such that $\Im\ellip\mup=2K'y$ or, equivalently,~$\Im\nu=(-\log q)\cdot y$ with $|y|\le\frac{1}{2}$.

\begin{remark} \label{rem:dexpk-cc-asymp}
It is easy to see that the estimate~\eqref{eq:stretched-exp} and the asymptotics~\eqref{eq:dexpk-cu-asymp-near-R} remain true if one replaces~$u\in\Gamma^{\circ}$ by~$v\in\Gamma^{\bullet}$. However, a slightly more accurate consideration is required for the contribution of the first step~$a\sim u(a)=w_0$ to the integrand in the definition~\eqref{eq: def_chi_chi} of~$\cG_{(a)}$, $a\in\Upsilon^\times$. In this case we have
\begin{align*}
\frac{\dexpk k{\ellip\mu}{u(a)}a}{\cd{\frac12(\ellip\mu-\ellip\alpha_a)}k\sn{\frac12(\ellip\mu-\ellip\alpha_a)}k}\ &=\ -i(k')^{-\frac 14}\cdot\ns{\tfrac12(\ellip\mu-\ellip\alpha_a)}k\\
&=\ -ik(k')^{-\frac 12}\cdot (k')^{\frac 14}\cd{\tfrac12(\ellip\nu-\ellip\phi_{w_0a})}k
\end{align*}
due to~\cite[Eq.~22.4.3]{NIST:DLMF} and since~$\alpha_a=\phi_{v(a)a}=\phi_{u(a)a}-\pi$. This means that $\dexpk k{\ellip\mu}{c}{a}$ admits a similar asymptotics to~\eqref{eq:dexpk-cu-asymp-near-R} with the prefactor~$k^2(k')^{-\frac12}$ instead of~$ik$, an additional multiple~$\cos(\tfrac12(\nu-\phi_{u(a)a}))=\sin(\tfrac{1}{2}(\nu-\alpha_a))$ and with the error term $O(q^2e^{4|\Im\nu|})$ instead of~$O(q^2e^{\frac 72|\Im\nu|})$.
\end{remark}

\subsection{Asymptotics of the massive full-plane kernels~(\ref{eq:defX1}--\ref{eq: def_chi_chi})} \label{sub: massive_asymptotics_proof}
We are now ready to compute the required asymptotics of the spinors~$\cX_\rr$, $\cX_\ri$, $\rG_{u}$, $\rG_{[v]}$ and~$\rG_{(a)}$. In order to formulate the result for the latter one, denote
\[
f(z)\,:=\,(-i\varsigma^2)\cdot 4|m|e^{-i\arg z}K_{1}(2|mz|),\qquad f^{\star}(z)\,:=\,4imK_{0}(2|mz|),
\]
where~$K_1$ and~$K_0$ are modified Bessel functions of the second kind (see~\cite[Section~10.25]{NIST:DLMF}). Also, let
\begin{equation}
\label{eq:x-feta-def}
f^{[\eta]}(z)\ :=\ \tfrac{1}{2}[\,\overline{\eta}f(z)+\eta f^{\star}(z)\,]\ \ \text{for}\ \ \eta\in\C,
\end{equation}
this is a special solution of the massive holomorphicity equation
$\overline{\pa}f+\varsigma^2 m\overline{f}=0$ that behaves like a Cauchy kernel $(-i\varsigma^2)\overline{\eta}\cdot z^{-1}$ at the origin; cf.~\eqref{eq:Cauchy-crit-asymp}. Let~$m>0$ be fixed { and recall that~$K=K(k)=\frac{\pi}{2}(1+2m\delta+O(\delta^2))$ as~$\delta\to 0$; see~\eqref{eq: K_k_kprime}.} We claim that the following asymptotics hold as~$\qtoem\to 0$:
\begin{align}
{ \delta^{-\frac 12}}\cX_\rr(c)\ & =\ \re\left[\,\overline{\eta}_{c}\cdot e^{-2m\Re[\overline{\varsigma}^2(c-o)]}\,\right]+O(\delta^2)\,,\label{eq: asymp_X1}\\
{ \delta^{-\frac 12}}\cX_\ri(c)\ & =\ \re\left[\,\overline{\eta}_{c}\cdot ie^{2m\Re[\overline{\varsigma}^2(c-o)]}\,\right]+O(\delta^2)\,;\label{eq: asymp_Xi}\\
{ \delta^{-\frac 12}}\coruz uc\ & =\ { K^{-\frac 12}\cdot} \re\left[\,\overline{\eta}_{c}\cdot \varsigma\frac{e^{-2m|c-u|}}{\sqrt{c-u}}\,\right]+{ O(\delta^{2})\,,}\label{eq: asymp_chi_mu}\\
{ \delta^{-\frac 12}}\corvz vc\ & =\ { K^{-\frac 12}\cdot} \re\left[\,\overline{\eta}_{c}\cdot (-i\varsigma)\frac{e^{2m|c-v|}}{\sqrt{c-v}}\,\right]+{ O(\delta^{2})\,;}\label{eq: asymp_chi_sigma}\\
{ \delta^{-1}}\corzz{a}{c}\ & =\ { \frac 2\pi}\cdot\re\left[\,\overline{\eta}_{c}f^{[\eta_{a}]}(c-a)\,\right]+{ O(\delta^2)}\,,\label{eq: asymp_chi_chi}
\end{align}
where~(\ref{eq: asymp_X1}--\ref{eq: asymp_Xi}) are uniform on compact subsets and~(\ref{eq: asymp_chi_mu}--\ref{eq: asymp_chi_chi}) are uniform provided that $|c-w|=O(1)$ and~$|c-w|^{-1}=O(1)$, where~$w=u,v,a$, respectively.
{ More generally, the same asymptotics hold if~$\delta\to 0$ and~$q\to 0$ simultaneously so that~$m:=2q\delta^{-1}$ stays uniformly bounded away from~$0$ and~$\infty$.}


\begin{remark} Similar asymptotics for~$m<0$ follow by the duality, which amounts to exchanging the lattices~$\Gamma^\circ\leftrightarrow\Gamma^{\bullet}$, changing the sign of~$q$ and~$m$, and replacing~$\varsigma$ by~$\pm i\varsigma$; see Remark~\ref{rem:duality-asymp}. In particular, note that~(\ref{eq: asymp_X1}--\ref{eq: asymp_Xi}) provide, via Definition~\ref{eq:shol-def-mass}, a proof of Theorem~\ref{thm:F1Fi-asymp}. Similarly, (\ref{eq: asymp_chi_mu}--\ref{eq: asymp_chi_sigma}) and Definition~\ref{def:shol-def-mass} yield Theorem~\ref{thm:G-asymp-mass} since~$K=\frac{\pi}{2}+O(q)$ as~$q\to 0$.
\end{remark}

\begin{proof} [{\bf Proof of the asymptotics~(\ref{eq: asymp_X1}--\ref{eq: asymp_Xi})}] This immediately follows from definitions~(\ref{eq:defX1}--\ref{eq:defXi}) of the spinors~$\cX_\rr$, $\cX_\ri$ and asymptotics~\eqref{eq:dexpk-cu-asymp-near-R} of discrete exponentials. Indeed, for~$\nu \in\R$, note that
\[
\cos(\tfrac{1}{2}(\nu-\alpha_c))\,=\,\Re[\,\overline{\eta}_c\cdot \varsigma e^{-i\frac{\nu}{2}}\,]\quad \text{and}\quad
|c-o|\cos(\nu-\phi_{co})\,=\,\Re[\,e^{-i\nu}(c-o)\,]\,.
\]
Therefore,~\eqref{eq:dexpk-cu-asymp-near-R} directly yields~(\ref{eq: asymp_X1}--\ref{eq: asymp_Xi}).
\end{proof}

\begin{proof}[{\bf Proof of the asymptotics~(\ref{eq: asymp_chi_mu})}] Recall that, due to periodicity reasons, the integral in~(\ref{eq: def_chi_mu}) is equal to the integral over a shifted segment $\ellip{\mu}\in[\ellip{\phi}_{cu};\ellip{\phi}_{cu}+4iK'],$ i.\,e., $\nu\in[\ellip{\phi}_{cu}-2iK';\ellip{\phi}_{cu}+2iK']$. This contour can be deformed, without crossing the poles of the integrand, to the broken line passing through the points
\[
\ellip{\phi}^\Lambda_{cu}-2iK';\quad \ellip{\phi}_{cu}^\Lambda-iK';\quad \ellip{\phi}_{cu}-iK';\quad \ellip{\phi}_{cu}+iK';\quad \ellip{\phi}_{cu}^\Lambda+iK';\quad \ellip{\phi}_{cu}^\Lambda+2iK'.
\]
The contribution of the first and the last (vertical) segments to the integral~\eqref{eq: def_chi_mu} is stretch-exponentially small as~$q\to 0$ due to the estimate~\eqref{eq:stretched-exp}. The contribution of horizontal segments at height~$\pm iK'$ is also stretch-exponentially small due to the asymptotics~\eqref{eq:dexpk-cu-asymp-near-R} and since~$|\phi^\Lambda_{cu}-\phi_{cu}|\le \frac{\pi}{2}-\theta_0$.

Let $\rho_{cu}:=2m|c-u|$. On the middle segment
\[
\Re\ellip\nu\,=\,\ellip\phi_{cu},\qquad |\Im\ellip\nu|\,\le\,K'
\]
the asymptotics~\eqref{eq:dexpk-cu-asymp-near-R} reads as
\[
\dexpk k{\ellip \mu}{c}{u}\ =\ ik\cdot\exp\big[\!-\rho_{cu}\cosh(\Im\nu)\,]\cdot \big(\cos(\tfrac{1}{2}(\nu-\alpha_c))+O(q^{2}e^{\frac72|\Im\nu|})\big);
\]
{ Provided that~$\rho_{cu}\ge \mathrm{cst}>0$,} this yields the asymptotics
\[
\coruz uc\ =\ -\frac{(k')^{\frac 14}}{2\pi}\cdot ik\cdot \frac{2iK}{\pi}\cdot\biggl[\int_{-\infty}^{+\infty}e^{-\rho_{cu}\cosh t}\cos(\tfrac{1}{2}(it+\phi_{cu}-\alpha_c))\,dt+O(q^2)\biggr],
\]
where~$t=\Im\nu$ and we used the fact the integrand is stretch-exponentially small in~$q\to 0$ outside the segment~$|t|\le K'\sim -\frac12\log q$.

The leading term can be evaluated explicitly by writing
\[
\cos(\tfrac{1}{2}(it+\phi_{cu}-\alpha_c))\ =\ \cosh(\tfrac12t)\cos(\tfrac12(\phi_{cu}-\alpha_c))- i\sinh(\tfrac12t)\sin(\tfrac12(\phi_{cu}-\alpha_c))
\]
and using \cite[Eq.~10.32.9 and Eq.~10.39.2]{NIST:DLMF}, which gives
\begin{equation}\label{eq:x-leading-term-Gu}
\int_{-\infty}^{+\infty}e^{-\rho_{cu}\cosh t}\cosh(\tfrac{1}{2}t)\,dt\ =\ (2\pi/\rho_{cu})^{\frac 12}\cdot e^{-\rho_{cu}}.
\end{equation}
To complete the proof, it remains to note that (see~\eqref{eq: K_k_kprime})
\[
\frac{k\,(k')^{\frac 14}K}{\pi^2}\cdot \biggl(\frac \pi{m}\biggr)^{\!\frac 12}\,=\ \delta^{\frac{1}{2}}K^{-\frac{1}{2}}\cdot (1+O(q^2))\ \ \text{as}\ \ \textstyle \qtoem\to 0
\]
and that ${\cos(\tfrac 12(\phi_{cu}-\alpha_c))}\cdot |c-u|^{-\frac12}=\Re\big[e^{\frac{i}{2}\alpha_c}(c-u)^{-\frac 12}\big]=\Re[\,\overline{\eta}_c\cdot \varsigma\,(c-u)^{-\frac12}]$.
\end{proof}

\begin{proof}[{\bf Proof of the asymptotics~(\ref{eq: asymp_chi_sigma})}]
The treatment of $\corvz vc$ only differs in the computation of
the leading term, which we now perform; recall that the asymptotics~\eqref{eq:dexpk-cu-asymp-near-R} and the estimate~\eqref{eq:stretched-exp} remain true if one replaces~$u\in\Gamma^{\circ}$ by~$v\in\Gamma^{\bullet}$.
Instead of the integral
\[
\int_{\phi_{cu}-i\infty}^{\phi_{cu}+i\infty} e^{-\rho_{cu}\cos(\nu-\phi_{cu})}\cos(\tfrac{1}{2}(\nu-\alpha_c))d\nu
\]
that gave rise to the leading term in the asymptotics~\eqref{eq: asymp_chi_mu} we now need to consider a similar integral computed, e.\,g., along the broken line
\[
\nu\,\in\,[\phi_{cv}-i\infty\,;\phi_{cv}]\,\cup\,[\phi_{cv}\,;\phi_{cv}+2\pi]\,\cup\,[\phi_{cv}+2\pi\,;\phi_{cv}+i\infty]\,.
\]

We write this integral as the sum of (a) the integral computed over the vertical line~$[\phi_{cv}-i\infty\,;\phi_{cv}+i\infty]$ that was already evaluated in~\eqref{eq:x-leading-term-Gu}, and (b) a similar integral computed over the contour
\[
\nu=\phi_{cv}+it,\quad \text{where}\quad t\,\in\,\Gamma\ :=\ [+\infty\,;0]\,\cup\,[0\,;-2\pi i]\,\cup\,[-2\pi i\,;-2\pi i+\infty]\,;
\]
note that we have
\[
\cos(\tfrac12(\nu-\alpha_c))\ =\ \tfrac 12 (e^{-\frac t2}e^{\frac i2(\phi_{cv}-\alpha_c)}+e^{\frac t2}e^{-\frac i2(\phi_{cv}-\alpha_c)})\,.
\]
It follows from~\cite[Eq.~10.32.12 and~Eq.~10.39.1]{NIST:DLMF} that
\begin{align*}
\int_{\Gamma}e^{-\rho_{cv}\cosh t}e^{\,\mp \frac{t}{2}}dt\ &=\ \pm i\int_{+\infty+i\pi}^{+\infty-i\pi}e^{\,\rho_{cv}\cosh t}e^{\,\mp\frac{t}{2}}dt\\
&=\ \pm 2\pi\cdot  I_{\pm\frac 12}(\rho_{cv})\ =\ \pm (2\pi/\rho_{cv})^{\frac 12}\cdot (e^{\rho_{cv}}\mp e^{-\rho_{cv}})\,,
\end{align*}
which gives
\begin{align*}
\int_\Gamma e^{-\rho_{cv}\cosh t}&\cos(\tfrac12(\phi_{cv}+it-\alpha_c))dt\\
&=\ (2\pi/\rho_{cv})^{\frac 12}\cdot (ie^{\rho_{cv}}\sin(\tfrac 12(\phi_{cv}-\alpha_c))-e^{-\rho_{cv}}\cos(\tfrac12(\phi_{cv}-\alpha_c)))\,.
\end{align*}
The second term cancels out with the contribution (a) of the line~$\Re\nu=\phi_{cv}$. Taking into account the additional factor~$ i$ in~\eqref{eq: def_chi_sigma} as compared to~\eqref{eq: def_chi_mu}, we obtain the following expression for the leading term in the asymptotics~\eqref{eq: asymp_chi_sigma}:
\[
K^{-\frac 12}\cdot \sin(\tfrac 12({ \alpha_c-\phi_{cv}}))|c-v|^{-\frac 12}\cdot e^{2m|c-v|}+O(\delta^2)\,.
\]
It remains to note that
\[
\sin(\tfrac 12(\alpha_c-\phi_{cv}))|c-v|^{-\frac 12}\,=\,\Im\big[\varsigma\,\overline{\eta}_c\cdot(v-c)^{-\frac 12}\big]
\,=\,\Re\big[\overline{\eta}_c\cdot (-i\varsigma)(v-c)^{-\frac 12}\big].\qedhere
\]
\end{proof}

\begin{proof}[{\bf Proof of the asymptotics~(\ref{eq: asymp_chi_chi})}] Due to Remark~\ref{rem:dexpk-cc-asymp}, the integrand in the definition~\eqref{eq: def_chi_chi} behaves similarly to the integrand in~\eqref{eq: def_chi_mu}. Following the same lines as in the proof of~\eqref{eq: asymp_chi_mu} we obtain the asymptotics
\[
\rG_{(a)}(c)\ =\ \frac{i}{2\pi}\cdot k^2(k')^{-\frac12}\cdot \frac{2iK}{\pi}\cdot\biggl[\int_{-\infty}^{+\infty}e^{-\rho_{ca}\cosh t}\,V(t)dt+O(q^2)\biggr],
\]
where~$\rho_{ca}=2m|c-a|$ and
\begin{align*}
V(t)\ &=\ \sin(\tfrac12(\phi_{ca}+it-\alpha_a))\cdot \cos(\tfrac12(\phi_{ca}+it-\alpha_c))\\
&=\ \tfrac{1}{2}\cosh t\cdot \sin(\phi_{ca}-\tfrac{1}{2}(\alpha_a+\alpha_c))+\tfrac{i}{2}\sinh t\cdot \cos(\phi_{ca}-\tfrac{1}{2}(\alpha_a+\alpha_c))\\
&+\ \tfrac{1}{2}\sin(\tfrac{1}{2}(\alpha_c-\alpha_a)).
\end{align*}
Using~\cite[Eq.~10.32.9]{NIST:DLMF} we get
\begin{align*}
\int_{-\infty}^{\infty}e^{-\rho\cdot\cosh t}\,V(t)dt\ &=\ K_{1}(\rho_{ca})\sin(\phi_{ca}-\tfrac{1}{2}(\alpha_a+\alpha_c))+K_0(\rho_{ca})\sin(\tfrac{1}{2}(\alpha_c-\alpha_a))\\[-2pt]
&=\ -K_{1}(\rho_{ca})\cdot \Im[\,e^{-i\phi_{ca}}\cdot\varsigma^2\overline{\eta}_c\overline{\eta}_a\,]+K_0(\rho_{ca})\cdot \Im[\,\overline{\eta}_c\eta_a\,]\\
&=\ -\Re\big[\,\overline{\eta}_c\cdot(\overline{\eta}_a\cdot(-i\varsigma^2)e^{-i\phi_{ca}}K_1(\rho_{ca})+\eta_a\cdot iK_0(\rho_{ca})\big)\big]\\
&=\ -\Re\big[\,\overline{\eta}_c\cdot(2m)^{-1}f^{[\eta_a]}(2m|c-a|)\,],
\end{align*}
see definition~\eqref{eq:x-feta-def} of the massive Cauchy kernel~$f^{[\eta]}$. Finally, it easily follows from~\eqref{eq: K_k_kprime} that
\[
\frac{k^2(k')^{-\frac12}K}{\pi^2}\ =\ \frac{8q}{\pi}\cdot(1+O(q^2))\ =\ \frac{4m\delta}{\pi}\cdot(1+O(\delta^2)),
\]
which completes the computation.
\end{proof}


\end{document}